\documentclass[12pt]{amsart}

\usepackage[matrix,arrow,curve,frame]{xy}
\usepackage{amsmath,amsthm,amssymb,enumerate}
\usepackage{latexsym}
\usepackage{amscd}
\usepackage[colorlinks=false]{hyperref}
\usepackage[mathscr]{eucal}
\usepackage{color}

\newtheorem{thm}[subsection]{Theorem}

\newtheorem{prob}[subsection]{Problem}

\newtheorem{cor}[subsection]{Corollary}
\newtheorem{lemma}[subsection]{Lemma}
\newtheorem{conj}[subsection]{Conjecture}

\theoremstyle{definition}

\newtheorem{question}[subsection]{Question}

\numberwithin{equation}{section}

\def\Z{\mathbb Z}

\def\phi{{\varphi}}

\def\ra{\rightarrow}
\def\bra{\langle}
\def\ket{\rangle}
\def\half{{1\over2}}
\def\cA{{\mathcal A}}
\def\cB{{\mathcal B}}
\def\cC{{\mathcal C}}
\def\cD{{\mathcal D}}
\def\cE{{\mathcal E}}
\def\cF{{\mathcal F}}

\def\cH{{\mathcal H}}
\def\cI{{\mathcal I}}

\def\cM{{\mathcal M}}
\def\cN{{\mathcal N}}
\def\cO{{\mathcal O}}
\def\cP{{\mathcal P}}

\def\cR{{\mathcal R}}
\def\cS{{\mathcal S}}

\def\cV{{\mathcal V}}
\def\cW{{\mathcal W}}

\def\ga{{\mathfrak a}}
\def\gb{{\mathfrak b}}

\def\gg{{\mathfrak g}}
\def\gh{{\mathfrak h}}

\def\gl{{\mathfrak l}}

\def\go{{\mathfrak o}}
\def\gp{{\mathfrak p}}

\def\gs{{\mathfrak s}}

\newfont{\german}{eufm10}

\begin{document}
\pagestyle{plain}

\title
{Vertex algebras and commutative algebras}

\author{Bong H. Lian}
\address{Brandeis University}
\email{lian@brandeis.edu}
\thanks{B. L. is supported by the Simons Collaboration on Homological Mirror Symmetry and Applications Grant \#385579 (2016-2023) }

\author{Andrew R. Linshaw} 
\address{University of Denver}
\email{andrew.linshaw@du.edu}
\thanks{A. L. is supported by Simons Foundation Grant \#635650 and NSF DMS-2001484}

\begin{abstract} 
This paper begins with a brief survey of the period prior to and soon after the creation of the theory of vertex operator algebras (VOAs). This survey is intended to highlight some of the important developments leading to the creation of VOA theory.
The paper then proceeds to describe progress made in the field of VOAs in the last 15 years which is based on fruitful analogies and connections between VOAs and commutative algebras. First, there are several functors from VOAs to commutative algebras that allow methods from commutative algebra to be used to solve VOA problems. To illustrate this, we present a method for describing orbifolds and cosets using methods of classical invariant theory. This was essential in the recent solution of a conjecture of Gaiotto and Rap\v{c}\'ak that is of current interest in physics. We also recast some old conjectures in the subject in terms of commutative algebra and give some generalizations of these conjectures.
We also give an overview of the theory of topological VOAs (TVOAs), with applications to BRST cohomology theory and conformal string theory, based on work in the 90's. We construct a functor from TVOAs to Batalin-Vilkovisky algebras -- supercommutative algebras equipped with a certain odd Poisson structure realized by a second order differential operator -- and present a number of interesting applications. This paper is based in part on the lecture given by the first author at the Harvard CMSA Math-Science Literature Lecture Series on May 22, 2020.

\end{abstract}

\maketitle
\tableofcontents
\section{Introduction} \label{sec:intro}

We begin with a brief survey of the period prior to and soon after the creation of the theory of VOAs which is intended to highlight developments that reflect the authors' views (and biases) about the subject. We make no claim of completeness. As a short survey of early history, it will inevitably miss many of the more recent important results. Emphases are placed on the mutually beneficial cross-influences between physics and VOAs in their concurrent early developments, and the early history survey is aimed for a general audience.

\subsection{Early History 1970s -- 90s: two parallel universes} 

In 1968, Veneziano proposed a model using the Euler beta function to explain the `$st$-channel crossing' symmetry in 4-meson scattering, and the Regge trajectory (an angular momentum vs binding energy plot for the Coulomb potential), and was later generalized to $n$-meson scattering by others.

Around 1970, Nambu, Nielsen, and Susskind provided the first interpretation of the Veneziano amplitude in terms of a Fock space representation of infinitely many harmonic oscillators. The $n$-particle amplitude then became an $n$-point correlation function of certain `vertex operators'
$$:e^{i k\cdot\phi(z)}:$$ 
on a Fock space representation of free bosonic fields $\phi_i(z)$ of a `string' \cite{GGRT}. This can be viewed as part of a {\it fundamental theory of strings}, producing fundamental particles of arbitrarily high spins as string resonances, including gravity \cite{SS74}.

In 1981, Polyakov proposed a path integral formulation of the bosonic string theory -- a theory with a manifest two-dimensional conformal symmetry, pointing to a fundamental connection between string theory and two-dimensional conformal field theory (2d CFT). In 1984, Belavin-Polyakov-Zamolodchikov introduced the so-called conformal bootstrap program to systematically study 2d CFTs as {\it classical string vacua}, and models for {\it universal critical phenomena}. 
They also introduced a powerful mathematical formalism -- the operator product expansion for 2d CFT observables.
For example, the 2d conformal symmetry of such a theory can be stated in terms of the the left (and right) moving 2d stress energy tensor field of central charge $c$ (conformal anomaly) whose OPE is given by:
$$T(z)T(w)\sim{c/2\over (z-w)^4}+{ 2T(w)\over (z-w)^2}+{\partial_w T(w)\over z-w}+\cdots$$
The OPE is equivalent to that the state space is a representation of the two copies of the {\it Virasoro algebra} defined by the Fourier modes of $T(z)=\sum L_n z^{-n-2}$ (and those of $\bar T(\bar z)$):
$$[L_n,L_m]=(n-m)L_{n+m}+{c\over 12}(n^3-n)\delta_{n+m}.$$

In 1985, Friedan-Martinec-Shenker showed that conformal strings and other 2d CFTs are fixed points of a renormalization group flow. A fundamental subalgebra of observables of a 2d CFT is its `chiral algebra' -- the algebra of left (or holomorphic) operator on a formal punctured disk.
For example, the so-called bosonic $bc$-ghost systems of weight $\lambda$ (and its fermionic version the $\beta\gamma$-ghost system of weight $\lambda$) is defined by the OPE
$$b(z)c(w)\sim {1\over z-w}$$
where $b,c$ are primary fields of respective conformal weights $\lambda$, and $1-\lambda$.
In connection to strings, both the $bc$ and the $\beta\gamma$ systems also appear fundamental, because they both arise naturally as {\it determinants} in Polyakov's path integral  in the {\it conformal gauge}, following the so-called Faddeev-Popov quantization procedure.

Around the same time, thanks to a new generation of young physicists led by Candelas, Dixon, Friedan, Martinec, Moore, Seiberg, Shenker, Vafa, E. Verlinde, H. Verlinde, Warner, Witten, and many others, a vast collaborative effort was emerging to
organize an important class of 2d CFTs using representations of {\it rational chiral algebras} (consisting of the left moving operators), and their so-called fusion rules. A rational chiral algebra has only finite number of irreducible representations. Notable known examples include the BPZ classification of the Virasoro central charges and highest weights of the {\it rational minimal model CFTs} using Kac's determinant formula. (Cf. Y. Zhu's thesis and W. Wang's thesis);
the {\it Wess-Zumino-Witten theory} for loop groups of compact groups gave (unitary) rational CFTs;
Goddard-Kent-Olive's {\it coset constructions} from compact Lie groups gave (unitary) rational CFTs.
Their `orbifolds' also yield many more rational CFTs. 
Around 1986, Moore-Seiberg introduced the fundamental `duality axioms' of CFTs. For the chiral algebra, they imply that matrix coefficients (or correlation functions) of left moving operators admit meromorphic continuations on $\mathbb P^1$. Most importantly, the operators are {\it formally} commutative after analytic continuations. For example, 
$$\bra T(z) T(w)\ket={c/2\over(z-w)^4}.$$

On the string theory front, a number of historical milestones were made as well.
In 1985, Green-Schwarz's discovery of an amazing anomaly cancellation for the open string with gauge group $SO(32)$, and discovery of spacetime supersymmetry in the Ramond-Neveu-Schwarz superstring theory marked the beginning of a new revival of superstring theory. At the same time, Candelas-Horowitz-Strominger-Witten (the so-called `string quartet') proposed superstring compactifications on Calabi-Yau threefolds as a new class of string vacua, which can be realized as supersymmetric `nonlinear sigma models' (NLSMs).
In 1986, Gepner made the remarkable discovery that some of these NLSMs can in fact be realized as certain explicit unitary rational $\cN=2$ SCFT's at  special points (`Gepner points') in moduli spaces of superstring compactifications on Calabi-Yau threefolds. The most famous example is the quintic threefold in $\mathbb P^4$.

\subsection{The Moonshine Universe}

Parallel to the physical 2d CFT universe is a mathematical universe in which history was in the making. In 1978, McKay found evidence of the existence of an infinite dimensional $\Z$-graded representation $V^\natural$  of the hypothetical {\it Monster group $\mathbb M$} (independently predicted by Fischer and Griess in 1973):-- the coefficients of the $q$-series of the $j$-function can be partitioned by dimensions $r_i$ of irreducible $\mathbb M$-modules:
$${\displaystyle {\begin{aligned}
j(q)=q^{-1}+196884{q}+2149&3760{q}^{2}+864299970{q}^{3}
+20245856256{q}^{4}
+\cdots \\
1&=r_{1}\\
196884&=r_{1}+r_{2}\\21493760&=r_{1}+r_{2}+r_{3}\\864299970&=2r_{1}+2r_{2}+r_{3}+r_{4}
\\20245856256&=3r_{1}+3r_{2}+r_{3}+2r_{4}+r_{5}\\ 
&=2r_{1}+3r_{2}+2r_{3}+r_{4}+r_{6}
\end{aligned}}}$$
Thompson interpreted $j(q)$ as the graded- or  $q$-trace
$$j(q)=qtr_{V^\natural} 1:=\sum_n tr_{V^\natural[n]}1 q^n.$$
This would later become the so-called genus 1 partition function $tr ~q^{L_0-{c\over24}}$ of the {\it holomorphic vertex algebra} of the underlying chiral algebra of $V^\natural$!
This led to the expectation that the $q$-trace of a general element $g\in\mathbb M$ should be interesting as well. 
Soon after, Conway and Norton computed leading terms of the hypothetical $q$-traces, and saw that they agree with $q$-series of certain special genus $0$ {\it modular functions} or hauptmodul. They gave a complete list of these functions, now know as {\it the McKay-Thompson series} $T_g$. They also formulated 

\noindent {\bf The Conway-Norton Moonshine Conjecture:}
{\sl $\exists$ a graded $\mathbb M$-module $V^\natural$ having the McKay-Thompson series $T_g$ as its $q$-traces.}

In 1980, Griess announced his construction of $\mathbb M$ (the `Friendly Giant'). It is the automorphism group of the Griess algebra, a commutative non-associative algebra of dimension $196,884$.  (Cf. Griess's lecture at the Harvard CMSA on May 6, 2020). This algebra would later become the weight $2$ piece $V^\natural[2]$ of the more elaborate {\it Moonshine Vertex Operator Algebra $V^\natural$!}
The new decade of the 80s also marked an exciting emergence of a representation theory for a class of interrelated infinite dimensional graded algebras, which include the Virasoro, Kac-Moody algebras, $\cW$-algebras, their various `coset' constructions.
For example, Frenkel and Kac gave the first mathematical construction of the level 1 irreducible representations of affine Kac-Moody Lie algebras using `free bosonic vertex operators' $e_\alpha(z)$, $\alpha\in L$, (a mathematical counterpart to physicists vertex operators $:e^{ik\cdot\phi}(z):$) for any weight lattice $L$ of $ADE$ type. A few years later,
Borcherds axiomatized the notion of a {\it vertex algebra} by an infinite set of linear operator identities. 
In 1988, Frenkel-Lepowsky-Meurman gave a new definition (including the Virasoro) of what they called {\it vertex operator algebras} (VOA), based on a Jacobi identity of formal power series. FLM's and Borcherds's formulations are logically equivalent, but FLM's is technically and conceptually a bit easier to work with. FLM also gave a general Fock space construction of a lattice VOA $V_L$ from any even lattice $L$. They also constructed the $ \mathbb{Z}/2\mathbb{Z}$ orbifold of $V_L$ (a VOA counterpart of physicists' orbifold CFT), using the notion of twisted vertex operators. For the Leech lattice $L=\mathbb L$, its $\mathbb{Z}/2\mathbb{Z}$ orbifold $V^\natural$ would yield the Moonshine VOA, with the correct genus 1 partition function $j(q)$.
Marking the beginning of a new mathematical decade in representation theory, Borcherds announced in 1992 his solution to the Conway-Norton Conjecture. The FLM construction of the Moonshine VOA $V^\natural$ plays a central role in his solution.

\subsection{A new universe is born.} We have seen that VOAs arose out of conformal field theory in the 1980s and have been developed mathematically from various points of view in the literature \cite{B,FLM,K}.

Going forward, our perspective in this survey is that VOAs should be regarded as generalizations of {\it differential graded commutative rings}, that is, commutative rings $R$ with an $\mathbb{N}$ or $\frac{1}{2} \mathbb{N}$-grading by weight, equipped with an even derivation $\partial$ which raises the weight by one. This perspective is manifest in the notion of a {\it commutative quantum operator algebra} which was introduced by the first author with Zuckerman \cite{LZII}, and is equivalent to the notion of a VOA. In this survey we shall explain various connections between VOAs and commutative rings, including a functor introduced by Zhu that attaches to a VOA $\cV$ a commutative ring $R_{\cV}$, and Li's canonical decreasing filtration which exists on any VOA such that the associated graded algebra $\text{gr}(\cV)$ is a differential graded commutative ring that contains $R_{\cV}$ as the zeroth graded piece. In a different direction, we also explain how VOAs can be {\it defined over a commutative ring $R$} instead of over a field. The underlying space of a VOA is then an $R$-module rather than a vector space, and if $R$ is the coordinate ring of some variety $X$, we can regard $\cV$ as a vector bundle over $X$. Given an ideal $I\subseteq R$, the quotient of $\cV$ by the VOA ideal generated by $I$ can be viewed as specialization along the corresponding subvariety of $X$. 

The notion of VOAs over commutative rings has two very different applications. First, if $R$ is a ring of rational functions of degree at most zero in a formal variable $\kappa$, it is meaningful to take the limit as $\kappa\rightarrow \infty$. Certain features of the original VOA can be deduced from this limit, which is generally a simpler structure. Second, VOAs with prescribed strong generating type that have as many free parameters in their operator products that are compatible with the axioms of VOAs, can be regarded as {\it universal objects} and are very useful for classifying VOAs by generating type. After introducing these ideas, we will discuss several applications. These include the {\it vertex algebra Hilbert problem}, which asks whether for a strongly finitely generated VOA $\cV$ and a reductive group of $G$ of automorphisms of $\cV$, the orbifold, or invariant subalgebra $\cV^G$, is also strongly finitely generated. We also discuss the structure of cosets of affine VOAs inside larger VOAs. Finally, using all these results we will outline the second author's recent proof with Thomas Creutzig of the {\it Gaoitto-Rap\v{c}\'ak triality conjectures}. This is the statement that the affine cosets of three different $\cW$-(super)algebras are isomorphic. Triality for $\cW$-superalgebras of type $A$ was proven in \cite{CLIV}, and a similar but more involved statement for families of $\cW$-superalgebras in types $B$, $C$, and $D$ appears in \cite{CLV}. These results simultaneously generalize several well-known results including Feigin-Frenkel duality and the coset realization of principal $\cW$-algebras. We conclude with some speculations on the interaction between Zhu's commutative algebra and the orbifold functor, and we recast an old conjecture as a special case of a much more general conjecture that these functors commute in an appropriate sense.

\subsection{Background} In Section \ref{sec:diffgraded}, we introduce the notion of a {\it differential graded commutative algebra (DGA)}. A rich source of examples is the {\it arc space} construction which associates to any commutative $\mathbb{C}$-algebra $A$ the ring $A_{\infty}$ which is the ring of functions on the arc space of the scheme $X = \text{Spec}\ A$. These rings naturally have the structure of abelian vertex algebras, so vertex algebras can be viewed as generalizations of differential commutative rings. In Section \ref{sec:vertex} we introduce the equivalent notions of commutative quantum operator algebras and vertex algebras, and we describe the examples we need, which include free field algebras, affine vertex algebras, and $\cW$-algebras. 

\subsection{From VOAs to odd Batalin-Vilkovisky algebras}
In Section \ref{sec:TVOA}, we give a brief introduction to the notion of a topological VOA (TVOA), first introduced by the first author and Zuckerman \cite{LZI}, and specialize it to a number of cases. Many old and new examples and structures can be seen to arise in this construction. In the first case, we specialize it to the absolute BRST complex of the Virasoro algebra, which is fundamental in the so-called $\cN=0$ conformal string theory. As a parallel, we point out a further application to absolute BRST complex of the $\cN=1$ Virasoro superalgebra arising in the RNS superstring theory. This was done in a joint (unpublished) work of the first author with Moore and Zuckerman \cite{LMZ92}. The main result here is that the cohomology algebra of a TVOA is naturally a Batalin-Vilkovisky (BV) algebra -- a certain odd version of a Poisson algebra. Moreover, the cohomology of a module over a TVOA, also descends to a module over the BV algebra in question. In the case of a BRST complex, the BV structure encompasses many new and old structures. There is a natural second order differential operator -- a BV operator -- that gives rise to an odd Poisson structure on cohomology first observed by Gerstenhaber on the Hochschild cohomology. In particular, the zeroth cohomology group $H^0_{BRST}$ naturally inherits a commutative algebra structure (or commutative superalgebra structure in the case of $\cN=1$ strings); the first cohomology group $H^1_{BRST}$ naturally inherits a Lie algebra structure (or a Lie superalgebra structure in the case of the $\cN=1$ strings). In the case of the $\cN=0$ string coupled to the Liouville theory, also known as the `c=1 model',  $H^0_{BRST}$ is known the `ground ring', introduced by Witten. In general $H^1_{BRST}$ with coefficient in a VOA $V$ with central charge 26 is a Lie algebra that includes the Lie algebra structure of the old `physical states'. An example of this Lie algebra is Borcherds' Monster Lie algebra.

\subsection{Zhu's commutative algebra and Li's filtration} In Section \ref{sec:zhuli}, we recall two functors from the category of VOAs to the category of commutative rings. The first one, which was introduced by Yongchang Zhu \cite{Z}, assigns to a VOA $\cV$ a commutative ring $R_{\cV}$ called {\it Zhu's commutative ring}. It is defined as the vector space quotient of $\cV$ by the span of all elements of the form $a_{(-2)} b$ for all $a,b \in \cV$. The normally ordered product on $\cV$ then descends to a commutative, associative product on $R_{\cV}$. Moreover, strong generators for $\cV$ give rise to generators for $R_{\cV}$, and $R_{\cV}$ is finitely generated if and only if $\cV$ is strongly finitely generated. Recall that $\cV$ is called {\it $C_2$-cofinite} if $R_{\cV}$ is finite-dimensional as a vector space. This is a key starting assumption in Zhu's work on modularity of characters of rational VOAs \cite{Z}, and it has the important consequence that $\cV$ has finitely many simple $\mathbb{Z}_{\geq 0}$-graded modules.

The second functor comes from a canonical decreasing filtration due to Haisheng Li \cite{LiII} that is defined on any VOA $\cV$, such that the associated graded algebra $\text{gr}(\cV)$ is a differential graded commutative ring. Typically, $\cV$ is linearly isomorphic to $\text{gr}(\cV)$, and a strong generating set for $\cV$ gives rise to a generating set for $\text{gr}(\cV)$ as a differential algebra. In fact, $R_{\cV}$ can be identified with the zeroth graded component of $\text{gr}(\cV)$, so $R_{\cV}$ generates $\text{gr}(\cV)$ as a differential algebra. 

In general, it is a difficult problem to find all differential algebraic relations in $\text{gr}(\cV)$, and this is an important and active research program. However, if $\cV$ is {\it freely generated}, meaning that it is linearly isomorphic to a differential polynomial algebra, then $\text{gr}(\cV)$ is just a differential polynomial algebra. Many interesting VOAs have this property, including all free field algebras, universal affine VOAs, and universal $\cW$-algebras. In these cases, the functor $\cV \mapsto \text{gr}(\cV)$ provides a way to systematically study $\cV$ by using the tools of commutative algebra.

\subsection{Orbifolds and the vertex algebra Hilbert problem} In Section \ref{sec:hilbert}, we discuss the {\it orbifold} construction, which begins with a VOA $\cV$ and a group of automorphisms $G$, and considers the invariant subVOA $\cV^G$ and its extensions. This construction was introduced originally in physics; see for example \cite{DVVV,DHVWI,DHVWII}, as well as \cite{FLM} for the construction of the Moonshine VOA  $V^{\natural}$ as an extension of the $ \mathbb{Z}/2\mathbb{Z}$-orbifold of the VOA associated to the Leech lattice.

A VOA $\cV$ is called {\it rational} if every $\mathbb{Z}_{\geq 0}$-graded $\cV$-module is completely reducible. A longstanding conjecture is that if $\cV$ is $C_2$-cofinite and rational, then $\cV^G$ inherits these properties for any finite automorphism group $G$. This has been proven by Carnahan and Miyamoto \cite{CM} when $G$ is a cyclic group. Very recently, McRae has shown that for any finite $G$, rationality of $\cV^G$ would follow from $C_2$-cofiniteness under mild hypothesis \cite{McRI}.

If $\cV$ is strongly finitely generated and $G$ is reductive, a natural question is whether $\cV^G$ is also strongly finitely generated. This is analogous to Hilbert's celebrated theorem that if $G$ is reductive and $V$ is a finite-dimensional $G$-module, the ring of $G$-invariant polynomial functions $\mathbb{C}[V]^G$ is always finitely generated \cite{HI,HII}. Accordingly, we call this the {\it vertex algebra Hilbert problem}. In general, it turns out that the answer is {\it no}. To illustrate the phenomenon, we consider two examples in detail: the rank one degenerate Heisenberg algebra $\cH_{\text{deg}}$ , which is just the free abelian vertex algebra with one generator, and the rank one nondegenerate Heisenberg algebra $\cH$, both with $ \mathbb{Z}/2\mathbb{Z}$ action. The orbifolds $\cH_{\text{deg}}^{ \mathbb{Z}/2\mathbb{Z}}$ and $\cH^{ \mathbb{Z}/2\mathbb{Z}}$ are linearly isomorphic and have the same graded character; however, $\cH_{\text{deg}}^{ \mathbb{Z}/2\mathbb{Z}}$ is not strongly finitely generated whereas $\cH^{ \mathbb{Z}/2\mathbb{Z}}$ is strongly generated by just two fields, by a theorem of Dong and Nagatomo \cite{DN}. The reason for this is quite subtle and depends sensitively on the nonassociativity of the Wick product. 

This theorem turns out to have a vast generalization. First, if we replace $\cH$ by the rank $n$ Heisenberg algebra $\cH(n)$, its full automorphism group is the orthogonal group $O(n)$. This is an example of a {\it free field algebra}, and there are four families of standard free field algebras whose generators are either even or odd, and whose automorphism groups are either orthogonal or symplectic. If $\cV$ is a VOA which is a tensor product of finitely many free field algebras, and $G$ is any reductive group of automorphisms of $\cV$, it was shown in \cite{CLIV} that $\cV^G$ is strongly finitely generated. The key idea is to pass to the associated graded algebra and use an interplay of ideas from classical invariant theory together with the representation theory of VOAs.

In fact, there is a much larger class of vertex algebras that includes not only free field algebras, but all universal affine VOAs and universal $\cW$-algebra, where the vertex algebra Hilbert problem has a positive solution. These vertex algebras can be defined over a suitable ring of rational functions of degree at most zero in some formal variable $\kappa$. They have the property that the $\kappa\rightarrow \infty$ limit is a tensor product of free field algebras, and that passing to the limit commutes with the $G$-invariant functor in an appropriate sense. Finally, the property of strong finite generation of the limit implies strong finite generation of these algebras as one-parameter vertex algebras. It remains an open question to find good criteria for the vertex algebra Hilbert theorem to hold in general, but we make the following conjecture.

\begin{conj} Let $\cV$ be a simple, strongly finitely generated VOA which is $\mathbb{N}$-graded by conformal weight with finite-dimensional graded pieces. Then for any reductive group of automorphisms $G$ of $\cV$, $\cV^G$ is strongly finitely generated.
\end{conj}

\subsection{The structure of coset VOAs} In Section \ref{sec:cosets} we discuss some applications of our results on orbifolds to the structure of {\it coset VOAs}. The coset construction is a basic method for constructing new VOAs from old ones. Given a VOA $\cA$ and a subalgebra $\cV$, the coset $\cC = \text{Com}(\cV, \cA)$ is the subalgebra of $\cA$ which commutes with $\cV$. In general, $\cC \otimes \cV$ is conformally embedded in $\cA$, and there are important links between the representation theory of $\cV$, $\cC$, and $\cA$. It is expected that good properties of $\cA$ and $\cV$ such as rationality and $C_2$-cofiniteness are inherited by $\cC$. In the case where $\cV$ is a Heisenberg or lattice VOA there is a good theory of such cosets due to the second author and Creutzig, Kanade, and Ridout \cite{CKLR}. In particular, if $\cV$ is a lattice VOA and $\cA$ is rational, $\cC$ is always rational. However, few such results are known in a general setting. 

If $\cV$ is an affine VOA, we call $\cC$ an {\it affine coset}. One can ask whether a Hilbert theorem holds for affine cosets: in other words, if $\cV$ is affine and $\cA$ is strongly finitely generated, is $\cC$ also strongly finitely generated? In general this is a difficult question, but there is a large class of VOAs where this can be answered affirmatively. For example, let $\gg$ be any simple Lie superalgebra, $f$ any nilpotent element in the even part of $\mathfrak{g}$, which we complete to an $\gs\gl_2$-triple $\{e,f,h\}$ in $\gg$. Let $\cW^k(\gg,f)$ the corresponding $\cW$-algebra obtained via generalized Drinfeld-Sokolov reduction, which up to isomorphism only depends on the conjugacy class of $f$. Let $\ga$ be the centralizer of the $\gs\gl_2$ in $\gg$; then $\cW^k(\gg,f)$ contains an affine VOA $V^{k'}(\ga)$ for some shifted level $k'$. If $\gb$ is any reductive Lie subalgebra of $\ga$, and $V^{\ell}(\gb) \subseteq V^{k'}(\ga)$ is the corresponding affine subalgebra, then it was shown in \cite{CLIV} that $\text{Com}(V^{\ell}(\gb), \cW^k(\gg,f))$ is strongly finitely generated. The key idea is that such cosets also admit an appropriate limit where the limit algebra is an orbifold of a free field algebra, so the result can be deduced from our general results on orbifolds.

\subsection{Application: the Gaoitto-Rap\v{c}\'ak triality conjecture}
In Section \ref{sec:gaiotto} we outline the second author's recent proof with Thomas Creutzig of the {\it Gaoitto-Rap\v{c}\'ak triality conjectures} \cite{CLIV, CLV}. Triality is the statement that the affine cosets of three different $\cW$-(super)algebras are isomorphic as one-parameter VOAs. In type $A$, these affine cosets (tensored with a rank 1 Heisenberg algebra) are called $Y$-algebras by Gaiotto and Rap\v{c}\'ak in \cite{GR}. These are VOAs arising from local operators at the corner of interfaces of twisted $\mathcal N=4$ supersymmetric gauge theories. These interfaces should satisfy a permutation symmetry which then induces a corresponding symmetry on the associated VOAs. This led \cite{GR} to conjecture a triality of isomorphisms of $Y$-algebras. Similarly, there are families of $\cW$-(super)algebras of types $B$, $C$, and $D$ whose affine cosets are called orthosymplectic $Y$-algebras in \cite{GR}, and were expected to satisfy a similar but more complicated set of trialities.

The triality theorem is a common generalization of many known results including Feigin-Frenkel duality \cite{FFII} and the coset realization of principal $\cW$-algebras of types $A$ and $D$ proven in \cite{ACL}. Another special case is a new coset realization of principal $\cW$-algebras of type $B$ and $C$, which is quite different from the coset realizations of $\cW^k(\gg)$ in the simply-laced cases given in \cite{ACL} since it involves affine vertex superalgebras. There was substantial evidence for these conjectures before, but the key idea in the proof is that the explicit strong generating type of these cosets can be deduced using our general results on the structure of affine cosets. In the type $A$ cases, using Weyl's first and second fundamental theorems of invariant theory for $\gg\gl_m$ \cite{W}, we will show that these cosets are all of type $\cW(2,3,\dots, N)$ for some $N$. This means that they have minimal strong generating sets consisting of one field in each weight $2,3,\dots, N$. Similarly, in the case of types $B$, $C$, and $D$, the affine cosets are all of type $\cW(2,4,\dots, 2N)$ for some $N$. This is shown using Weyl's first and second fundamental theorems of invariant theory for symplectic and orthogonal groups, as well as Sergeev's theorems for the Lie supergroup $Osp(1|2n)$ \cite{SI,SII}.

The second ingredient in the proof is that VOAs of type $\cW(2,3,\dots, N)$ and $\cW(2,4,\dots, 2N)$ satisfying some mild hypotheses can be classified. There exists a universal two-parameter VOA of type $\cW(2,3,\dots)$ defined over the polynomial ring $\mathbb{C}[c,\lambda]$. Its existence and uniqueness was a longstanding conjecture in the physics literature that was recently proved by the second author in \cite{LVI}. The one-parameter quotients of this structure are in bijection with a family of curves in the plane called {\it truncation curves}. All the above affine cosets in type $A$ are such one-parameter quotients, and the triality theorem is proven by explicitly describing their truncation curves. Similarly, there is a universal two-parameter even spin VOA of type $\cW(2,4,\dots)$ which was also conjectured to exist in the physics literature \cite{CGKV}, and was constructed by Kanade and the second author in \cite{KL}. The relevant affine cosets in types $B$, $C$, and $D$ arise as one-parameter quotients of this structure. Again, triality is proven by explicitly describing the truncation curves. 

Intersection points of truncation curves yield nontrivial isomorphisms between different one-parameter quotients of $\cW(c,\lambda)$ at these points, and similarly for the even spin algebra. This yields many nontrivial isomorphisms between seemingly unrelated VOAs at special parameter values. For example, one can classify isomorphisms between the simple cosets in our triality families and various principal $\cW$-algebras. This allows us to prove new rationality results by exhibiting VOAs that were not previously known to be rational, as extensions of known rational VOAs. Many such results are obtained in \cite{CLV}, and as an example we explain the proof that $L_k(\go\gs\gp_{1|2n})$ is rational for all integers $n,k \geq 1$. This completes the classification of affine vertex superalgebras which are rational and $C_2$-cofinite. It was known previously that these were the only $C_2$-cofinite affine vertex superalgebras \cite{GK, AiLin}, but the rationality was only known in the case $n=1$.

\subsection{Orbifolds and the associated variety}
In Section \ref{sec:orbifoldandassvar}, we conclude with some speculations on the interaction between the orbifold and associated variety functors. For any VOA $\cV$, we have Zhu's commutative ring $R_{\cV}$, and Arakawa has defined the associated scheme $\tilde{X}_{\cV} = \text{Spec} \ R_{\cV}$ and the associated variety $X_{\cV}$, which is the corresponding reduced scheme. A strong generating set for $\cV$ always gives rise to a generating set for $R_{\cV}$, so if $\cV$ is strongly finitely generated $X_{\cV}$ is of finite type. In particular, $\cV$ is $C_2$-cofinite if and only if $X_{\cV}$ consists of just one point.

A longstanding open problem in the subject is the following: if $\cV$ is $C_2$-cofinite and $G$ is a finite group of automorphisms, is $\cV^G$ also $C_2$-cofinite? Miyamoto has shown the answer is affirmative if $G$ is cyclic (and hence solvable) in \cite{MiIII}. Jointly with Carnahan \cite{CM}, he has also shown that when $\cV$ is in addition rational, then $\cV^G$ is rational when $G$ is solvable. If $G$ is not solvable the $C_2$-cofiniteness question remains open.

It is perhaps fruitful to ask a more general question: for a finite group $G$, does the $G$-invariant functor commute with the associated variety functor? More precisely, if $G$ acts on a VOA $\cV$, there is an induced action of $G$ on $R_{\cV}$. The inclusion $\cV^G \hookrightarrow \cV$ induces a homomorphism of rings $R_{\cV^G} \rightarrow \cR_{\cV}$ whose image clearly lies in $(R_{\cV})^G$. It is easy to find examples where this map $R_{\cV^G} \rightarrow (\cR_{\cV})^G$ fails to be an isomorphism, but it is natural to ask whether it becomes an isomorphism at the level of reduced rings. If this is the case, and if $\cV^G$ is also strongly finitely generated, this would imply the above conjecture as a special case. We end with a few nontrivial examples where this isomorphism of reduced rings indeed holds.

\section{Differential graded algebras} \label{sec:diffgraded} 
Throughout this paper, our base field will always be $\mathbb{C}$ and all rings will be $\mathbb{C}$-algebras. A {\it differential graded algebra} (DGA) will be a commutative $\mathbb{C}$-algebra $A$ equipped with a grading which we call {\it weight}: $$A = \bigoplus_{n \geq 0} A_n,\qquad n \in \mathbb{Z} \ \text{or} \ \frac{1}{2}\mathbb{Z},$$ and a derivation $\partial$ on $A$, that is, a $\mathbb{C}$-linear operator satisfying $\partial (ab) = (\partial a) b + b \partial a$ for $a,b\in A$. We also declare that $\partial$ is homogeneous of weight $1$, meaning that $\partial (A_n) \subseteq A_{n+1}$. We will also sometimes work with differential graded superalgebras; these have an additional $ \mathbb{Z}/2\mathbb{Z}$-grading $A = A^{\bar{0}} \oplus A^{\bar{1}}$ which is compatible with the weight grading in the sense that $A = \bigoplus_{n\geq 0} A_n^{\bar{0}} \oplus A_n^{\bar{1}}$. We write $|a| = \bar{i}$ if $a \in A^{\bar{i}}$, and $A$ is supercommutative meaning that $a b - (-1)^{|a| |b|} ba = 0$. The derivation $\partial$ is then required to preserve the $ \mathbb{Z}/2\mathbb{Z}$-grading.

If $\{a_i|\ i\in I\}$ is a subset of $A$ such that $\{\partial^k a_i|\ i\in I,\ \ k \geq 0\}$ generates $A$ as a ring, we say that $A$ is {\it differentially generated} by $\{a_i|\ i\in I\}$. If $I$ can be taken to be a finite set, we say that $A$ is {\it differentially finitely generated}.
 
There is a large class of DGAs that arise naturally from commutative $\mathbb{C}$-algebras via the {\it arc space} construction, which we now recall. First, let $X$ be an irreducible scheme of finite type over $\mathbb{C}$. The zeroth jet scheme $X_0$ is just $X$, and the first jet scheme $X_1$ is the total tangent space of $X$. It is determined by its functor of points: for every $\mathbb{C}$-algebra $A$, we have a bijection
$$\text{Hom} (\text{Spec}  \ A, X_1) \cong \text{Hom} (\text{Spec}  \ A[t]/\langle t^{2}\rangle , X).$$  Similarly, for $m>1$, the $m^{\text{th}}$ jet scheme $X_m$ is determined by a bijection
$$\text{Hom} (\text{Spec} \ A, X_m) \cong \text{Hom} ( \text{Spec}  \ A[t]/\langle t^{m+1}\rangle , X)$$ for any $\mathbb{C}$-algebra $A$. The assignment $X\mapsto X_m$ is functorial, and a morphism $f:X\ra Y$ of schemes induces $f_m: X_m \ra Y_m$ for all $m\geq 1$. If $X$ is nonsingular, $X_m$ is irreducible and nonsingular for all $m$, and is just an affine bundle over $X$ with fiber the affine space of dimension $m\cdot \text{dim}(X)$. If $X$ is singular, the jet schemes are more subtle and encode information about the singularities of $X$.

There are projections $X_{m+1} \rightarrow X_{m}$ for all $m$, and the {\it arc space} is defined to be the inverse limit $$X_{\infty}=\lim_{\leftarrow} X_m.$$ Although it is generally not of finite type, $X_{\infty}$ has some better properties than the finite jet schemes. For example, by a theorem of Kolchin \cite{Kol}, $X_{\infty}$ it is always irreducible even though $X_m$ can be reducible for all $m\geq 1$. Arc spaces were first studied by Nash in \cite{Na}, and also appear in Kontsevich's theory of motivic integration \cite{Kon}. This theory has been developed by many authors including Batyrev, Craw, Denef, Ein, Loeser, Looijenga, Mustata, and Veys; see for example \cite{Bat,Cr,DLI,DLII,EM,Loo,M,Ve}.

Suppose that $X$ is the affine scheme $\text{Spec}\ R$ for a finitely generated $\mathbb{C}$-algebra $R$ with a presentation
$$R= \mathbb{C}[y_1,\dots,y_r] / \langle f_1,\dots, f_k\rangle.$$ Then $X_m$ is again an affine scheme and we can find explicit equations for it using the presentation of $R$ as follows. First, a morphism $\text{Spec}\ \mathbb{C}[t] / \langle t^{m+1} \rangle \rightarrow X$ corresponds to a ring homomorphism $\phi: R\rightarrow \mathbb{C}[t] / \langle t^{m+1} \rangle$, that is, an assignment 
\begin{equation} \phi(y_j) = y_j^{(0)} + y_j^{(1)} t + \cdots + y_j^{(m)} t^m,\qquad j=1,\dots, r,\end{equation} such that 
\begin{equation} \label{jet:requirement} f_{\ell} (\phi(y_1),\dots, \phi(y_r)) = 0,\ \text{for all} \ \ell =1,\dots, k.\end{equation}

We regard $y_i^{(i)}$ for $i=0,1,\dots,m$ as coordinate functions on $X_m$, and we define a derivation $D$ on the polynomial ring $\mathbb{C}[y_j^{(i)}|\ j = 1,\dots, r, \  i=0,\dots, m]$ as follows. On the generators, we have
$$D(y_j^{(i)}) = y_j^{(i+1)}, \ \text{for}\ i<m, \ \text{and} \ D(y_j^{(m)}) =0,$$ and this is extended by the Leibniz rule to monomials in the generators, and then by $\mathbb{C}$-linearity to the polynomial algebra.
In particular, $f_{\ell}^{(s)} = D^s( f_{\ell})$ is a well-defined polynomial in $\mathbb{C}[y_j^{(i)}|\ j = 1,\dots, r, \  i=0,\dots, m]$. The requirement \eqref{jet:requirement} translates to the condition that the coordinate functions $y_j^{(i)}$ satisfy $f_{\ell}^{(s)}$ for $\ell = 1,\dots, k$ and $0\leq s \leq m$. Therefore $X_m$ is the affine scheme $\text{Spec} \ R_m$, where 
$$R_m = \mathbb{C}[y_j^{(i)}|\ j = 1,\dots, r, \  i=0,\dots, m]/ \langle \{f_1^{(s)},\dots, f_k^{(s)}|\ s = 0,\dots, m\}\rangle.$$ By identifying $y_j$ with $y_j^{(0)}$, we see that $R$ is naturally a subalgebra of $R_m$. Additionally, we have $\mathbb{N}$-grading $R_m = \bigoplus_{n\geq 0} R_m[n]$ by weight, defined by $\text{wt}(y^{(i)}_j) = i$, and $R_m[0] = R$.

Similarly, if $X = \text{Spec}\ R$ as above, $X_{\infty}$ is an affine scheme of infinite type, and we can realize it similarly as $\text{Spec}\ R_{\infty}$ where 
$$R_{\infty} = \mathbb{C}[y_j^{(i)}|\ j = 1,\dots, r, \  i \geq 0]/ \langle \{f_1^{(s)},\dots, f_k^{(s)}|\ s \geq 0\}\rangle.$$ As above, $R_{\infty}$ is $\mathbb{N}$-graded by weight, and the weight zero component $R_{\infty}[0]$ is isomorphic to $R$, and differentially generates $R_{\infty}$. In fact, $R_{\infty}$ satisfies the following universal property: if $S$ is any differential graded ring $S = \bigoplus_{n\geq 0} S[n]$ such that $S[0] \cong R$ and $S$ is differentially generated by $S[0]$, then $S$ is a quotient of $R_{\infty}$.

\section{Vertex algebras} \label{sec:vertex}

Vertex algebras have been discussed from several points of view in the literature (see for example \cite{B,FLM,FHL,K,FBZ}). We will follow the formalism developed by the first author jointly with Zuckerman in \cite{LZII}, and partly in \cite{LiI}. 

The parallel between vertex algebras and commutative algebras was somewhat apparent in an early formulation of a chiral algebras by string theorists. Roughly, the chiral algebra of a two-dimensional conformal field theory (2d CFT) consists of the `left-moving' operators of the theory. One of the fundamental axioms of 2d CFTs is the `locality axiom', which states that the correlation functions of observables are real meromorphic functions on the Riemann sphere that are symmetric in the observables after analytic continuations \cite{MS}. In particular, when the observables in question are left-moving operators, their correlation functions are rational functions that are symmetric in the operators, again after analytic continuations. This turns out to be completely equivalent to the main axiom of a vertex algebra, as first observed in \cite{LZII}.

Let $V=V_0\oplus V_1$ be a super vector space over $\mathbb{C}$, $z,w$ be formal variables, and $\text{QO}(V)$ be the space of linear maps $$V\ra V((z))=\{\sum_{n\in\mathbb{Z}} v(n) z^{-n-1}|
v(n)\in V,\ v(n)=0\ \text{for} \ n>\!\!>0 \}.$$ Each element $a\in \text{QO}(V)$ can be represented as a power series
$$a=a(z)=\sum_{n\in\mathbb{Z}}a(n)z^{-n-1}\in \text{End}(V)[[z,z^{-1}]].$$ We assume that $a=a_0+a_1$ where $a_i:V_j\ra V_{i+j}((z))$ for $i,j\in \mathbb{Z}/2\mathbb{Z}$, and we write $|a_i| = i$.

For each $n \in \mathbb{Z}$, we have a bilinear operation on $\text{QO}(V)$, defined on homogeneous elements $a$ and $b$ by
$$ a(w)_{(n)}b(w)=\text{Res}_z a(z)b(w)\ \iota_{|z|>|w|}(z-w)^n- (-1)^{|a||b|}\text{Res}_z b(w)a(z)\ \iota_{|w|>|z|}(z-w)^n.$$
Here $\iota_{|z|>|w|}f(z,w)\in\mathbb{C}[[z,z^{-1},w,w^{-1}]]$ denotes the power series expansion of a rational function $f$ in the region $|z|>|w|$. For $a,b\in \text{QO}(V)$, we have the following identity of power series known as the {\it operator product expansion} (OPE):
 \begin{equation}\label{opeform} a(z)b(w)=\sum_{n\geq 0}a(w)_{(n)} b(w)\ (z-w)^{-n-1}+:a(z)b(w):. \end{equation}
Here $:a(z)b(w):\ =a(z)_-b(w)\ +\ (-1)^{|a||b|} b(w)a(z)_+$, where $a(z)_-=\sum_{n<0}a(n)z^{-n-1}$ and $a(z)_+=\sum_{n\geq 0}a(n)z^{-n-1}$. Often, \eqref{opeform} is written as
$$a(z)b(w)\sim\sum_{n\geq 0}a(w)_{(n)} b(w)\ (z-w)^{-n-1},$$ where $\sim$ means equal modulo the term $:a(z)b(w):$, which is regular at $z=w$. 

Note that $:a(w)b(w):$ is a well-defined element of $\text{QO}(V)$. It is called the {\it Wick product} or {\it normally ordered product} of $a$ and $b$, and it
coincides with $a_{(-1)}b$. For $n\geq 1$ we have
$$ n!\ a(z)_{(-n-1)} b(z)=\ :(\partial^n a(z))b(z):,\qquad \partial = \frac{d}{dz}.$$
For $a_1(z),\dots ,a_k(z)\in \text{QO}(V)$, the $k$-fold iterated Wick product is defined inductively by
\begin{equation}\label{iteratedwick} :a_1(z)a_2(z)\cdots a_k(z):\ =\ :a_1(z)b(z):,\qquad b(z)=\ :a_2(z)\cdots a_k(z):.\end{equation}
We usually omit the formal variable $z$ when no confusion can arise.

A subspace $\cA\subseteq \text{QO}(V)$ containing $1$ which is closed under all the above products is called a {\it quantum operator algebra} (QOA). We say that $a,b\in \text{QO}(V)$ are {\it local} if 
$$(z-w)^N [a(z),b(w)]=0$$ for some $N\geq 0$. A {\it commutative quantum operator algebra} is a QOA whose elements are pairwise local. This is well known to be equivalent to the notion of a vertex algebra in the sense of \cite{FLM}, and from now on we will use these terms interchangeably.

Many formal algebraic notions such as homomorphisms, ideals, quotients, modules, etc., are immediately clear. A vertex algebra $\cA$ is called {\it simple} if it has no nontrivial ideals. A vertex algebra $\cA$ is {\it generated} by a subset $S=\{\alpha^i|\ i\in I\}$ if $\cA$ is spanned by words in the letters $\alpha^i$, and all products, for $i\in I$ and $n\in\mathbb{Z}$. We say that $S$ {\it strongly generates} $\cA$ if $\cA$ is spanned by words in the letters $\alpha^i$, and all products for $n<0$. Equivalently, $\cA$ is spanned by $$\{ :\partial^{k_1} \alpha^{i_1}\cdots \partial^{k_m} \alpha^{i_m}:| \ i_1,\dots,i_m \in I,\ k_1,\dots,k_m \geq 0\}.$$ 
Suppose that $S$ is an ordered strong generating set $\{\alpha^1, \alpha^2,\dots\}$ for $\cA$ which is at most countable. We say that $S$ {\it freely generates} $\cA$, if $\cA$ has a Poincar\'e-Birkhoff-Witt basis
\begin{equation} \label{freegenerators} \begin{split} & :\partial^{k^1_1} \alpha^{i_1}  \cdots \partial^{k^1_{r_1}}\alpha^{i_1} \partial^{k^2_1}  \alpha^{i_2} \cdots \partial^{k^2_{r_2}}\alpha^{i_2}
 \cdots \partial^{k^n_1} \alpha^{i_n} \cdots  \partial^{k^n_{r_n}} \alpha^{i_n}:,\qquad 
 1\leq i_1 < \dots < i_n,
 \\ & k^1_1\geq k^1_2\geq \cdots \geq k^1_{r_1},\quad k^2_1\geq k^2_2\geq \cdots \geq k^2_{r_2},  \ \ \cdots,\ \  k^n_1\geq k^n_2\geq \cdots \geq k^n_{r_n},
 \\ &  k^{t}_1 > k^t_2 > \dots > k^t_{r_t} \ \ \text{if} \ \ \alpha^{i_t}\ \ \text{is odd}. 
 \end{split} \end{equation}

We recall some important identities that hold in any vertex algebra $\cA$. For fields $a,b,c \in \cA$, we have
\begin{equation} \label{VOA:identities} 
\begin{split} &  (\partial a)_{(n)} b  = -na_{n-1}, \qquad \forall n\in \mathbb{Z},
\\ & a_{(n)} b   = (-1)^{|a||b|} \sum_{p \in \mathbb{Z}} (-1)^{p+1} (b_{(p)} a)_{(n-p-1)} 1,\qquad \forall n\in \mathbb{Z},
\\ &:(:ab:)c:\  - \ :abc:\  =  \sum_{n\geq 0}\frac{1}{(n+1)!}\big( :(\partial^{n+1} a)(b_{(n)} c):
 + (-1)^{|a||b|} (\partial^{n+1} b)(a_{(n)} c):\big), 
\\ & a_{(n)} (:bc:) -\ :(a_{(n)} b)c:\ - (-1)^{|a||b|}\ :b(a_{(n)} c): \  = \sum_{i=1}^n
\binom{n}{i} (a_{(n-i)}b)_{(i-1)}c, \qquad \forall n \geq 0,
\\ & a_{(r)}(b_{(s)} c)   = (-1)^{|a||b|} b_{(s)} (a_{(r)}c) + \sum_{i =0}^r \binom{r}{i} (a_{(i)}b)_{(r+s - i)} c,\qquad \forall r,s \geq 0.\end{split} \end{equation}
The last identities above are known as {\it Jacobi identities} of type $(a,b,c)$.

\subsection{Conformal structure} A conformal structure with central charge $c$ on a vertex algebra $\cA$ is a Virasoro vector $L(z) = \sum_{n\in \mathbb{Z}} L_n z^{-n-2} \in \cA$ satisfying
\begin{equation} \label{virope} L(z) L(w) \sim \frac{c}{2}(z-w)^{-4} + 2 L(w)(z-w)^{-2} + \partial L(w)(z-w)^{-1},\end{equation} such that $L_{-1} \alpha = \partial \alpha$ for all $\alpha \in \cA$, and $L_0$ acts diagonalizably on $\cA$. We say that $\alpha$ has conformal weight $d$ if $L_0(\alpha) = d \alpha$, and we denote the conformal weight $d$ subspace by $\cA[d]$. In all our examples, this grading will be by $\mathbb{Z}_{\geq 0}$ or $\frac{1}{2} \mathbb{Z}_{\geq 0}$. We say $\cA$ is of type $\cW(d_1,d_2,\dots; e_1,e_1,\dots)$, if it has a minimal strong generating set consisting of one even field in each conformal weight $d_1, d_2, \dots $, and one odd field in each conformal weight $e_1, e_2,\dots$. A vertex algebra with a conformal structure is known as a {\it vertex operator algebra} (VOA). Most examples that appear in the literature have this structure, but there are important exceptions including abelian vertex algebras, affine vertex algebras at the critical level, and certain free field algebras which we discuss later in Section \ref{sec:freewalg}.

\subsection{Abelian vertex algebras}
The category of abelian vertex algebras is equivalent to the category of differential graded (super)commutative rings. In particular, if $A$ is such a ring with differential $D$, for any $a\in A$, we can define a field 
$$a(z) = \sum_{n \in \mathbb{Z}} a(n)z^{-n-1} \in End(A)[[z, z^{-1}]].$$ Here $a(n) = 0$ for $n\geq 0$, and $a(-n-1) \in End(A)$ is left multiplication by $\frac{1}{n!} D^n a$. All OPEs are trivial in this algebra, and it is isomorphic to $A$ as a vector space. A rich source of examples comes from arc spaces: for any commutative ring $R$, the ring $R_{\infty}$ has an abelian vertex algebra structure. Note that $R_{\infty}$ has no conformal structure, but it has an action of the subalgebra $\{L_n|\ n\geq -1\}$ of the Virasoro algebra, which is called a quasiconformal structure by Ben-Zvi and Frenkel in \cite{FBZ}. This is enough for their main construction to work, namely, the coordinate-free description of the vertex operation on any Riemann surface. Finally, we note that for abelian vertex algebras, the notion of generation and strong generation are the same. In general, these notions are very different; there are vertex algebras such as the universal $\cW_{\infty}$-algebras of type $\cW(2,3,\dots)$ and $\cW(2,4,\dots)$ \cite{LVI,KL} that are generated by just one field, but are not strongly finitely generated.

\subsection{Free field algebras} The simplest nonabelian vertex algebras are the free field algebras. These have the property that only the vacuum appears in the OPEs between the generating fields. They are important because many interesting vertex algebras admit free field realizations; that is, they can be realized as subalgebras of free field algebras. Free field algebras have another useful interpretation as large level limits of affine VOAs and $\cW$-algebras, which we shall discuss later in Sections \ref{sec:hilbertaffine} and \ref{sec:freewalg}. 

There are four standard examples of free field algebras which we describe below. They are defined starting from a finite-dimensional vector space $V$ with either a symmetric or symplectic form, and the generators are either even or odd. These are simple and freely generated, and their full automorphism groups are the same as the group of linear automorphisms of $V$ which preserve the bilinear form. These examples all have conformal structures, but later we will introduce a more general class of free field algebras that do not have conformal structures.

\bigskip

\noindent {\it Heisenberg algebra}. Let $V$ be an $n$-dimensional complex vector space equipped with a nondegenerate, symmetric bilinear form $\langle, \rangle$. Associated to $V$ is a VOA $\cH(V)$ called a Heisenberg algebra. It has even generators $\alpha^u$ for $u \in V$, and OPE relations  $$\alpha^u(z) \alpha^v(w) \sim \langle u, v \rangle (z-w)^{-2}.$$ The full automorphism group of $\cH(V)$, that is, the group of linear automorphisms which preserve the OPEs, is the orthogonal group $O(n)$, since the automorphisms of $\cH(V)$ are exactly the linear automorphisms of $V$ which preserve the pairing $\langle, \rangle$. 

Often it is convenient to choose an orthonormal basis $v_1,\dots, v_n$ for $V$ relative to $\langle, \rangle$. We denote the corresponding fields by $\alpha^{i},\dots, \alpha^n$, which then satisfy $\alpha^{i} (z) \alpha^{j}(w) \sim \delta_{i,j} (z-w)^{-2}$. We often denote $\cH(V)$ by $\cH(n)$, since it is just the tensor product of $n$ copies of the rank one Heisenberg VOA. Note that $\cH(n)$ has a Virasoro element $$L^{\cH} =  \frac{1}{2} \sum_{i=1}^n  :\alpha^i \alpha^i:$$ of central charge $n$, under which $\alpha^i$ is primary of weight one. \bigskip

{\it Free fermion algebra}. One can also associate to the $n$-dimensional vector space $V$ the free fermion algebra $\cF(V)$. It is a vertex superalgebra with odd generators $\phi^u$ which are linear in $u \in V$, and satisfy $$\phi^u(z) \phi^v(w) \sim \langle u, v \rangle (z-w)^{-1}.$$ As above, the full automorphism group of $\cF(V)$ is the orthogonal group $O(n)$. If we fix an orthonormal basis $v_1,\dots, v_n$ for $V$ relative to $\langle, \rangle$, we denote the corresponding fields by $\phi^{1},\dots, \phi^n$, and they satisfy $\phi^{i} (z) \phi^{j}(w) \sim \delta_{i,j} (z-w)^{-1}$. We often denote $\cF(V)$ by $\cF(n)$. We have the Virasoro element  $$L^{\cF} =  -\frac{1}{2} \sum_{i=1}^n  :\phi^i \partial \phi^i:$$ of central charge $\frac{n}{2}$, under which $\phi^i$ is primary of weight $\frac{1}{2}$.

\bigskip

{\it $\beta\gamma$-system}. Let $V$ now be a symplectic vector space of dimension $2n$, with skew-symmetric bilinear form $\langle, \rangle$. We can attach to $V$ a VOA $\cS(V)$ with even generators $\psi^v$ which are linear in $v\in V$, and satisfy $$\psi^u(z) \psi^v(w) \sim \langle u, v \rangle (z-w)^{-1}.$$ There is an action of the symplectic group $Sp(2n)$ on $\cS(V)$ by linear transformation on $V$ which preserve $\langle, \rangle$. As usual, we fix a symplectic basis $u_1,\dots, u_n, v_1,\dots, v_n$ for $V$, so that $\langle u_i, v_j\rangle = \delta_{i,j}$ and $\langle u_i, u_j\rangle = 0 = \langle v_i, v_j\rangle$, and denote the corresponding elements $\phi^{u_i}$ and $\phi^{v_i}$ by $\beta^i$ and $\gamma^i$, respectively. Then \begin{equation} \label{eq:betagammaope} \begin{split} \beta^i(z)\gamma^{j}(w) &\sim \delta_{i,j} (z-w)^{-1},\quad \gamma^{i}(z)\beta^j(w)\sim -\delta_{i,j} (z-w)^{-1},\\  \beta^i(z)\beta^j(w) &\sim 0,\qquad\qquad\qquad \gamma^i(z)\gamma^j (w)\sim 0.\end{split} \end{equation}We often use the notation $\cS(n)$; it has Virasoro element $$L^{\cS} = \frac{1}{2} \sum_{i=1}^n \big(:\beta^{i}\partial\gamma^{i}: - :\partial\beta^{i}\gamma^{i}:\big)$$ of central charge $-n$, under which $\beta^{i}$, $\gamma^{i}$ are primary of weight $\frac{1}{2}$. Note that $Sp(2n)$ is the full automorphism group of $\cS(n)$ preserving $L^{\cS}$. One can deform the weights of the fields $\beta^i,\gamma^i$ to $(\lambda,1-\lambda)$ by deforming the Virasoro element to
$$L_\lambda^{\cS} = \sum_{i=1}^n \big(\lambda:\beta^{i}\partial\gamma^{i}: -(1-\lambda) :\partial\beta^{i}\gamma^{i}:\big)$$ 
under which the $\beta^i$, $\gamma^i$ are primary of weight $(\lambda,1-\lambda)$. Note that for $\lambda \neq  \frac{1}{2}$, $L_{\lambda}^{\cS}$ is preserved only by the subgroup $GL(n) \subseteq Sp(2n)$.

\bigskip

{\it $bc$-system}. There is a vertex superalgebra $\cE(n)$ which is analogous to the $\beta\gamma$-system $\cS(n)$, but is now based on an orthogonal space $V$ of dimension $2n$. The bosonic fields $\beta^i,\gamma^i$ are replaced by their fermionic counterparts $b^i,c^i$ of the same conformal weights, and the defining OPEs now become
\begin{equation} \label{eq:betagammaope} \begin{split} b^i(z)c^{j}(w) &\sim \delta_{i,j} (z-w)^{-1},\quad c^{i}(z)b^j(w)\sim \delta_{i,j} (z-w)^{-1},\\  b^i(z)b^j(w) &\sim 0,\qquad\qquad\qquad c^i(z)c^j (w)\sim 0.\end{split} \end{equation}
Analogously, this algebra contains the Virasoro element
$$L_\lambda^{\cE} = \sum_{i=1}^n \big(\lambda:b^{i}\partial c^{i}: -(1-\lambda) :\partial b^{i} c^{i}:\big)$$
which gives $b^i,c^i$ the conformal weights $(\lambda,1-\lambda)$. Note that in our earlier notation, $\cE(n)$ is isomorphic to the free fermion algebra $\cF(2n)$, and the conformal vector $L^{\cF}$ above corresponds to $\lambda = \frac{1}{2}$.
In the case where $V$ is two dimensional, and $\lambda=2$, the $bc$-system is a basic ingredient of the $\cN=0$ BRST complex (playing the role of the conformal `bosonic ghosts').
For the $\cN=1$ BRST complex, the $\beta\gamma$-system is the fermionic counterpart of the $bc$-system playing the parallel role of the `fermionic ghosts' of the same weights.

\bigskip

{\it Symplectic fermion algebra}.  Finally, we can attach to the $2n$-dimensional symplectic vector space $V$ a vertex superalgebra $\cA(V)$ called a symplectic fermion algebra. It has odd generators $\xi^v$ which are linear in $v \in V$, and satisfy $$\xi^u(z) \xi^v(w) \sim \langle u, v \rangle (z-w)^{-2}.$$ As above, $\cA(V)$ has full automorphism group the symplectic group $Sp(2n)$. Fixing a symplectic basis  $u_1,\dots, u_n, v_1,\dots, v_n$ for $V$ as above, we denote the corresponding generartors $\xi^{u_i}$ and $\xi^{v_i}$ by $e^{i}$ and $f^{i}$, respectively. They satisfy  \begin{equation} \label{eq:sfope} \begin{split} e^{i} (z) f^{j}(w)&\sim \delta_{i,j} (z-w)^{-2},\quad f^{j}(z) e^{i}(w)\sim - \delta_{i,j} (z-w)^{-2},\\ e^{i} (z) e^{j} (w)&\sim 0,\qquad\qquad\qquad f^{i} (z) f^{j} (w)\sim 0.\end{split} \end{equation} It has the Virasoro element $$L^{\cA} =  - \sum_{i=1}^n  :e^i f^i:$$ of central charge $-2n$, under which $e^i, f^i$ are primary of weight one.

\subsection{Affine vertex algebras} Let $\gg$ be a simple, finite-dimensional, Lie (super)algebra with dual Coxeter number $h^{\vee}$, equipped with the standard supersymmetric invariant bilinear form $(-|-)$. The {\it universal affine vertex (super)algebra} $V^k(\gg)$ associated to $\gg$ is freely generated by elements $X^{\xi}$ which are linear in $\xi \in\gg$, and satisfy
$$X^{\xi}(z)X^{\eta} (w)\sim k(\xi |\eta) (z-w)^{-2} + X^{[\xi,\eta]}(w) (z-w)^{-1} .$$ We may choose dual bases $\{\xi\}$ and $\{\xi'\}$ of $\gg$, satisfying $(\xi'|\eta)=\delta_{\xi,\eta}$. If $k+h^\vee\neq 0$, there is a Virasoro element
\begin{equation} \label{sugawara}
L^{\gg} = \frac{1}{2(k+h^\vee)}\sum_\xi :X^{\xi}X^{\xi'}:
\end{equation} of central charge $ c= \frac{k \cdot \text{sdim}(\gg)}{k+h^\vee}$. This is known as the {\it Sugawara conformal vector}, and each $X^{\xi}$ is primary of weight one. As a module over the affine Lie (super)algebra $\hat{\gg} = \gg[t,t^{-1}] \oplus \mathbb{C}$, $V^k(\gg)$ is isomorphic to the vacuum $\hat{\gg}$-module, and is freely generated by $\{X^{\xi_i}\}$ as $\xi_i$ runs over a basis of $\gg$. We denote by $L_k(\gg)$ the simple quotient of $V^k(\gg)$ by its maximal proper ideal graded by conformal weight.

By a celebrated theorem of Frenkel and Zhu \cite{FZ}, for a simple Lie algebra $\gg$, $L_k(\gg)$ is $C_2$-cofinite and rational if and only if $k\in \mathbb{N}$. It is also known that for $\gg$ a simple Lie superalgebra, $L_k(\gg)$ is $C_2$-cofinite only in the case where $\gg = \go\gs\gp_{1|2n}$ for some $n\geq 1$ and $k \in \mathbb{N}$ \cite{GK,AiLin}. The rationality of $L_k(\go\gs\gp_{1|2n})$ was long expected, but until recently was only known for $\go\gs\gp_{1|2}$ \cite{CFK}. This was proven by the second author and Creutzig in \cite{CLV} as an application of the Gaiotto-Rap\v{c}\'ak triality conjecture.

\subsection{Affine $\cW$-algebras}
The $\cW$-algebras $\cW^k(\gg,f)$ are an important family of VOAs associated to a simple Lie (super)-algebra $\gg$ and an even nilpotent element $f \in \gg$. They are a common generalization of affine vertex superalgebras (the case $f = 0$) and the Virasoro algebra (the case $\gg = \gs\gl_2$ and $f$ the principal nilpotent element). The first example beyond these cases is the Zamolodchikov $\cW_3$-algebra which corresponds to $\gg = \gs\gl_3$ and $f$ the principal nilpotent \cite{Zam}. It is of type $\cW(2,3)$ meaning that it has one strong generator in weights $2$ and $3$, and its structure is more complicated than that of affine and Virasoro VOAs since the OPE of the weight $3$ field with itself contains nonlinear terms. For a general simple Lie algebra $\gg$, the definition of the principal $\cW$-algebra $\cW^k(\gg) = \cW^k(\gg,f_{\text{prin}})$ using quantum Drinfeld-Sokolov reduction was given by Feigin and Frenkel in \cite{FFI}. It is of type $\cW(d_1,\dots, d_m)$, where $d_1,\dots, d_m$ are the degrees of the fundamental invariants of $\gg$. The definition of $\cW^k(\gg,f)$ for an arbitrary simple Lie superalgebra $\gg$ and even nilpotent element $f \in \gg$, is due to Kac, Roan, and Wakimoto \cite{KRW}, and is a generalization of quantum Drinfeld-Sokolov reduction.

The principal $\cW$-algebras $\cW^k(\gg)$ have appeared in several important problems in mathematics and physics including the Alday-Gaiotto-Tachikawa correspondence \cite{AGT,BFN,MO,SV}, and the quantum geometric Langlands program \cite{AF,CG,FrII,FG, GI,GII}. They are closely related to the classical $\cW$-algebras which arose in work of Adler, Gelfand, Dickey, Drinfeld, and Sokolov involving integrable hierarchies of soliton equations \cite{Ad,GD,Di,DS}. For a general nilpotent element $f\in \gg$, $\cW^k(\gg,f)$ can also be regarded as a chiralization of the finite $\cW$-algebra $\cW^{\text{fin}}(\gg,f)$ of Premet \cite{Pre}; see \cite{DSKII}.

$\cW$-algebras are expected to give rise to many new examples of rational conformal field theories via the {\it Kac-Wakimoto conjecture} \cite{KWIV}, which was later refined by Arakawa and van Ekeren \cite{ArIII,AvE}. Let $\gg$ be a simple Lie algebra and $k = -h^{\vee} + \frac{p}{q}$ an admissible level for $\widehat{\gg}$. The associated variety of $L_k(\gg)$ is then the closure of a nilpotent orbit $\mathbb{O}_q$ which depends only on the denominator $q$ \cite{ArIII}. If $f\in \gg$ is a nilpotent lying in $\mathbb{O}_q$, the simple $\cW$-algebra $\cW_k(\gg,f)$ is known to be nonzero and $C_2$-cofinite \cite{ArIII}. Such pairs $(f,q)$ are called {\it exceptional pairs} in \cite{AvE}, and they generalize the notion of exceptional pair due to Kac and Wakimoto \cite{KWIV} and Elashvili, Kac, and Vinberg \cite{EKV}. The corresponding $\cW$-algebras are known as {\it exceptional $\cW$-algebras} and were conjectured by Arakawa to be rational in \cite{ArIII}, generalizing the original conjecture of \cite{KWIV}. In a celebrated paper \cite{ArIV}, Arakawa proved the rationality in full generality for principal nilpotents. More recently, Arakawa and van Ekeren have proven rationality of all exceptional $\cW$-algebras in type $A$, and all exceptional subregular $\cW$-algebras of simply laced types \cite{AvE} \footnote{After this paper was written, McRae proved the rationality of exceptional $\cW$-algebras in full generality in \cite{McRII}}. One application of the Gaiotto-Rap\v{c}\'ak triality is to prove some new cases of this conjecture involving non-principal $\cW$-algebras in types $B$ and $C$ that were not accessible with existing technology. The triality theorems also give new examples of rational $\cW$-algebras that are not exceptional.

We shall briefly recall the definition and basic properties of $\cW$-algebras, following \cite{KWIII,KRW}. Let $\gg$ be a simple Lie superalgebra with nondegenerate invariant supersymmetric bilinear form $( \ \ | \ \ ): \gg \times \gg \rightarrow \mathbb{C}$.
Fix a basis $\{q^\alpha\}_{\alpha \in S}$ of $\gg$ indexed by a set $S$, which is homogeneous with respect to the $ \mathbb{Z}/2\mathbb{Z}$-grading on $\gg$. Here  
$$|\alpha| = \begin{cases} 0 & \ q^\alpha \ \text{even} \\ 1 & \ q^\alpha \ \text{odd} \end{cases}$$
We define the structure constants in $\gg$ by
$$[q^\alpha, q^\beta] = \sum_{\gamma \in S}f^{\alpha\beta}_{\gamma} q^\gamma.$$
The affine VOA $V^k(\gg)$ is strongly generated by fields $\{X^\alpha\}_{\alpha \in S}$ satisfying
$$
X^\alpha(z)X^\beta(w) \sim k(q^\alpha|q^\beta)(z-w)^{-2} + \sum_{\gamma\in S} f^{\alpha \beta}_{\gamma} X^\gamma (w) (z-w)^{-1}.$$
We define $X_\alpha$ to be the field corresponding to $q_\alpha$, where $\{q_\alpha\}_{\alpha \in S}$ is the dual basis for $\gg$ with respect to $( \ \ | \ \  )$.

Let $f$ be a nilpotent element in the even part of $\gg$. By the Jacobson-Morozov theorem, $f$ can be completed to a copy of $\gs\gl_2$ triple $\{f, x, e\} \subseteq \gg$ satisfying
$$[x, e] =e, \qquad  [x, f]=-f,\qquad  [e, f]= 2x.$$ The $\cW$-(super)algebra $\cW^k(\gg,f)$ we are going to define depends only on the conjugacy class of $f$. First, we have the decomposition of $\gg$ as an $\gs\gl_2$-module:
$$ \gg = \bigoplus_{k \in  \frac{1}{2}\mathbb{Z}} \gg_k, \qquad \gg_k = \{  a \in \gg | [x, a] = ka \}.$$
Let $S_k$ be a basis of $\gg_k$, so that our index set $S$ decomposes as $S=\bigcup_k S_k$. Set
$$
\gg_+ = \bigoplus_{k \in  \frac{1}{2}\mathbb Z_{>0}} \gg_k, \qquad \gg_- = \bigoplus_{k \in  \frac{1}{2}\mathbb Z_{<0}} \gg_k,$$
with corresponding bases $S_+$ of $\gg_+$, and $\gg_-$ is naturally identified with the dual of $\gg_+$. 

We have an invariant bilinear form on $\gg_{\frac{1}{2}}$ given by
$$\langle a, b \rangle := ( f | [a, b] ).$$
Let $F(\gg_+)$ be the free field algebra associated to the vector superspace $\gg_+ \oplus \gg_+^*$. This vertex superalgebra is strongly generated by fields $\{\varphi_\alpha, \varphi^\alpha\}_{\alpha \in S_+}$, where $\varphi_\alpha$ and $\varphi^\alpha$ are odd if $\alpha$ is even and even if $\alpha$ is odd. The OPEs are
$$
\varphi_\alpha(z) \varphi^\beta(w) \sim \frac{\delta_{\alpha, \beta}}{(z-w)}, \qquad \varphi_\alpha(z) \varphi_\beta(w) \sim 0 \sim \varphi^\alpha(z) \varphi^\beta(w).$$

Let $F(\gg_{\frac{1}{2}})$ be the neutral vertex superalgebra associated to $\gg_{\frac{1}{2}}$ with bilinear form $\langle \ \ , \ \ \rangle$. This is strongly generated by $\{\Phi_\alpha\}_{\alpha \in S_{\frac{1}{2}}}$, where $\Phi_\alpha$ is even if $\alpha$ is even, and odd if $\alpha$ is odd. The OPEs are given by
\begin{equation}\label{eq:neutral}
\Phi_\alpha(z) \Phi_\beta(w) \sim \frac{\langle q^\alpha , q^\beta \rangle}{(z-w)}  \sim \frac{ ( f | [q^\alpha, q^\beta] )}{(z-w)},
\end{equation}
and fields corresponding to the dual basis with respect to $\langle \ \ , \ \ \rangle$ are denoted by $\Phi^\alpha$. 
The complex is defined by 
$$
C(\gg, f, k) := V^k(\gg) \otimes F(\gg_+) \otimes F(\gg_{\frac{1}{2}}).$$
One defines a $\mathbb Z$-grading by giving the $\varphi_\alpha$ charge $-1$, the $\varphi^\alpha$ charge $1$, and all others charge zero. 
One further defines the odd field $d(z)$ of charge minus one by
\begin{equation}
\begin{split}
d(z) &= \sum_{\alpha \in S_+} (-1)^{|\alpha|} X^\alpha \varphi^\alpha  
-\frac{1}{2} \sum_{\alpha, \beta, \gamma \in S_+}  (-1)^{|\alpha||\gamma|} f^{\alpha \beta}_\gamma \varphi_\gamma \varphi^\alpha \varphi^\beta + \\
&\quad  \sum_{\alpha \in S_+} (f | q^\alpha) \varphi^\alpha + \sum_{\alpha \in S_{\frac{1}{2}}} \varphi^\alpha \Phi_\alpha.
\end{split} 
\end{equation}
The zero-mode $d_0$ is a differential since $[d(z), d(w)]=0$ by \cite[Thm. 2.1]{KRW}. Set $m_\alpha=j$ if $\alpha \in S_j$. The $\mathcal W$-algebra is  defined to be its homology 
$$\cW^k(\gg, f) := H\left(C(\gg, f, k), d_0 \right).$$
Consider the Virasoro fields  
\begin{equation}
\begin{split}
L_{\text{sug}} &=  \frac{1}{2(k+h^\vee)} \sum_{\alpha \in S} (-1)^{|\alpha|} :X_\alpha X^\alpha:,\\
L_{\text{ch}} &= \sum_{\alpha \in S_+} \left(-m_\alpha :\varphi^\alpha \partial \varphi_\alpha: + (1-m_\alpha) :(\partial\varphi^\alpha )\varphi_\alpha:  \right),\\
L_{\text{ne}} &= \frac{1}{2} \sum_{\alpha \in S_{\frac{1}{2}}} :(\partial \Phi^\alpha) \Phi_\alpha : ,
\end{split}
\end{equation} and define
\begin{equation} L = L_{\text{sug}} + \partial x + 
L_{\text{ch}} + L_{\text{ne}}.\end{equation} 
Then $L$ represents an element of $\cW^k(\gg, f)$ and has central charge 
\begin{equation}\label{eq:c}
c(\gg, f, k) = \frac{k\, \text{sdim}\, \gg}{k+h^\vee} -12 k(x|x) -\sum_{\alpha \in S_+} (-1)^{|\alpha|} (12m_\alpha^2-12m_\alpha +2) -\frac{1}{2}\, \text{sdim}\, \gg_{\frac{1}{2}}.
\end{equation}
Set
\begin{equation}
J^\alpha = X^\alpha + \sum_{\beta, \gamma \in S_+} (-1)^{|\gamma|} f^{\alpha \beta }_\gamma :\varphi_\gamma\varphi^\beta:
\end{equation}
Their operator product is given in the equivalent formalism of $\lambda$-brackets as \cite[Eq. 2.5]{KWIII}
\[
[J^\alpha {}_\lambda J^\beta] = f^{\alpha \beta}_\gamma J^\gamma + \lambda\left(k(q^\alpha | q^\beta) + \frac{1}{2}\left(\kappa_\gg(q^\alpha, q^\beta) -\kappa_{\gg_0}(q^\alpha, q^\beta) \right) \right)
\]
with $\kappa_\gg$, $\kappa_{\gg_0}$ the Killing forms, that is the supertrace of the adjoint representations, of $\gg$, $\gg_0$, respectively.
The action of $d_0$ is \cite[Eq. 2.6]{KWIII}
\begin{equation}
\begin{split}
d_0(J^\alpha) &=  \sum_{\beta \in S_+} ([f, q^\alpha], q^\beta)\varphi^\beta + \sum_{\substack{ \beta \in S_+ \\ \gamma \in S_{\frac{1}{2}}}} (-1)^{|\alpha|(|\beta|+1)} f^{\alpha\beta}_\gamma \varphi^\beta \Phi_\gamma -\\
&  \sum_{\substack{\beta \in S_+ \\ \gamma \in S\setminus S_+}} (-1)^{|\beta|(|\alpha|+1)}f^{\alpha\beta}_\gamma \varphi^\beta J^\gamma  +
 \sum_{\beta \in S_+} ( k ( q^\alpha | q^\beta ) +  \text{str}_{\gg_+}\left( p_+ ( \text{ad}(q^\alpha)) \text{ad}(q^\beta) \right) \partial\varphi^\beta
\end{split}
\end{equation}
with $p_+$ the projection onto $\gg_+$ and $\text{str}_{\gg_+}$ the supertrace on $\gg_+$. Set 
\begin{equation}\label{eq:Ialpha}
I^\alpha := J^\alpha + \frac{(-1)^{|\alpha|}}{2} \sum_{\beta \in S_{\frac{1}{2}}}  f^{\beta\alpha}_\gamma \Phi^\beta\Phi_\gamma
\end{equation}
for $\alpha \in \gg_0$. Set $\ga := \gg^f \cap \gg_0$. It is a Lie subsuperalgebra of $\gg$. The next theorem tells us that $\cW^k(\gg, f)$  contains an affine vertex superalgebra of type $\ga$. 
\begin{thm} \cite[Thm 2.1]{KWIII}\label{thm:structureI}
\begin{enumerate}
\item $d_0(I^\alpha)=0$ for $q^\alpha \in \ga$ and
\[
[I^\alpha {}_\lambda I^\beta] = f^{\alpha \beta}_\gamma J^\gamma + \lambda\left(k(q^\alpha | q^\beta) + \frac{1}{2}\left(\kappa_\gg(q^\alpha, q^\beta) -\kappa_{\gg_0}(q^\alpha, q^\beta)-\kappa_{\frac{1}{2}}(q^\alpha, q^\beta) \right) \right),
\]
with $\kappa_{\frac{1}{2}}$ the supertrace of $\gg_0$ on $\gg_{\frac{1}{2}}$. 
\item
\[
 [ L {}_\lambda J^\alpha ]=(\partial +(1- j)\lambda )J^\alpha + \delta_{j, 0} \lambda^2\left(\frac{1}{2}{\rm str}_{\gg_+}({\rm ad}\ q^\alpha) -(k+h^\vee)( q^\alpha |x)\right),
 \]
 for $\alpha \in S_j$, and the same formula holds for $I^\alpha$ if $q^\alpha \in \ga$.
\end{enumerate}
\end{thm}
The main structural theorem is the following.

\begin{thm} \cite[Thm 4.1]{KWIII} \label{thm:Walgmain}
Let $\gg$ be a simple finite-dimensional Lie superalgebra with an invariant bilinear
form $( \ \ |  \ \ )$, and let $x, f$ be a pair of even elements of $\gg$ such that ${\rm ad}\ x$ is diagonalizable with
eigenvalues in $\frac{1}{2} \mathbb Z$ and $[x,f] = -f$. Suppose that all eigenvalues of ${\rm ad}\ x$ on $\gg^f$ are non-positive, so that we have a decomposition $\gg^f = \bigoplus_{j\leq 0} \gg^f_j$. 
 Then
 \begin{enumerate}
\item For each $q^\alpha \in \gg^f_{-j}$, ($j\geq  0$) there exists a $d_0$-closed field $K^\alpha$ of conformal weight
$1 + j$ (with respect to $L$) such that $K^\alpha - J^\alpha$ is a linear combination of normal ordered products of the fields $J^\beta$, where $\beta \in S_{-s}, 0 \leq s < j$, the fields $\Phi_\alpha$, where $\alpha \in S_{\frac{1}{2}}$, and the derivatives of these fields.
\item The homology classes of the fields $K^\alpha$, where $\{ q^\alpha\}_{\alpha \in S^f}$ is a basis of $\gg^f$ indexed by the set $S^f$ and  compatible with its $\frac{1}{2}\mathbb Z$-gradation, strongly and freely generate the VOA $\cW^k(\gg, f)$.
\item $H_0(C(\gg, f, k), d_0) = \cW^k(\gg, f)$ and $H_j(C(\gg, f, k), d_0) = 0$ if $j \neq 0$.
\end{enumerate}
\end{thm}

There are often several different ways to construct $\cW$-algebras in addition to the above definition. For example, the principal $\cW$-algebra can be realized inside the Heisenberg algebra of rank equal to $\text{rank} \ \gg$ as the kernel of screening operators associated to the simple roots of $\gg$. For simply laced types, there is also a coset construction $\cW^k(\gg)$ which was conjectured by many authors and recently proven in \cite{ACL}. 

 Let $\gg$ be a simple Lie algebra with dual Coxeter number $h^{\vee}$, $\cW^k(\gg)$ the universal principal $\cW$-algebra of $\gg$ at level $k$, and $\cW_k(\gg)$ its unique simple quotient. It was proven by Arakawa \cite{ArII,ArIV} that $\cW_k(\gg)$ is rational  when $k$ is a {\it nondegenerate admissible level}. These $\cW$-algebras are called the {\it minimal series principal $\cW$-algebras} since in the case that $\gg=\mathfrak{sl}_2$ they are exactly the minimal series Virasoro VOAs. They are not necessarily unitary, but if $\gg$ is simply laced, there exists a sub-series called the {\it discrete series} which were conjectured for many years to be unitary.

We denote by $X^{\xi}(z)$ the generating fields of $V^k(\gg)$ for $\xi \in\gg$, and by abuse of notation, we also denote by $X^{\xi}(z)$ the generating fields of $L_k(\gg)$. There exists a diagonal homomorphism $$V^{k+1}(\gg) \rightarrow V^k(\gg)\otimes L_1(\gg),\qquad X^{\xi}(z) \mapsto X^{\xi}(z) \otimes 1 + 1 \otimes X^{\xi}(z),\qquad \xi \in \gg,$$ which descends to a homomorphism $L_{k+1}(\gg) \rightarrow L_k(\gg)\otimes L_1(\gg)$ for all positive integer and admissible values of $k$. The main result of \cite{ACL} is the following.

\begin{thm}  \label{ACLmain}
Let $\gg$ be simply laced and let $k,\ell$ be complex numbers related by \begin{equation*} \ell +h^{\vee}=\frac{k+h^{\vee}}{k+h^{\vee}+1}.\end{equation*} \begin{enumerate}
\item For generic values of $\ell$, we have a VOA isomorphism 
$$\cW^{\ell}(\gg)\cong  \text{Com}(V^{k+1}(\gg),V^k(\gg)\otimes L_1(\gg)).$$ 
\item Suppose that $k$ is an admissible level for $\hat{\gg}$. Then $\ell$ is a nondegenerate admissible level for $\hat{\gg}$ so that $\cW_{\ell}(\gg)$ is a minimal series $\cW$-algebra. Then $$\cW_{\ell}(\gg)\cong  \text{Com}(L_{k+1}(\gg),L_k(\gg)\otimes L_1(\gg)).$$
\end{enumerate}
\end{thm}

This was conjectured in \cite{BBSS} in the case of discrete series, which correspond to the case $k \in \mathbb{N}$, and by Kac and Wakimoto \cite{KWI,KWIII} for arbitrary minimal series $\cW$-algebras. The conjectural character formula of \cite{FKW} for minimal series representations of $\cW$-algebras that was proved in \cite{ArI}, together with the character formula of \cite{KWIII} of branching rules, proves the matching of characters. This provided strong evidence for this conjecture, although it was previously known only for $\mathfrak{sl}_2$ \cite{GKO,KWIII} for a general admissible $k$, and for $\mathfrak{sl}_3$ for positive integers $k$ \cite{AJ}.  

Previous works on this problem have focused on the case where $k$ is either a positive integer or an admissible level, and it is difficult to study all such cases in a uniform manner. In \cite{ACL}, the case of {\it generic level} was considered first; by Theorem 8.1 of \cite{CLIII}, once statement (1) is proven for generic level, statement (2) then follows. For generic level, we have the free field realization of $\cW^{\ell}(\gg)$ coming from the Miura map $\gamma_{\ell}: \cW^{\ell}(\gg) \hookrightarrow \pi$, where $\pi$ is the Heisenberg VOA of rank $d = \text{rank}(\gg)$. This is obtained by applying the Drinfeld-Sokolov reduction functor to the Wakimoto free field realization $V^{\ell}(\gg) \hookrightarrow M_{\gg} \otimes \pi$, where $M_{\gg}$ is the $\beta\gamma$-system of rank $d$ \cite{FrI}. The technical heart of \cite{ACL} is to construct {\it another} VOA homomorphism $\Psi_k: \text{Com}(V^{k+1}(\gg),V^k(\gg)\otimes L_1(\gg)) \hookrightarrow \pi$, and show that its image coincides with the image of $\gamma_{\ell}$.

Theorem \ref{ACLmain} is a starting assumption of the conformal field theory to higher spin gravity correspondence of \cite{GGI}, and has several applications. First, it implies the unitarity of all discrete series principal $\cW$-algebras. Second, it implies the rationality of several families of cosets, $\text{Com}(L_n(\mathfrak{gl}_m), L_n(\mathfrak{sl}_{m+1}))$ and $\text{Com}(L_{2n}(\mathfrak{so}_m), L_{2n}(\mathfrak{so}_{m+1}))$ for positive integers $n,m$, and a family of $\cN=2$ superconformal VOAs called Kazama-Suzuki cosets \cite{KS}. Third, if $\cV$ is a VOA, $\cA \subseteq \cV$ a subVOA, and $\cC = \text{Com}(\cA, \cV)$, there is a strong link between the representation theories of $\cA$, $\cC$ and  $\cV$. For $k$ an admissible level, Theorem \ref{ACLmain} together with results of \cite{CKM}, implies that properties of the category of ordinary modules for the affine VOA $L_k(\gg)$ such as rigidity, are inherited from known properties of the module category of the corresponding $\cW$-algebra. In \cite{CHY,C}, a certain subcategory of the Bernstein-Gelfand-Gelfand category $\cO$ for $\widehat{\gg}$ was shown to have a vertex tensor category structure, which is modular under some arithmetic conditions on $k$. This gives the first examples of such modular tensor categories for VOAs that are not $C_2$-cofinite.

Given the many applications of the coset construction of $\cW^k(\gg)$ for simply-laced $\gg$, is an important problem to find coset realizations of $\cW^k(\gg)$ for other Lie types, as well as for principal $\cW$-superalgebras and non-principal $\cW$-algebras. One of the applications of the triality theorems is to give a new coset realization of $\cW^k(\gg)$ when $\gg$ is of type $B$ or $C$. Also, another perspective on triality was given by Conjecture 1.1 of \cite{CLIV} and Conjecture 1.2 of \cite{CLV}. Here, one conjectures the existence of a simple vertex superalgebra which is a large extension of a certain affine VOA tensored with one of the $\cW$-(super)algebras appearing in the triality theorems. The existence of these structures would yield coset realizations of all the $\cW$-superalgebras considered in \cite{CLIV,CLV}, vastly generalizing the coset realization in types $A$ and $D$ given by Theorem \ref{ACLmain}.

\section{Topological vertex operator algebras}\label{sec:TVOA}

A topological vertex operator algebra (TVOA) consists of the following data:
a vertex operator superalgebra $C^*$, 
a weight one even current $F(z)$ whose charge $F_0$ is the fermion number operator,
a weight one odd primary field $J(z)$ having fermion number one and 
having a square zero charge $J_0=Q$, 
and a weight two odd primary field $G(z)$ 
having fermion number -1  and satisfying $[Q,G(z)]=L(z)$
where $L(z)$ is the stress-energy field (Virasoro element).
We denote the cohomology of the complex $(C^*,Q)$ by $H^*(C)$.

Examples of TVOAs have arisen in physics, primarily from twisting two-dimensional $\cN=2$ SCFT (see for example \cite{DVV}\cite{EguYa}).
Another prototype example of a TVOA is given by the $\cN=0$ absolute BRST complex $C^*_{BRST}=V\otimes \Lambda^*$, 
with coefficient in a VOA $V$ of central charge of 26. In this case,
$$
C^*\equiv(C^*,F,J,G)=(C^*_{BRST},:c(z)b(z):,L^V(z)+\half L^\Lambda(z))c(z): + \frac{3}{2}\partial^2 c(z),b(z)).$$
Here $L^V,L^\Lambda$ are the Virasoro elements of $V$ and the $bc$-system of weights $(2,-1)$ respectively.
This was the main object of study in \cite{LZII}, in which the following general theorem was proved.

\begin{thm}
On the cohomology $H^*(C)$ of a TVOA, the Wick product on $C^*$ descends to a supercommutative associative product. Moreover,
the Fourier coefficient $G_0$ descends a BV operator on cohomology. That is, the bracket measuring the failure of $G_0$ to be a derivation:
$$\{u,v\}:=G_0(u\cdot v)-(G_0u)\cdot v-(-1)^{|u|} u\cdot (G_0v)$$
defines a Gerstenhaber bracket or an odd Poisson bracket on cohomology $\{,\}: H^p\times H^q\longrightarrow H^{p+q-1}$. More precisely, we have
\begin{itemize}
\item $\{u,v\}=-(-1)^{(|u|-1)(|v|-1)}\{v,u\}$
\item $(-1)^{(|u|-1)(|t|-1)}\{u,\{v,t\}\} +(-1)^{(|t|-1)(|v|-1)}\{t,\{u,v\}\} +(-1)^{(|v|-1)(|u|-1)}\{v,\{t,u\}\}=0$
\item $\{u,v\cdot t\}=\{u,v\}\cdot t +(-1)^{(|u|-1)|v|}v\cdot\{u,t\}$
\item $G_0\{u,v\}=\{G_0u,v\}+(-1)^{|u|-1}\{u,G_0v\}$.
\end{itemize}
In particular $H^0$ is a commutative algebra, and $H^1$ is a Lie algebra with Lie bracket given by the Gerstenhaber bracket.
\end{thm}

We note that if a TVOA arises as a twisted $\cN=2$ SCFT, then it can be shown that the BV operator on cohomology is identically zero. This is a special feature of
$\cN=2$ theories. (See \cite{LZII}.)

Another important new example of a TVOA is the absolute BRST complex of the $\cN=1$ Virasoro superalgebra. This was studied in an unpublished work of the first author with Moore and Zuckerman
in \cite{LMZ92}.

Finally, as an application we mention one more interesting special case of the preceding theorem, which gives a new cohomological construction of Borcherd's Monster Lie algebra. Let $L$ be the rank 2 unimodular lattice, and $V_L$ the corresponding free field VOA with momentum lattice $L$.
We choose the standard Virasoro element of central charge 2 on $V_L$ as our conformal structure. Let $V^\natural$ be the Moonshine VOA of FLM. The TVOA given by the absolute BRST complex of the $\cN=0$ Virasoro algebra with coefficient in $V^\natural\otimes V_L$ was studied in details in \cite{LZIII}. A vanishing theorem for the BV algebra $H^*(C)$ was proved. In particular that $H^1(C^*)$ in this case was shown to be precisely Borcherds' Monster Lie algebra.

\section{Zhu's commutative algebra, Li's filtration, and associated graded algebras} \label{sec:zhuli} 

Given a vertex algebra $\cV$, define \begin{equation} \label{def:zhucomm} C(\cV) = \text{Span}\{a_{(-2)} b|\ a,b\in \cV \},\qquad R_{\cV} = \cV / C(\cV).\end{equation} It is well known that $R_{\cV}$ is a commutative, associative algebra with product induced by the normally ordered product \cite{Z}. Also, if $\cV$ is graded by conformal weight, $R_{\cV}$ inherits this grading. The following notions are due to Arakawa \cite{ArII,ArIII}; see also \cite{AMI}.
\begin{enumerate} 
\item The {\it associated scheme} of $\cV$ is  $\tilde{X}_{\cV} = \text{Spec} \ R_{\cV}$,
\item The {\it associated variety} of $\cV$ is $X_{\cV} = \text{Specm} \ R_{\cV} = (\tilde{X}_{\cV})_{\text{red}}$.
\end{enumerate}
Here $(\tilde{X}_{\cV})_{\text{red}}$ denotes the reduced scheme of $\tilde{X}_{\cV}$. If $\{\alpha_i|\ i\in I\}$ is a strong generating set for $\cV$, the images of these fields in $R_{\cV}$ will generate $R_{\cV}$ as a ring. In particular, $R_{\cV}$ is finitely generated if and only if $\cV$ is strongly finitely generated.

A vertex algebra $\cV$ is called $C_2$-cofinite, or lisse, if $R_{\cV}$ is finite-dimensional as a vector space over $\mathbb{C}$. If $\cV$ is strongly finitely generated, this is equivalent to the condition that every element of $R_{\cV}$ is nilpotent. In this case, $\tilde{X}_{\cV}$ is a finite set and $X_{\cV}$ consists of just one point. This finiteness condition was introduced by Zhu \cite{Z} and plays a fundamental notion in his proof of modular invariance of characters of rational VOAs which satisfy this condition.

If $A$ is a commutative $\mathbb{C}$-algebra, the ring $A_{\infty}$ of functions on the arc space of $\text{Spec}\ A$ can be thought of as an abelian vertex algebra. Then Zhu's commutative algebra $R_{A_{\infty}}$ is isomorphic to $A$ \cite{FBZ}. In particular, {\it any} commutative $\mathbb{C}$-algebra arises as $R_{\cV}$ for some vertex algebra $\cV$. On the other hand, if we consider $R_{\cV}$ for simple vertex algebras $\cV$, it is not at all apparent which commutative rings can be obtained in this way. For example, no simple VOA $\cV$ is known where $X_{\cV}$ fails to be irreducible. In fact, if $\cV$ is simple and quasi-lisse it is known that $X_{\cV}$ is always irreducible \cite{AMII}, and this is expected to hold in a more general setting.

One case where the structure of $X_{\cV}$ has received considerable attention is when $\cV$ is a simple affine VOA $L_k(\gg)$ for an admissible level $k$. In this case an important theorem of Arakawa states that $X_{\cV}$ lies in the nilpotent cone, and this is a key starting point in his result that $L_k(\gg)$ is rational in category $\cO$ \cite{ArV}.

The ring $R_{\cV}$ turns out to be part of a much larger commutative algebra which we now recall. For any vertex algebra $\cV$, Haisheng Li has defined a canonical decreasing filtration on $\cV$ in \cite{LiII} as follows.
$$F^0(\cV) \supseteq F^1(\cV) \supseteq \cdots,$$ where $F^p(\cV)$ is spanned by elements of the form
$$:\partial^{n_1} a^1 \partial^{n_2} a^2 \cdots \partial^{n_r} a^r:,$$ 
where $a^1,\dots, a^r \in \cV$, $n_i \geq 0$, and $n_1 + \cdots + n_r \geq p$. Note that $\cV = F^0(\cV)$ and $\partial F^i(\cV) \subseteq F^{i+1}(\cV)$. Set $$\text{gr}(\cV) = \bigoplus_{p\geq 0} F^p(\cV) / F^{p+1}(\cV),$$ and for $p\geq 0$ let 
$$\sigma_p: F^p(\cV) \ra F^p(\cV) / F^{p+1}(\cV) \subseteq \text{gr}(\cV)$$ be the projection. Note that $\text{gr}(\cV)$ is a graded commutative algebra with product
$$\sigma_p(a) \sigma_q(b) = \sigma_{p+q}(a_{(-1)} b),$$ for $a \in F^p(\cV)$ and $b \in F^q(\cV)$. We say that the subspace $F^p(\cV) / F^{p+1}(\cV)$ has degree $p$. Note that $\text{gr}(\cV)$ has a differential $\partial$ defined by $$\partial( \sigma_p(a) ) = \sigma_{p+1} (\partial a),$$ for $a \in F^p(\cV)$. Finally, $\text{gr}(\cV)$ has the structure of a Poisson vertex algebra \cite{LiII}; for $n\geq 0$, we define $$\sigma_p(a)_{(n)} \sigma_q(b) = \sigma_{p+q-n} a_{(n)} b.$$ Zhu's commutative algebra $R_{\cV}$ is isomorphic to the subalgebra $F^0(\cV) / F^1(\cV)\subseteq \text{gr}(\cV)$, since $F^1(\cV)$ coincides with the space $C(\cV)$ defined by \eqref{def:zhucomm}. Moreover, $\text{gr}(\cV)$ is generated by $R_{\cV}$ as a differential graded commutative algebra \cite{LiII}.

Since $\text{gr}(\cV)$ is a differential graded algebra containing $R_{\cV}$ as the weight zero component and is generated by $R_{\cV}$, by the universal property of $\mathbb{C}[(\tilde{X}_{\cV})_{\infty}]$, we have a surjective homomorphism of differential graded rings
\begin{equation} \label{dgalgebras} \Phi_{\cV}: \mathbb{C}[(\tilde{X}_{\cV})_{\infty}] \ra \text{gr}(\cV).\end{equation} Following Arakawa \cite{ArII, ArIII}, we define the {\it singular support} of $\cV$ to be 
\begin{equation} \label{def:singsupp} \text{SS}(\cV)=\text{Spec}\ \text{gr}(\cV),\end{equation} which is then a closed subscheme of $(\tilde{X}_{\cV})_{\infty}$. A natural question which was raised by Arakawa and Moreau \cite{AMI} is when the map \eqref{dgalgebras} is an isomorphism. A vertex algebra $\cV$ for which this map is an isomorphism of schemes is called {\it classically free} by van Ekeren and Heluani \cite{EH}. This property has turned out to be important in their notion of chiral homology, and they have proven that Virasoro minimal models $\text{Vir}_{p,q}$ are classically free if and only of $p = 2$. Here $\text{Vir}_{p,q}$ denotes the simple Virasoro VOA with central charge $c = 1 - 6 \frac{(p-q)^2}{pq}$.  More examples and counterexamples of classically free VOAs have appeared in work of Andrews, van Ekeren, Heluani, Jennings-Shaffer, Li and Milas, and their characters have interesting connections with combinatorial identities \cite{L,LM,MJS,AEH}.

A weaker condition than classical freeness is that the map \eqref{dgalgebras} is an isomorphism at the level of varieties, i.e., the induced map on reduced rings is an isomorphism. The VOAs in the above papers are all $C_2$-cofinite, so the associated variety consists of just a point and this condition holds automatically since the reduced rings are both isomorphic to $\mathbb{C}$. Arakawa and Moreau have shown that \eqref{dgalgebras} is an isomorphism at the level of varieties in a much richer class of examples, namely quasi-lisse VOAs \cite{AMII}. This notion is due to Arakawa and Kawasetsu \cite{AK}, and means that the associated variety is a Poisson variety with finitely many symplectic leaves. Under these conditions, the category of ordinary modules contains finitely many irreducible objects whose normalized characters satisfy a modular linear differential equation \cite{AK}. Quasi-lisse VOAs form a rich class that includes simple affine VOAs at admissible levels, and also more exotic examples coming from four-dimensional $\cN=2$ superconformal field theories \cite{BLLPRR}. 

In \cite{AL}, Arakawa and the second author gave several examples of VOAs that are not quasi-lisse and not classically free, where \eqref{dgalgebras} is an isomorphism at the level of varieties. At the moment, there are no known examples of simple VOAs where \eqref{dgalgebras} fails to be an isomorphism at the level of varieties, and it is reasonable to expect that this holds for all simple VOAs. One of the examples in \cite{AL} involves the simple $\cW_3$-algebra at central charge $c = -2$, which is the $\cW_{2,3}$ singlet algebra \cite{AdM}. In this case, the associated variety is the cuspidal curve given by $x^2 = y^3$. 

We note that if a VOA $\cV$ is not classically free, it is a difficult problem to describe $\text{gr}(\cV)$; even though it is generated by $R_{\cV}$ as a DGA, the description of the relations is nontrivial. The kernel of the map \eqref{dgalgebras} is always a differential ideal, that is, an ideal which is closed under the derivation $D$, and a natural question raised by Arakawa and the second author in \cite{AL} is whether it is finitely generated as a differential ideal. In the above example $\cV = \cW_{2,3}$, where $X_{\cV}$ is the cuspidal curve, this kernel coincides with the nilradical of $\mathbb{C}[(\tilde{X}_{\cV})_{\infty}]$, and it is conjectured to be generated by two elements as a differential ideal; see \cite{AL} as well as \cite{KnSe}. More recently, in the case of the Virasoro minimal model $Vir_{3,4}$, which is not classically free, the kernel of \eqref{dgalgebras} was proven to be a differentially finitely generated ideal by Andrews, Heluani and van Ekeren in \cite{AEH}.

In addition to the canonical decreasing filtration that exists on any VOA $\cV$, Li has also introduced the notion of a {\it good increasing filtration} that exists under some mild conditions on $\cV$ \cite{LiII}. If $\cV$ possesses a good increasing filtration, the associated graded algebra with respect to this filtration is also a graded commutative ring, and typically coincides with the associated graded algebra with respect to the Li's decreasing filtration. In many of our applications, it is more convenient to work with good increasing filtrations, and we will use the same notation $\text{gr}(\cV)$ since in all of our examples this structure is independent of which filtration we use.

The passage from a VOA $\cV$ to its associated graded algebra $\text{gr}(\cV)$ is a powerful tool for studying $\cV$. In our applications to the structure of orbifolds and cosets of $\cV$ using classical invariant theory, this is effective when $\cV$ is either a free field algebra, a universal affine VOA, or a universal $\cW$-algebra. In these examples, $\cV$ is freely generated, hence automatically classically free. In order to apply these ideas to a wider class of VOAs, it will be necessary to gain a better understanding of the structure of $\text{gr}(\cV)$, and hence of the kernel of \eqref{dgalgebras}.

\section{Vertex algebras over commutative rings} \label{sec:voacommring} 

Let $R$ be a commutative $\mathbb{C}$-algebra. A vertex algebra over $R$ is an $R$-module $\cA$ with a vertex algebra structure which we define in the same way as before, except that all linear maps are replaced with $R$-module homomorphisms. For any $R$-module $M$, we define $\text{QO}_R(M)$ to be the set of $R$-module homomorphisms $a: M \ra M((z))$, which is itself naturally an $R$-module. An element $a \in \text{QO}_R(M)$ can be uniquely represented by a power series $$a(z) = \sum_{n\in \mathbb{Z}} a(n) z^{-n-1} \in \text{End}_R(M)[[z,z^{-1}]].$$ 
Here $a(n) \in \text{End}_R(M)$ is an $R$-module endomorphism, and for each element $v\in M$, $a(n) v = 0$ for $n>\!\!>0$. 
We define the products $a_{(n)} b$ as before. They are $R$-module homomorphisms 
$$\text{QO}_R(M) \otimes_R \text{QO}_R(M) \ra \text{QO}_R(M).$$ A QOA will be an $R$-module $\cA \subseteq \text{QO}_R(M)$ containing $1$ and closed under all the above products. Locality is defined in the same way as before, and a vertex algebra over $R$ is a QOA $\cA\subseteq \text{QO}_R(M)$ whose elements are pairwise local. We mention that a comprehensive theory of vertex algebras over commutative rings has recently been developed by Mason \cite{Ma}, but the main difficulties are not present when $R$ is a $\mathbb{C}$-algebra. In particular, the translation operator $\partial$ can be defined in the usual way and it is not necessary to replace it with a Hasse-Schmidt derivation.

We say that a subset $S = \{\alpha^i|\ \ i\in I\} \subseteq \cA$ generates $\cA$ if $\cA$ is spanned as an $R$-module by all words in $\alpha^i$ and the above products. Similarly, $S$ strongly generates $\cA$ if $\cA$ is spanned as an $R$-module by all iterated Wick products of these generators and their derivatives. If $S = \{\alpha^1, \alpha^2,\dots\}$ is an ordered strong generating set for $\cA$ which is at most countable, we say that $S$ {\it freely generates} $\cA$, if $\cA$ has an $R$-basis consisting of all normally ordered monomials of the form \eqref{freegenerators}. In particular, this implies that $\cA$ is a free $R$-module.

Let $\cV$ be a vertex algebra over $R$ and let $c\in R$. Suppose that $\cV$ contains a field $L$ satisfying the Virasoro OPE relation \eqref{virope}, such that $L_0$ acts on $\cV$ by $\partial$ and $L_1$ acts diagonalizably, and we have an $R$-module decomposition $$\cV = \bigoplus_{d\in R} \cV[d],$$ where $\cV[d]$ is the $L_0$-eigenspace with eigenvalue $d$. We then call $\cV$ as VOA over $R$. In all our examples, the grading will be by $\mathbb{Z}_{\geq 0}$ regarded as a subsemigroup of $R$, and $\cV[0] \cong R$.

Suppose that $R$ is the ring of functions on a variety $X \subseteq \mathbb{C}^n$, so that $X = \text{Specm}\ R$. We can regard the vertex algebra $\cV$ as being defined on $X$ in the sense that for each point $p\in X$, the evaluation at $p$ yields a vertex algebra over $\mathbb{C}$. Similarly, we have the notion of specialization along a subvariety. Let $I\subseteq R$ be an ideal corresponding to a closed subvariety $Y \subseteq X$, and let $I \cdot \cV$ denote the set of finite sums of the form $\sum_i f_i v_i$ where $f_i \in I$ and $v_i \in \cV$. Clearly $I \cdot \cV$ is the vertex algebra ideal generated by $I$, and the quotient
$$\cV^I = \cV / (I \cdot \cV)$$ is a vertex algebra over $R/I$.

Let $\cV$ be a vertex algebra over $R$ with weight grading 
\begin{equation} \label{eq:gradedrvoa} \cV = \bigoplus_{n\geq 0} \cV[n],\qquad \cV[0] \cong R.\end{equation}
A vertex algebra ideal $\cI \subseteq \cV$ is called {\it graded} if $$\cI = \bigoplus_{n\geq 0} \cI[n],\qquad \cI[n] = \cI \cap \cV[n].$$ 
We say that $\cV$ is {\it simple} if for every proper graded ideal $\mathcal{I}\subseteq \mathcal{V}$, $\mathcal{I}[0] \neq \{0\}$. This is different from the usual notion of simplicity, namely, that $\cV$ has no nontrivial proper graded ideals, but it agrees with the usual notion if $R$ is a field. If $I \subseteq R$ is a nontrivial proper ideal, it will generate a nontrivial proper graded vertex algebra ideal $I\cdot \cV$. Then $$\mathcal{V}^I = \mathcal{V} / (I\cdot \mathcal{V})$$ is a vertex algebra over $R/I$. Even if $\mathcal{V}$ is simple over $R$, $\mathcal{V}^I$ need not be simple over $R/I$.

Let $R$ be the ring of functions on a variety $X$, and let $\cV$ be a simple vertex algebra over $R$ with weight grading \eqref{eq:gradedrvoa}. Let $I\subseteq R$ be an ideal such that $\cV^I$ is {\it not} simple. This means that $\cV^I$ possesses a maximal proper graded ideal $\cI$ such that $\cI[0] = \{0\}$. Then the quotient $$\cV_{I} = \cV^I / \cI$$ is a simple vertex algebra over $R/I$. Letting $Y\subseteq X$ be the closed subvariety corresponding to $I$, we can regard $\cV_I$ as a simple vertex algebra defined over $Y$. Continuing this process, we can consider the poset all ideals $I\subseteq R$ for which $\cV^I$ is not simple over $R/I$. If $I_1, I_2$ are ideals in this poset, let $\cV_{I_1} = \cV^{I_1} / \cI_1$ and $\cV_{I_2} = \cV^{I_2} / \cI_2$ be the corresponding simple vertex algebras over $R/I_1$ and $R / I_2$, respectively. Let $Y_1, Y_2 \subseteq X$ be the closed subvarieties corresponding to $I_1, I_2$, and let $p \in Y_1 \cap Y_2$ be a point in the intersection. Let $\cV^p_{I_1}$ and $\cV^p_{I_2}$ be the vertex algebras over $\mathbb{C}$ obtained by evaluating at $p$. Then $\mathcal{V}^p_{I_1}$ and $ \mathcal{V}^p_{I_2}$ need not be simple, and we let $\cV_{I_1,p}$ and $\cV_{I_2,p}$ denote their simple quotients. Clearly $p$ corresponds to a maximal ideal $M \subseteq R$ containing both $I_1$ and $I_2$, and we have isomorphisms
$$\cV_{I_1,p} \cong \cV_M \cong \cV_{I_2,p}.$$ 
In general, $\cV_{I_1}$ and $\cV_{I_2}$ can be very different vertex algebras, and the above pointwise isomorphism which arises from the intersection of the varieties $Y_1$ and $Y_2$, is nontrivial.

Given a VOA $\cV = \bigoplus_{n\geq 0} \cV[n]$ over $R$ with $\cV[0] \cong R$, $\cV[n]$ has a symmetric bilinear form \begin{equation} \label{bilinearform} \langle,\rangle_n: \cV[n]\otimes_{R} \cV[n] \ra R,\qquad \langle \omega,\nu \rangle_n = \omega_{(2n-1)}\nu.\end{equation}
If each $\cV[n]$ is a free $R$-module of finite rank, we define the level $n$ Shapovalov determinant $\text{det}_n \in R$ to be the determinant of the matrix of $\langle,\rangle_n$. Under mild hypotheses, namely, $\cV[0] \cong R$, and $\cV[1]$ is annihilated by $L_1$, an element $\omega \in \cV[n]$ lies in the radical of the form $\langle,\rangle_n$ if and only of $\omega$ lies in the maximal proper graded ideal of $\cV$ \cite{LiI}. Then $\cV$ is simple if and only if $\text{det}_n \neq 0$ for all $n$. If $\cV$ is simple and $R$ is a unique factorization ring, each irreducible factor $a$ of $\text{det}_n$ give rise to a prime ideal $(a) \subseteq R$. The corresponding varieties $V(a)$ are just the irreducible subvarieties of $\text{Specm}\ R$ where degeneration of $\cV$ occurs.

\section{Orbifolds and the vertex algebra Hilbert problem} \label{sec:hilbert} 
\subsection{Orbifold construction} Let $\cV$ be a VOA and let $G$ be a group of automorphisms of $\cV$. Unless otherwise stated, we will always assume that $\cV$ has a conformal vector $L$ which is invariant under $G$. The invariant subalgebra $\cV^G$ is called an {\it orbifold} of $\cV$, and it has the same conformal vector, so the inclusion $\cV^G\hookrightarrow \cV$ is a conformal embedding. This construction was introduced in physics; see for example \cite{DVVV,DHVWI,DHVWII}, as well as \cite{FLM} for the construction of the Moonshine VOA  $V^{\natural}$ as an extension of the $ \mathbb{Z}/2\mathbb{Z}$-orbifold of the VOA associated to the Leech lattice.

One motivation for studying orbifolds is their connection with the longstanding problem of classifying rational conformal field theories. When $c<1$ this has been achieved using the classification of unitary representations of the Virasoro algebra \cite{KL}, but for $c\geq 1$ it is still out of reach. Two cases have received much attention and appear tractable: the case $c=1$ \cite{G}\cite{Ki}, and the case of holomorphic VOAs, that is, rational VOAs $\cV$ whose module category contains only one irreducible object, namely $\cV$ itself. For $c=1$, it is conjectured that in addition to the rank one lattice VOAs and their $\mathbb{Z}/2\mathbb{Z}$-orbifolds, the remaining cases are $V_L^G$ where $L$ is the $A_1$ root lattice and $G$ is one of the alternating groups $A_4$ or $A_5$, or the symmetric group $S_4$ \cite{DJ}. If $\cV$ is holomorphic, its central charge $c$ is necessarily an integer multiple of $8$ \cite{Sch}. For $c = 8$ and $c = 16$, the classification is easy and appears in \cite{DM}, but the case $c = 24$ is a famous problem. In \cite{Sch}, Schellekens showed that there are exactly $71$ possible characters of such VOAs, and conjectured that there is a unique holomorphic VOA with $c=24$ realizing each of these. Very recently, due to the efforts of many researchers, these have been constructed using orbifold theory, and their uniqueness is known with the exception of the Moonshine VOA $V^{\natural}$ \cite{EMSI,EMSII,ELMS,LS}.

Recall that a vertex algebra $\cV$ is called {\it strongly finitely generated} if there exists a finite set of generators such that the set of iterated Wick products of the generators and their derivatives spans $\cV$. This property has many important consequences, and in particular implies that Zhu's associative algebra $A(\cV)$ and Zhu's commutative algebra $R_{\cV}$ are finitely generated. Recall Hilbert's theorem that if a reductive group $G$ acts on a finite-dimensional complex vector space $V$, the invariant ring $\mathbb{C}[V]^G$ is finitely generated \cite{HI,HII}. This theorem was very influential in the development of commutative algebra and algebraic geometry. In fact, Hilbert's basis theorem, Nullstellensatz, and syzygy theorem were all introduced in connection with this problem. One can ask similar questions in the setting of noncommutative rings. There are many example such as universal enveloping algebras, Weyl algebras, etc., which admit filtrations for which the associated graded algebra is commutative, and the problem can be reduced to the commutative case. The analogous problem for vertex algebras is the following. 

\begin{prob} Given a simple, strongly finitely generated vertex algebra $\cV$ and a reductive group $G$ of automorphisms of $\cV$, is $\cV^G$ strongly finitely generated?
\end{prob}

This is much more subtle than the case of noncommutative rings, and generally fails for abelian vertex algebras. The main difficulty is that vertex algebras are not Noetherian, and this phenomenon depends sensitively on the nonassociativity of vertex algebras. 

\subsection{Example: the rank one Heisenberg algebra} Before we discuss this problem in a general setting, we want to give an illustrative example. The rank one Heisenberg algebra $\cH = \cH(1)$ is freely generated by one field $\alpha$ satisfying $\alpha(z) \alpha(w) \sim (z-w)^{-2}$. With respect to the usual Virasoro element $L = \frac{1}{2} :\alpha\alpha:$, the generator $\alpha$ is primary of weight $1$. There is only one nontrivial automorphism of $\cH$, namely, the involution $\theta$ defined by $\theta(\alpha) = -\alpha$.

We also wish to consider the {\it degenerate} Heisenberg algebra $\cH_{\text{deg}}$ which is also generated by a field $\alpha$ of weight $1$, but this field has regular OPE with itself. As vector spaces, $$\cH_{\text{deg}} \cong \mathbb{C}[\alpha, \partial \alpha, \partial^2\alpha, \dots] \cong \text{gr}(\cH) \cong \cH,$$ but $\cH_{\text{deg}}$ is just a differential graded commutative algebra, that is, $\cH_{\text{deg}} \cong \mathbb{C}[\alpha, \partial \alpha, \partial^2\alpha, \dots]$ as vertex algebras. We may also define an involution $\theta$ on $\cH_{\text{deg}}$ by $\theta(\partial^k\alpha) = -\partial^k \alpha$. The invariant algebras $\cH^{ \mathbb{Z}/2\mathbb{Z}}$ and $\cH_{\text{deg}}^{ \mathbb{Z}/2\mathbb{Z}}$ are linearly isomorphic, but they have very different behavior as vertex algebras.

We begin with $\cH_{\text{deg}}^{ \mathbb{Z}/2\mathbb{Z}}$. First, since $\theta$ acts diagonalizably on the generators, $\theta$ acts by $\pm 1$ on each monomial $\partial^{i_1} \alpha \cdots \partial^{i_r} \alpha \in \cH_{\text{deg}}$, and such a monomial is invariant under $\theta$ if and only if $r$ is even. Therefore $\cH_{\text{deg}}^{ \mathbb{Z}/2\mathbb{Z}}$ is generated by all quadratic monomials 
$$q_{i,j} = \partial^i \alpha \partial^j \alpha,\qquad i,j \geq 0.$$
 Note that $q_{i,j} = q_{j,i}$ so it suffices to take $q_{i,j}$ with $i \geq j \geq 0$. There is some redundancy in this generating set from the point of view of differential algebras because $\partial (q_{i,j}) = q_{i+1,j} + q_{i,j+1}$. It is easy to see that the sets 
$$\{\partial^j q_{2i,0}|\ i,j \geq 0\},\qquad \{q_{i,j}|\ 0 \leq i \leq j\},$$ are linearly independent and span the same vector space. Therefore 
$$\{q_{2i,0}|\ i \geq 0\}$$ generates $\cH_{\text{deg}}^{ \mathbb{Z}/2\mathbb{Z}}$ as a differential algebra; equivalently it strongly generates $\cH_{\text{deg}}^{ \mathbb{Z}/2\mathbb{Z}}$. In fact, we claim that this is a {\it minimal} generating set, that is, if we remove any of these generators the remaining elements will no longer generate. Although $\{\partial^j q_{2i,0}|\ i,j \geq 0\}$ is a linearly independent set, there are many {\it algebraic} relations among these elements. In terms of our original generating set $\{q_{i,j}|\ 0\leq i \leq j\}$, all such relations are consequences of the quadratic relations
$$q_{i,j} q_{k,\ell} - q_{i,\ell} q_{k,j} = 0.$$ Since these are homogeneous of degree two, there are no relations among $\{\partial^j q_{2i,0}|\ i,j \geq 0\}$ which have a linear term. This shows that our generating set is indeed minimal. In particular, we conclude that $\cH_{\text{deg}}^{ \mathbb{Z}/2\mathbb{Z}}$ is {\it not finitely generated as a differential algebra}. More generally, if $G$ is any finite group and $V$ is any nontrivial finite-dimensional $G$-module, $G$ then acts on the ring of functions $\mathbb{C}[V]$ and on the differential polynomial ring $\mathbb{C}[V_{\infty}]$, and $\mathbb{C}[V_{\infty}]^G$ is never finitely generated as a differential algebra \cite{LSS}.

Now we turn to the orbifold $\cH^{ \mathbb{Z}/2\mathbb{Z}}$ of the nondegenerate Heisenberg algebra, where the generator $\theta$ of $ \mathbb{Z}/2\mathbb{Z}$ acts by $\theta(\alpha) = -\alpha$ as above. The following well-known theorem is due to Dong and Nagatomo \cite{DN} and is the starting point for understanding the structure of $\mathbb{Z}/2\mathbb{Z}$-orbifolds of lattice VOAs.

\begin{thm} \label{HZ2} The $ \mathbb{Z}/2\mathbb{Z}$-orbifold of the rank one Heisenberg VOA $\cH^{ \mathbb{Z}/2\mathbb{Z}}$ is of type $\cW(2,4)$. In particular, it is strongly generated by the Virasoro field $L = \frac{1}{2} :\alpha \alpha:$ and a weight $4$ field $W$ which we can take to be $:(\partial^2 \alpha) \alpha:$.
\end{thm}

Even though $\cH^{ \mathbb{Z}/2\mathbb{Z}}$ and $\cH_{\text{deg}}^{ \mathbb{Z}/2\mathbb{Z}}$ are isomorphic as vector spaces, the former is strongly finitely generated and the latter is not. This phenomenon is very subtle and is a consequence of the nonassociativity of the Wick product. For the benefit of the reader, we provide a proof of this which is somewhat different from the original proof.

First, recall that $\cH$ has a basis consisting of the normally ordered monomials
$$\{ : \partial^{k_1} \alpha \cdots \partial^{k_r} \alpha: |\ 0\leq k_1\leq \cdots \leq k_r\}.$$ With respect to the natural good increasing filtration $\cH_{(0)} \subseteq \cH_{(1)} \subseteq \cdots$ where $\cH_{(-1)} \cong \mathbb{C}$ and $\cH_{(d)}$ is spanned by all monomials $: \partial^{k_1} \alpha \cdots \partial^{k_r} \alpha:$ with $r \leq d$, $\text{gr}(\cH) = \bigoplus_{d \geq 0} \cH_{(d)} / \cH_{(d-1)}$ is isomorphic to 
the differential polynomial algebra $\mathbb{C}[\alpha, \partial \alpha, \partial^2 \alpha,\dots]$. We then have linear isomorphisms
$$\cH \cong \text{gr}(\cH) \cong \mathbb{C}[\alpha, \partial \alpha, \partial^2 \alpha,\dots] \cong \cH_{\text{deg}}.$$

The action of $\theta \in  \mathbb{Z}/2\mathbb{Z}$ is given on generators by $\theta(\partial^i \alpha) = - \partial^i \alpha$. It follows that $\cH^{ \mathbb{Z}/2\mathbb{Z}}$ is spanned by all normally ordered monomials of even degree
$$:(\partial^{k_1} \alpha) \cdots (\partial^{k_r} \alpha):,\qquad k_1\geq \cdots \geq k_r,\qquad r \ \text{even}.$$

It is easy to see by induction on length that such monomials are strongly generated by the normally ordered quadratics
$$\omega_{i,j} = \ :(\partial^i \alpha) (\partial^j \alpha):,\qquad  i,j \geq 0.$$ As above, $\omega_{i,j} = \omega_{j,i}$ so we can take $i \geq j \geq 0$. Also, we have 
$$\partial (\omega_{i,j}) = \omega_{i+1,j} + \omega_{i,j+1},$$ so if we want generators for $\cH^{ \mathbb{Z}/2\mathbb{Z}}$ as a differential algebra, there is some redundancy in the above generating set. It is easy to see that the sets
$$\{\partial^j \omega_{2i,0}|\ i,j \geq 0\},\qquad \{\omega_{i,j}|\ \ i \geq j \geq 0\}$$ are linearly independent and span the same vector space, so that 
$\{\omega_{2i,0}|\ i \geq 0\}$ is also a strong generating set for $\cH^{ \mathbb{Z}/2\mathbb{Z}}$. However, this set is no longer a minimal generating set. Recall the classical relations among the generators $q_{i,j} \in \cH_{\text{deg}}^{ \mathbb{Z}/2\mathbb{Z}}$
$$q_{i,j} q_{k,\ell} - q_{i,\ell} q_{k,j} = 0.$$
If we consider the corresponding normally ordered elements in $\cH^{ \mathbb{Z}/2\mathbb{Z}}$, namely, 
$$:\omega_{i,j} \omega_{k,\ell}: - : \omega_{i,\ell} \omega_{k,j}:$$ this expression is no longer zero due to nonassociativity of the Wick product. For example, consider the relation of lowest weight $6$, namely $q_{0,0} q_{1,1} - q_{0,1} q_{0,1}$. The corresponding normally ordered expression in $\cH^{ \mathbb{Z}/2\mathbb{Z}}$ is
\begin{equation} \label{heisrelation:first} :\omega_{0,0} \omega_{1,1}:\  - \ : \omega_{0,1} \omega_{0,1}: \ = \frac{7}{6} \omega_{3,1} -\frac{1}{12} \omega_{4,0}.\end{equation} 
 
An easy calculation shows that 
\begin{equation} \begin{split} &\omega_{1,1} = -\omega_{2,0} + \frac{1}{2} \partial^2 \omega_{0,0},
\\ & \omega_{3,1} = -\omega_{4,0} + \frac{3}{2} \partial^2 \omega_{2,0}-\frac{1}{4} \partial^4 \omega_{0,0}.
\end{split} \end{equation}

Rescaling \eqref{heisrelation:first} by a factor of $-\frac{4}{5}$, replacing all variables with $\omega_{0,0}, \omega_{2,0}, \omega_{4,0}$ and their derivatives, and solving for $\omega_{4,0}$, we can rewrite this equation as
\begin{equation} \label{heisrelation:second} \omega_{0,4}= - \frac{2}{5} :\omega_{0,0} \partial^2\omega_{0,0}:\  +\frac{4}{5} \ : \omega_{0,0} \omega_{2,0}:\  + \frac{1}{5} : \partial  \omega_{0,0} \partial \omega_{0,0}:  +\frac{7}{5} \partial^2 \omega_{2,0}-\frac{7}{30} \partial^4 \omega_{0,0}.\end{equation}
It follows that the field $\omega_{4,0}$ is not needed in our strong generating set, since it can be expressed as a normally ordered polynomial in $\omega_{0,0}$ and $\omega_{2,0}$ and their derivatives. Accordingly, we call \eqref{heisrelation:second} a {\it decoupling relation} for $\omega_{4,0}$. This is very different from the situation of $\cH_{\text{deg}}^{ \mathbb{Z}/2\mathbb{Z}}$ where such decoupling relations do not exist because there are no relations with a linear term.

\begin{lemma} \label{decoup:bootstrap} For all $n\geq 2$, there exists a decoupling relation $$\omega_{2n,0} = P_{2n}(\omega_{0,0}, \omega_{2,0}),$$ where $P_{2n}(\omega_{0,0}, \omega_{2,0})$ is a normally ordered polynomial in $\omega_{0,0}, \omega_{2,0}$ and their derivatives. 
\end{lemma}

\begin{proof}
Starting from \eqref{heisrelation:second}, we can obtain a similar relation $\omega_{6,0} = P_6(\omega_{0,0}, \omega_{2,0})$ by applying the operator $(\omega_{2,0})_{(1)}$ to both sides of \eqref{heisrelation:second}, and using the following calculations:
\begin{equation} \begin{split}  (\omega_{2,0})_{(1)} \omega_{4,0} & = 4 \omega_{4,2} + 12 \omega_{6,0}
\\ & = 16 \omega_{6,0} - 16 \partial^2 \omega_{4,0}  + 20 \partial^4 \omega_{2,0}  - 4 \partial^6 \omega_{0,0},
\\  (\omega_{2,0})_{(1)} \omega_{2,0} & = 12 \omega_{4,0} - 8 \partial^2 \omega_{2,0} + 2 \partial^4 \omega_{0,0},
\\  (\omega_{2,0})_{(1)} \omega_{0,0} & = 8 \omega_{2,0},
\\  (\omega_{2,0})_{(1)} \partial^2 \omega_{2,0} & = 24 \partial^2 \omega_{4,0} - 22 \partial^4 \omega_{2,0}  + 5 \partial^6 \omega_{0,0},
\\  (\omega_{2,0})_{(1)} \partial^2 \omega_{0,0} & = 20 \partial^2 \omega_{2,0}  - 2 \partial^4 \omega_{0,0},
\\  (\omega_{2,0})_{(1)} \partial \omega_{0,0} & = 14 \partial \omega_{2,0} - \partial^3 \omega_{0,0}.
\end{split}\end{equation}
All appearances of $\omega_{4,0}$ and its derivatives can be eliminated by replacing them with \eqref{heisrelation:second} and its derivatives, so we can express $\omega_{6,0}$ as a normally ordered polynomial in $\omega_{0,0}, \omega_{2,0}$ and their derivatives.

Inductively, assume that for all $i$ such that $2\leq i \leq n-1$, we have a decoupling relation $$\omega_{2i,0} = P_{2i}(\omega_{0,0}, \omega_{2,0}).$$ We apply $(\omega_{2,0})_{(1)}$ to both sides of $\omega_{2n-2,0} = P_{2n-2}(\omega_{0,0}, \omega_{2,0})$. It is easy to check that $$(\omega_{2,0})_{(1)} \omega_{2n-2,0} \equiv (4n+4) \omega_{2n,0}$$ modulo the span of $\partial^{2+2k} \omega_{2n-2k-2,0}$ for $k = 0,1,\dots, n-1$. Inductively, we can express all terms except for $(4n+4) \omega_{2n,0}$ as normally ordered expressions on $\omega_{0,0}, \omega_{2,0}$ and their derivatives, so solving for $\omega_{2n,0}$ we obtain a decoupling relation $\omega_{2n,0} = P_{2n}(\omega_{0,0}, \omega_{2,0})$. 
\end{proof}

Finally, since $\{\omega_{2n,0}|\ n \geq 0\}$ strongly generates $\cH^{ \mathbb{Z}/2\mathbb{Z}}$, the existence of decoupling relations for $\{\omega_{2n,0}|\ n\geq 2\}$ completes the proof of Dong-Nagatomo's theorem that $\omega_{0,0}, \omega_{2,0}$ is a strong generating set. The minimality is apparent since there are no normally ordered relations among $\omega_{0,0}, \omega_{2,0}$ and their derivatives below weight $8$.

\section{Vertex algebra Hilbert problem for free field algebras} \label{sec:hilbertfreefield} 
The following theorem which vastly generalizes Theorem \ref{HZ2}, was proven in \cite{CLIII}, and is based on a series of papers by the second author \cite{LI}-\cite{LV}.

\begin{thm}\label{Hilbert:freefield}  Let $\cV$ be any VOA which a tensor product of free field algebras of the types $\cH(n)$, $\cF(m)$, $\cS(r)$, $\cA(s)$ introduced earlier. Given any reductive group $G$ of automorphisms of $\cV$, $\cV^G$ is strongly finitely generated.
\end{thm}

In this section, will give an outline of the proof. Let $\cV = \cH(n) \otimes \cF(m) \otimes \cS(r) \otimes \cA(s)$ as above. By definition, we assume that $G$ preserves the total conformal vector, which is the sum of the conformal vectors for the tensor factors. In particular, both the conformal weight grading and parity grading are preserved. So the space of weight $1/2$ is preserved which is the span of the generators of $\cS(n)$ and $\cF(m)$, and these are separately preserved. Therefore $G$ must act by automorphism on $\cS(n)$ and $\cF(m)$ separately. Next, since $\cS(n) \otimes \cF(m)$ commutes with $\cH(n) \otimes \cA(s)$, each $g\in G$ must map $\cH(n) \otimes \cA(s)$ to $\cH(n) \otimes \cA(s)$. Restricting to the weight $1$ space, which is spanned by the generators of $\cH(n) \otimes \cA(s)$, $G$ must preserve this space as well as the parity, hence it also acts on $\cH(n)$ and $\cA(s)$ separately. Therefore $G$ is a subgroup of $O(n) \otimes O(m) \otimes Sp(2r) \otimes Sp(2s)$, and $\cV^G$ can be regarded as a module over $\cH(n)^{O(n)} \otimes \cF(m)^{O(m)} \otimes \cS(r)^{Sp(2r)} \otimes \cA(s)^{Sp(2s)}$.

The proof of Theorem \ref{Hilbert:freefield} consists of the following steps
\begin{enumerate}
\item First we prove it in the special case when $\cV$ is one of the standard algebras $\cH(n)$, $\cF(m)$, $\cS(r)$, $\cA(s)$, and $G$ is the full automorphism group, which is either $O(n)$ or $Sp(2n)$. Then $\text{gr}(\cV)$ is either a symmetric or exterior algebra on $\bigoplus_{i\geq 0} V_i$, where each $V_i$ is the standard $\text{Aut}\ \cV$-module. The generators for such invariant rings are given by Weyl's first fundamental theorem of invariant theory (FFT) for the orthogonal and symplectic groups \cite{W}. This yields infinite generating sets for $\text{gr}(\cV)^{\text{Aut} \ \cV}$, and hence infinite strong generating sets for $\cV^{\text{Aut} \ \cV}$, since $\text{gr}(\cV)^{\text{Aut} \ \cV} \cong \text{gr}(\cV^{\text{Aut} \ \cV})$. The relations among these generators are given by Weyl's second fundamental theorem of invariant theory (SFT) \cite{W}, and using these relations we can eliminate all but a finite subset of the generators. From our explicit description of $\cV^{\text{Aut} \ \cV}$, we will also see that the Zhu algebra $A(\cV^{\text{Aut} \ \cV})$ is abelian in all cases. Therefore all irreducible ordinary $\cV^{\text{Aut} \ \cV}$-modules have one-dimensional top component, and are parametrized by $\text{Specm}\ A(\cV^{\text{Aut} \ \cV})$. 

\smallskip

\item For $\cV = \cH(n), \cF(m), \cS(r), \cA(s)$ and $G \subseteq \text{Aut} \ \cV$ an arbitrary reductive group, we will decompose $\cV^G$ as a module over $\cV^{\text{Aut} \ \cV}$. Even though there will typically be infinitely many irreducible $\cV^{\text{Aut} \ \cV}$-modules in this decomposition, using  Weyl's theorem on polarizations \cite{W}, we will show that a strong generating set for $\cV^G$ lies in the sum of finitely many of these modules.

\smallskip

\item Given a VOA $\cA$ with weight grading $\cA = \bigoplus_{n\geq 0} \cA[n]$ with $\cA[0] \cong \mathbb{C}$, and an ordinary $\cA$-module $M$, following \cite{MiII}, we define $C_1(\cM)$ to be the subspace of $\cM$ spanned by elements of the form $$:a(z) b(z):, \qquad a(z)\in \bigoplus_{k>0}\cA[k],\qquad b(z)\in \cM.$$ Then $\cM$ is said to be $C_1$-cofinite if $\cM / C_1(\cM)$ is finite-dimensional. For $\cV = \cH(n), \cF(m), \cS(r), \cA(s)$, we show that all irreducible $\cV^{\text{Aut} \ \cV}$-modules appearing in $\cV$, and in particular in $\cV^G$, are $C_1$-cofinite. Combined with the previous steps, this implies that strong finite generation of $\cV^G$. 

\smallskip

\item Finally, if $\cV =  \cH(n) \otimes \cF(m) \otimes \cS(r) \otimes \cA(s)$ and $G$ is any reductive group of automorphisms of $\cV$, recall that $G$ preserves the tensor factors and hence $G \subseteq O(n) \otimes O(m) \otimes Sp(2r) \otimes Sp(2s)$. Therefore $\cV^G$ decomposes as a module over the subalgebra
\begin{equation} \label{nicesub} \cH(n)^{O(n)} \otimes \cF(m)^{O(m)} \otimes \cS(r)^{Sp(2r)} \otimes \cA(s)^{Sp(2s)}.\end{equation}
As such, $\cV^G$ has a strong generating set lying in the sum of finitely many modules of the form $$M \otimes N \otimes P \otimes Q,$$ where $M$, $N$, $P$, and $Q$ are irreducible modules over $\cH(n)^{O(n)}$, $\cF(m)^{O(m)}$, $\cS(r)^{Sp(2r)}$, and $\cA(s)^{Sp(2s)}$, respectively.
As above, the strong finite generation of $\cV^G$ follows from the strong finite generation of the subalgebra \eqref{nicesub} of $\cV^G$ together with $C_1$-cofiniteness of the modules $M$, $N$, $P$, and $Q$.

\end{enumerate}

We now describe the first three steps in more detail.

\subsection{Step 1: The structure of  $\cV^{\rm{Aut} \ \cV}$}
The first step in our proof of Theorem \ref{Hilbert:freefield} is to describe $\cV^{\text{Aut} \ \cV}$ when $\cV$ is one of the standard free field algebras.

\begin{thm} \label{Hilbert:fullauto} For all $n\geq 1$, we have the following:
\begin{enumerate}
\item $\cS(n)^{Sp(2n)}$ has a minimal strong generating set 
$$w^b = \frac{1}{2}\sum_{i=1}^n \big(:\beta^{i} \partial ^b \gamma^{i}: -:(\partial^b\beta^{i}) \gamma^{i}:\big),\qquad b=1,3,\dots, 2n^2+4n-1.$$ Since $w^b$ has weight $b+1$, $\cS(n)^{Sp(2n)}$ is of type $\cW(2,4,\dots, 2n^2 + 4n)$.
\item $\cF(n)^{O(n)}$ has a minimal strong generating set
$$w^b = \frac{1}{2} \sum_{i=1}^n :(\partial^b \phi^i) \phi^i:, \qquad b = 1,3,\dots, 2n-1.$$ Since $w^b$ has weight $b+1$, $\cF(n)^{O(n)}$ is of type $\cW(2,4,\dots, 2n)$.
\item $\cA(n)^{Sp(2n)}$ has a minimal strong generating set 
$$w^b = \frac{1}{2}\sum_{i=1}^n \big(:e^i \partial^b f^i :+ :(\partial^b e^i) f^i:\big),\qquad b= 0,2,\dots, 2n-2.$$ Since $w^b$ has weight $b+2$, is of type $\cW(2,4,\dots, 2n)$.
\item $\cH(n)^{O(n)}$ has a minimal strong generating set 
$$w^b = \frac{1}{2}\sum_{i=1}^n :(\partial^b \alpha^i)\alpha^i:,\qquad  b = 0,2,\dots, 2N-2,$$ for some $N \geq \frac{1}{2}(n^2 + 3n)$. Since $w^b$ has weight $b+2$, $\cH(n)^{O(n)}$ is of type $\cW(2,4,\dots, 2N)$. We expect this bound is sharp but have proven it only for $n\leq 6$.
\end{enumerate}
\end{thm}

Statements (1) and (2) of Theorem \ref{Hilbert:fullauto} were proven in \cite{LV}, statement (3) was proven in \cite{CLII}, and statement (4) was proven in \cite{LIII,LIV}. 
We will sketch the proof of (1) only; the proofs of the others are similar and will comment briefly on the modifications that are needed, and also on why the precise bound in the case of (4) is still open.

Our starting point is that $\cS(n)$ has a good increasing filtration $$\cS(n)_{(0)} \subseteq \cS(n)_{(1)} \subseteq \cdots,\qquad \cS(n) = \bigcup_{d \geq 0} \cS(n)_{(d)}.$$ It is defined by taking $\cS(n)_{(d)}$ to be the span of all normally ordered monomials in the generators $\beta^i, \gamma^i$ and their derivatives, of total length at most $d$. Setting $\cS(n)_{(0)} \cong \mathbb{C}$ and $\cS(n)_{(-1)} \cong \{0\}$, the associated graded algebra $\text{gr}(\cS(n)) = \bigoplus_{d\geq 0} \cS(n)_{(d)} / \cS(n)_{(d-1)}$ is then isomorphic to $\mathbb{C}[U_{\infty}]$, where $U \cong \mathbb{C}^{2n}$ regarded as the standard $Sp(2n)$-module. In particular, $\text{gr}(\cS(n)) \cong \mathbb{C}[\beta^i_k, \gamma^i_k]$ where $i =1,\dots, n$ and $k\geq 0$. Here $\beta^i_k$ and $\gamma^i_k$ are the images of $\partial^k \beta^i$ and $\partial^k \gamma^i$ in $\text{gr}(\cS(n))$, and for each $k\geq 0$, $U_k = \text{Span}(\beta^i_k, \gamma^i_k)$ is a copy of the $2n$-dimensional standard $Sp(2n)$-module. As a differential algebra, note that 
\begin{equation} \label{griso1} \text{gr}(\cS(n)) \cong \mathbb{C}[U_{\infty}] \cong \mathbb{C}[ \bigoplus_{k\geq 0} U_k],\end{equation} where $\partial \beta^i_k = \beta^i_{k+1}$ and $\partial \gamma^i_k = \gamma^i_{k+1}$.

Since $Sp(2n)$ acts linearly on the space of generators $\beta^i, \gamma^i$, it preserves the filtration and \eqref{griso1} is an isomorphism of $Sp(2n)$-modules. Therefore
\begin{equation} \label{grisos} \text{gr}(\cS(n)^{Sp(2n)}) \cong \text{gr}(\cS(n))^{Sp(2n)} \cong \mathbb{C}[\bigoplus_{k\geq 0} U_k]^{Sp(2n)}.\end{equation} 

The generators and relations for $\mathbb{C}[ \bigoplus_{k\geq 0} U_k]^{Sp(2n)}$ are given by Weyl's {\it first and second fundamental theorems of invariant theory} for the standard representation of $Sp(2n)$ (Theorems 6.1.A and 6.1.B of \cite{W}).

\begin{thm} \label{weylfft} For $k\geq 0$, let $U_k$ be the copy of the standard $Sp(2n)$-module $\mathbb{C}^{2n}$ with symplectic basis $\{x_{i,k}, y_{i,k}| \ i=1,\dots,n\}$. Then $(\text{Sym} \bigoplus_{k\geq 0} U_k )^{Sp(2n)}$ is generated by the quadratics $$ q_{a,b} = \frac{1}{2}\sum_{i=1}^n \big( x_{i,a} y_{i,b} - x_{i,b} y_{i,a}\big),\qquad a,b \geq 0. $$ Note that $q_{a,b} = -q_{b,a}$. Let $\{Q_{a,b}|\ a,b\geq 0\}$ be commuting indeterminates satisfying $Q_{a,b} = -Q_{b,a}$. The kernel of the map $$ \mathbb{C}[Q_{a,b}]\ra (\text{Sym} \bigoplus_{k\geq 0} U_k)^{Sp(2n)},\qquad Q_{a,b}\mapsto q_{a,b},$$ is generated by the degree $n+1$ Pfaffians $p_I$, which are indexed by lists $I = (i_0,\dots,i_{2n+1})$ of integers satisfying $ 0\leq i_0<\cdots < i_{2n+1}$. For $n=1$ and $I = (i_0, i_1, i_2, i_3)$, we have $$p_I = q_{i_0, i_1} q_{i_2, i_3} - q_{i_0, i_2} q_{i_1, i_3}+q_{i_0, i_3} q_{i_1, i_2},$$ and for $n>1$ they are defined inductively by $$ p_I =  \sum_{r=1}^{2n+1} (-1)^{r+1} q_{i_0,i_r} p_{I_r},$$ where $I_r = (i_1,\dots, \widehat{i_r},\dots, i_{2n+1})$ is obtained from $I$ by omitting $i_0$ and $i_r$. \end{thm}

Under \eqref{grisos}, the generators $q_{a,b}$ correspond to strong generators
\begin{equation}\label{newgenomega} \omega_{a,b} = \frac{1}{2}\sum_{i=1}^n \big(:\partial^a\beta^{i} \partial ^b \gamma^{i}: -:\partial^b\beta^{i} \partial ^a \gamma^{i}:\big)\end{equation} of $\cS(n)^{Sp(2n)}$. We have some redundancy since $\partial (\omega_{a,b}) = \omega_{a+1,b} + \omega_{a,b+1}$, and it is easy to check that $\{\partial^k \omega_{0,2m+1}|\ k, m\geq 0\}$ and $\{\omega_{a,b}|\ 0\leq a < b\}$ span the same vector space. Hence 
\begin{equation} \label{stronggensnsp} \{w^{2m+1} = \omega_{0,2m+1}|\ k, m\geq 0\},\end{equation} is a strong generating set for $\cS(n)^{Sp(2n)}$. As usual, the corresponding set $\{q_{0,2m+1}|\ k, m\geq 0\}$ is a minimal generating set for $\text{gr}(\cS(n))^{Sp(2n)}$ as a differential algebra, but \eqref{stronggensnsp} is not a minimal strong generating set for $\cS(n)^{Sp(2n)}$, due to nontrivial relations coming from the second fundamental theorem of invariant theorem.

If we replace ordinary products in the above Pfaffians with Wick products in some order, the resulting expression does not vanish, but it lies in a lower filtered component of $\cS(n)$. This expression is $Sp(2n)$-invariant, and hence be expressed as a normally ordered polynomial in the above generators $w^{2m+1}$ and their derivatives. We arrive at a relation which we denote by $P_I$, whose leading term is a normal ordering of the classical Pfaffian $p_I$. If $d$ is the weight of $P_I$, let $R_n(I)$ denote the coefficient of $w^{2m+1}$ where $2m+1 = d-1$. It can be shown that this scalar $R_n(I)$ is independent of all choices of normally ordering in any of the terms of $P_I$.

What must be shown is that the first Pfaffian, which has weight $2n^2+ 4n + 2$, and corresponds to $I = (0,1,\dots, 2n+1)$, is actually a decoupling relation for the field $w^{2n^2+4n+1}$, that is, $R_n(I) \neq 0$. Then by a similar argument to the proof of Lemma \ref{decoup:bootstrap}, one can construct decoupling relations $$w^{2m+1} = P_m(w^{1}, w^{3},\dots, w^{2n^2+4n-1}),\ \text{for all} \ m \geq n^2+2n ,$$ by applying the operator $(w^{3})_{(1)}$ repeatedly. This suffices to prove that $\{w^{2m+1}|\ 0\leq m  \leq  n^2+2n-1\}$ is a strong generating set. The minimality is apparent from Weyl's second fundamental theorem, which implies that there can be no further decoupling relations.

The most efficient way to compute $R_n(I)$ for $I = (0,1,\dots, 2n+1)$ is actually to find a recursive formula that relates $R_n(I)$ for all lists $I$ of length $2n+2$, to the corresponding quantities $R_{n-1}(J)$ for $J$ a list of length $2n$. Recall that the Pfaffian $p_I$ has an expansion $$p_I =  \sum_{r=1}^{2n+1} (-1)^{r+1} q_{i_0,i_r} p_{I_r},$$ where $I_r = (i_1,\dots, \widehat{i_r},\dots, i_{2n+1})$ is obtained from $I$ by omitting $i_0$ and $i_r$. Let $P_{I_r}$ be the vertex algebra relation corresponding to $p_{I_r}$, whose leading term is a normal ordering of $p_{I_r}$. Similarly, for $a \in \{i_0,\dots, i_{2n+1}\} \setminus \{i_0, i_r\}$, let $I_{r,a}$ be the list obtained from $I_r = (i_1,\dots, \widehat{i_r},\dots, i_{2n+1})$ by replacing $i_a$ with $i_a+i_0+i_r+1$. By combining the recursive formula for the Pfaffian with identities that measure the noncommutativity and nonassociativity of the normally ordered product, one arrives at the following formula:
$$R_n(I) = -\frac{1}{2} \sum_{r=1}^{2n+1} (-1)^{r+1} \bigg( (-1)^{i_0} \sum_{a} \frac{R_{n-1}(I_{r,a})}{i_0+i_{a} +1} + (-1)^{i_r+1} \sum_{a}  \frac{R_{n-1}(I_{r,a})}{i_r+i_{a}+1}\bigg).$$
From this formula, it is not at all obvious whether or not $R_n(I) \neq 0$ for $I = 0,1,\dots, 2n+2$. The most difficult step is to show the following:
\begin{thm} \label{closedform} \begin{enumerate} 
\item For any list $I$, $R_n(I) = 0$ if the number of even and odd entries in $I$ are different. 
\item Suppose that $I = (i_0,j_0, i_1,j_1,\dots, i_n, j_n)$ is a list of nonnegative integers such that  $i_0,\dots, i_n$ are even and $j_0,\dots, j_n$ are odd. Then the following closed formula holds:
\begin{equation} \label{closedformula} R_n(I) = \frac{ n! \bigg(n+1 + \sum_{k=0}^{n} i_k+j_k\bigg)  \bigg(\prod_{0\leq k<l\leq n} (i_k - i_l)(j_k - j_l) \bigg)}{\prod_{0\leq k \leq n, \  0\leq l \leq n} (1+i_k + j_l)}.\end{equation}
\end{enumerate} \end{thm}

The proof of this formula is quite involved, and exploits the fact that the numbers $R_n(I)$ have the same symmetries under permutation of the entries as the Pfaffians. Specializing \eqref{closedformula} to the list $I = (0,1,\dots, 2n+1)$, it is apparent that $R_n(I) \neq 0$, which completes the proof of statement (1) of Theorem \ref{Hilbert:fullauto}.

It was shown in \cite{LV} that the Zhu algebra $A(\cS(n)^{Sp(2n)})$ is abelian, but an alternative explanation is the following. Being of type $\cW(2,4,\dots, 2N)$ for some $N$ and being generated by the fields in weights $2$ and $4$, it must be a quotient of the universal $2$-parameter VOA $\cW^{\text{ev}}(c,\lambda)$ constructed in \cite{KL}; all quotients of this algebra have abelian Zhu algebras by Theorem 5.5 and Corollary 5.6 of \cite{KL}. Therefore the irreducible ordinary modules of $\cS(n)^{Sp(2n)}$ are parametrized by points on the variety $\text{Spec}\ A(\cS(n)^{Sp(2n)})$ and have one-dimensional top component.

For statements (2) and (3) of Theorem \ref{Hilbert:fullauto} which involve vertex superalgebras, the generators are quadratic and the classical relations are the fermionic analogues of determinants in the case of (2), and Pfaffians in the case of (3). There are similar normally ordered polynomial relations among the generators that are quantum corrections of the classical relations. For relations of weight $m$, the nonvanishing of the coefficient of the term $w^b$ of weight $m$ is independent of all choices, and determines whether this relation is a decoupling relations for $w^b$. As above, there is a recursive formula for these coefficients. Since all contributions to this formula have the same sign, it is manifestly nonzero and it is not necessary to find a closed formula. Finally, in the case of (4), we can find a similar recursive formula for this coefficient, but it is not a sum of terms with the same sign and finding a closed formula is out of reach. However, using some asymptotic behavior of the recursive formula, we showed that for weight sufficiently high we can find a relation where it is nonzero in \cite{LIV}. Therefore we can find a decoupling relation for $w^b$ of weight $2N$ for some $N$. In all cases, by a similar argument to Lemma \ref{decoup:bootstrap} once we have a decoupling relations for $w^b$, can use it to construct decoupling relations for $w^b$ for all $m >b$. All these algebras arise as quotients of $\cW^{\text{ev}}(c,\lambda)$ and hence have abelian Zhu algebras. As above, their irreducible ordinary modules are parametrized by the points on the variety $\text{Spec}\ A(\cV^{\text{Aut}\ \cV})$ and have one-dimensional top component.

\subsection{Step 2: The structure of $\cV^G$ as a $\cV^{\text{Aut} \ \cV}$-module} As above, we focus on the case $\cV = \cS(n)$ since the other cases involve the same approach. First, we need the decomposition of $\cS(n)$ as a sum of irreducible modules for the action of $\cS(n)^{Sp(2n)}$ as well as $Sp(2n)$. We have
 \begin{equation}\label{dlmdecomp} \cS(n) \cong \bigoplus_{\nu\in H} L(\nu)\otimes M^{\nu},\end{equation} where $H$ indexes the irreducible, finite-dimensional representations $L(\nu)$ of $Sp(2n)$, and the $M^{\nu}$'s are inequivalent, irreducible, highest-weight $\cS(n)^{Sp(2n)}$-modules. This appears as Theorem 13.2 and Corollary 14.2 of \cite{KWY}, and can also be deduced using the general results of Dong, Li, and Mason in \cite{DLM}.

Let $G\subseteq Sp(2n)$ be a reductive group of automorphisms of $\cS(n)$, and recall that $$\text{gr}(\cS(n)^G )\cong (\text{gr}(\cS(n))^G \cong (\text{Sym} \bigoplus_{k\geq 0} U_k)^G = R$$ as commutative algebras, where $U_k\cong \mathbb{C}^{2n}$. For all $p\geq 1$, $GL(p)$ acts on $\bigoplus_{k =0}^{p-1} U_k $ and commutes with the action of $G$. There is an induced action of $GL(\infty) = \lim_{p\rightarrow \infty} GL(p)$ on $\bigoplus_{k\geq 0} U_k$ which commutes with $G$, so $GL(\infty)$ acts on $R$. Elements $\sigma \in GL(\infty)$ are known as {\it polarization operators}, and given $f\in R$, $\sigma f$ is known as a polarization of $f$. The following fundamental result of Weyl appears as Theorem 2.5A of \cite{W}.

\begin{thm} \label{weylpolarization} $R$ is generated by the polarizations of any set of generators for $(\text{Sym} \bigoplus_{k = 0} ^{2n-1} U_k)^G$. Since $G$ is reductive, $(\text{Sym} \bigoplus_{k = 0} ^{2n-1} U_k)^G$ is finitely generated, so there exists a finite set $\{f_1,\dots, f_r\}$, whose polarizations generate $R$. \end{thm}

An immediate consequence is that $\cS(n)^G$ has a strong generating set which lies in the direct sum of finitely many irreducible $\cS(n)^{Sp(2n)}$-modules. The reason is as follows. Each of the modules $L(\nu)$ appearing in \eqref{dlmdecomp} is a $G$-module, and since $G$ is reductive, it has a decomposition $L(\nu) =\bigoplus_{\mu\in H^{\nu}} L(\nu)_{\mu}$. Here $\mu$ runs over a finite set $H^{\nu}$ of irreducible, finite-dimensional representations $L(\nu)_{\mu}$ of $G$, possibly with multiplicity. We obtain a refinement of \eqref{dlmdecomp},
\begin{equation}\label{decompref} \cS(n) \cong \bigoplus_{\nu\in H} \bigoplus_{\mu\in H^{\nu}} L(\nu)_{\mu} \otimes M^{\nu}.\end{equation}
Under the linear isomorphism $\cS(n)^G\cong R$, let $f_i(z)$ and $(\sigma f_i)(z)$ correspond to $f_i$ and $\sigma f_i$, respectively, for $i=1,\dots, r$. Then the set $$\{(\sigma f_i)(z)\in \cS(n)^G|\ i=1,\dots,r,\ \sigma\in GL(\infty)\},$$ is a strong generating set for $\cS(n)^G$.

By enlarging the collection $f_1(z),\dots,f_r(z)$ if necessary, we may assume without loss of generality that each $f_i(z)$ lies in a single representation of the form $L(\nu_j)\otimes M^{\nu_j}$. Moreover, we may assume that $f_i(z)$ lies in a trivial $G$-submodule $L(\nu_j)_{\mu_0} \otimes M^{\nu_j}$, where $\mu_0$ denotes the trivial, one-dimensional $G$-module. In particular this means that $L(\nu_j)_{\mu_0}$ is one-dimensional. Since the actions of $GL(\infty)$ and $Sp(2n)$ on $\cS(n)$ commute, $(\sigma f_i)(z)\in L(\nu_j)_{\mu_0}\otimes M^{\nu_j}$ for all $\sigma\in GL(\infty)$. 

\begin{cor} \label{cor:fingen} $\cS(n)^G$ is a finitely generated vertex algebra.\end{cor}

\begin{proof} Since $\cS(n)^G$ is strongly generated by $\{ (\sigma f_i)(z)|\ i=1,\dots,r,\  \sigma\in GL(\infty)\}$, and each $M^{\nu_j}$ is an irreducible $\cS(n)^{Sp(2n)}$-module, $\cS(n)^G$ is generated by $f_1(z),\dots,f_r(z)$ as an algebra over $\cS(n)^{Sp(2n)}$. Since $\cS(n)^{Sp(2n)}$ is generated by $w^3$ as a vertex algebra, $\cS(n)^G$ is finitely generated. \end{proof}

\subsection{Step 3: $C_1$-cofiniteness of $\cV^{\text{Aut} \ \cV}$-modules} The final step is to prove that each irreducible $\cS(n)^{Sp(2n)}$-module appearing in the decomposition of $\cS(n)$ has the $C_1$-cofiniteness property. For this, we need an elementary fact about representations of associative algebras which can be found in \cite{LII}. Let $A$ be an associative $\mathbb{C}$-algebra, and let $W$ be a linear representation of $A$, via an algebra homomorphism $\rho: A\ra \text{End}(W)$. We may also regard $A$ as a Lie algebra with commutator as bracket. We now let $\rho_{\text{Lie}}:A\ra \text{End}(W)$ denote the same map $\rho$, which we now regard as a Lie algebra homomorphism. Clearly $\rho_{\text{Lie}}$ can be extended to a Lie algebra homomorphism $$\hat{\rho}_{\text{Lie}}: A\ra \text{End}(\text{Sym}(W)),$$ where $\hat{\rho}_{\text{Lie}}(a)$ acts by derivation on the $d^{\text{th}}$ symmetric power $\text{Sym}^d(W)$ as follows:
 $$\hat{\rho}_{\text{Lie}}(a)( w_1\cdots w_d) = \sum_{i=1}^d w_1 \cdots  \hat{\rho}_{\text{Lie}}(a)(w_i)  \cdots  w_d.$$ Letting $U(A)$ denote the universal enveloping algebra of $A$ (regarded as a Lie algebra), we finally extend this map to an algebra homomorphism $U(A)\ra \text{End}(\text{Sym}(W))$ which we also denote by $\hat{\rho}_{\text{Lie}}$. The following result appears as Lemma 3 of \cite{LII}.

\begin{lemma} \label{first} Given $\mu \in U(A)$ and $d\geq 1$, let $\Phi^d_{\mu}  = \hat{\rho}_{\text{Lie}}(\mu) \big|_{\text{Sym}^d(W)} \in \text{End}(\text{Sym}^d(W))$. Let $E$ denote the subspace of $\text{End}(\text{Sym}^d(W))$ spanned by $\{\Phi^d_{\mu}|\ \mu\in U(A)\}$, which has a filtration $$E_1\subseteq E_2\subseteq \cdots,\qquad  E = \bigcup_{r\geq 1} E_r.$$ Here $E_r$ is spanned by $\{\Phi^d_{\mu}|\  \mu \in U(A),\ \text{deg}(\mu) \leq r\}$. Then $E = E_d$.\end{lemma}

\begin{cor} \label{firstcor} Let $f\in \text{Sym}^d(W)$, and let $M\subseteq \text{Sym}^d(W)$ be the cyclic $U(A)$-module generated by $f$. Then $\{\hat{\rho}_{\text{Lie}}(\mu)(f)|\ \mu\in U(A),\ \text{deg}(\mu)\leq d\}$ spans $M$.\end{cor}

It is not difficult to check that the fields $\{w^{2m+1}|\ m \geq 0\}$ close linearly under OPE, and hence the modes 
$$\{w^{2m+1}(k)|\ m\geq 0, \ k\geq 0\}$$ span a Lie algebra, which we denote by $\cP^+$. We have a decomposition
 $$\cP^+ = \cP^+_{<0} \oplus \cP^+_0\oplus \cP^+_{>0},$$ where $\cP^+_{<0}$, $\cP^+_0$, and $\cP^+_{>0}$ are the Lie algebras spanned by $$\{w^{2m+1}(k)|\ 0\leq k< 2m+1\},\qquad \{w^{2m+1}(2m+1)\},\qquad \{w^{2m+1}(k)|\ k>2m+1\},$$ respectively. It is also apparent that $\cP^+$ preserves the filtration on $\cS(n)$, so each element of $\cP^+$ acts by a derivation of degree zero on $\text{gr}(\cS(n))$. 

Let $\cM$ be an irreducible, highest-weight $\cS(n)^{Sp(2n)}$-submodule of $\cS(n)$ with generator $f(z)$, and let $\cM'\subseteq \cM$ denote the $\cP^+$-submodule generated by $f(z)$. Since $f(z)$ has minimal weight among elements of $\cM$ and $\cP^+_{>0}$ lowers weight, $f(z)$ is annihilated by $\cP^+_{>0}$. Moreover, $\cP^+_0$ acts diagonalizably on $f(z)$, so $f(z)$ generates a one-dimensional $\cP^+_0\oplus \cP^+_{>0}$-module. By the Poincar\'e-Birkhoff-Witt theorem, $\cM'$ is a quotient of $$U(\cP^+)\otimes_{U(\cP^+_0\oplus \cP^+_{>0})} \mathbb{C} f(z),$$ and in particular is a cyclic $\cP^+_{<0}$-module with generator $f(z)$. Suppose that $f(z)$ has degree $d$, that is, $f(z)\in \cS(n)_{(d)} \setminus \cS(n)_{(d-1)}$. Since $\cP^+$ preserves the filtration on $\cS(n)$, and $\cM$ is irreducible, the nonzero elements of $\cM'$ lie in $\cS(n)_{(d)} \setminus \cS(n)_{(d-1)}$. Therefore, the projection $\cS(n)_{(d)}\ra \cS(n)_{(d)}/\cS(n)_{(d-1)} \subseteq \text{gr}(\cS(n))$ restricts to an isomorphism of $\cP^+$-modules \begin{equation}\label{isopmod} \cM'\cong \text{gr}(\cM')\subseteq \text{gr}(\cS(n)).\end{equation} By Corollary \ref{firstcor}, $\cM'$ is spanned by elements of the form \begin{equation} \label{spanningset} \{w^{2l_1+1}(k_1)\cdots w^{2l_r+1}(k_r) f(z) |\ w^{2l_i+1}(k_i)\in \cP^+_{<0},\ r\leq d\}.\end{equation}

We then need some refinements that restrict the possible modes that can appear in the spanning set \eqref{spanningset}. First, fix $m$ so that $f\in \text{Sym}^d(W_m)$, where $W_m\subseteq \text{gr}(\cS(n))$ has basis $\{\beta^i_j, \gamma^i_j|\ 1\leq  i \leq n,\ 0\leq j\leq m\}$. Then $\cM'$ is spanned by 
\begin{equation} \label{refinement:1} \{w^{2l_1+1}(k_1) \cdots w^{2l_r+1}(k_r) f(z)|\ w^{2l_i+1}(k_i)\in \cP^+_{<0},\ \  r\leq d,\ \ 0\leq k_i \leq 2m+1\}.\end{equation}
Since each $k_i \leq 2m+1$, for each fixed $N$, there are only finitely many elements of the form \eqref{refinement:1} of weight $N$. By the same argument as Lemma 8 of \cite{LII}, using the minimal strong generating type of $\cS(n)^{Sp(2n)}$ one can show such elements of sufficiently high weight all lie in $C_1(\cM)$. This proves that $\cM$ is $C_1$-cofinite.

Finally, by combining Steps 1, 2, and 3 in the case of $\cS(n)$, we obtain
\begin{thm} For any reductive group $G$ of automorphisms of $\cS(n)$, $\cS(n)^G$ is strongly finitely generated.\end{thm}
\begin{proof} By Corollary \ref{cor:fingen}, we can find $f_1(z),\dots, f_r(z) \in \cS(n)^G$ such that $\text{gr}(\cS(n))^G$ is generated by the corresponding polynomials $f_1,\dots, f_r\in \text{gr}(\cS(n))^G$, together with their polarizations. We may assume that each $f_i(z)$ lies in an irreducible, highest-weight $\cS(n)^{Sp(2n)}$-module $\cM_i$ of the form $L(\nu)_{\mu_0}\otimes M^{\nu}$, where $L(\nu)_{\mu_0}$ is a trivial, one-dimensional $G$-module. Furthermore, we may assume that $f_1(z),\dots, f_r(z)$ are highest-weight vectors for the action of $\cS(n)^{Sp(2n)}$. For each $\cM_i$, choose a finite set $S_i$ whose image in $\cM_i / C_1(\cM_i)$ is a basis for this space, and define $$S=\{w^1,w^3,\dots, w^{2n^2+4n-1} \} \cup \big(\bigcup_{i=1}^r S_i \big).$$ Since $\{w^1,w^3,\dots, w^{2n^2+4n-1}\}$ strongly generates $\cS(n)^{Sp(2n)}$ and $\bigoplus_{i=1}^r \cM_i$ contains a strong generating set for $\cS(n)^G$, $S$ must strongly generate $\cS(n)^G$. \end{proof}

\section{Hilbert problem for affine VOAs} \label{sec:hilbertaffine} 
It is tempting to try to study orbifolds of affine VOAs using a similar approach to the case of free field algebras. In other words, first describe $V^k(\gg)^{\text{Aut}\ V^k(\gg)}$, recalling that $\text{Aut}\ V^k(\gg)$ is the same as the automorphism group of $\gg$, and then for a general group $G$, decompose $V^k(\gg)^G$ as a $V^k(\gg)^{\text{Aut}\ V^k(\gg)}$-module. The difficulty with this approach is that the structure of $V^k(\gg)^{\text{Aut}\ V^k(\gg)}$ is much more complicated than that of $\cV^{\text{Aut}\ \cV}$ when $\cV$ is a free field algebra. Additionally, $V^k(\gg)^{\text{Aut}\ V^k(\gg)}$ does not have an abelian Zhu algebra in general, and it is much harder to establish the $C_1$-cofiniteness of the irreducible $V^k(\gg)^{\text{Aut}\ V^k(\gg)}$-modules appearing in $V^k(\gg)$.

There is another approach to this problem which we outline in this section. It is an application of the notion of a vertex algebra over a commutative ring introduced in Section \ref{sec:voacommring}. Let $\gg$ be a simple, finite-dimensional Lie algebra of dimension $n$, equipped with its standard bilinear for $(,)$. Fix a basis $\{\xi_1,\dots, \xi_n\}$ for $\gg$ which is orthonormal with respect to $(,)$, so that the corresponding generators $X^{\xi_i} \in V^k(\gg)$ satisfy
\begin{equation} \label{dfam:affine} X^{\xi_i}(z) X^{\xi_j}(w) \sim \delta_{i,j} k (z-w)^{-2} + X^{[\xi_i, \xi_j]}(w) (z-w)^{-1}.\end{equation}

Next, let $\kappa$ be a formal variable satisfying $\kappa^2 = k$, and let $F$ be the ring of complex, rational functions of $\kappa$ degree at most zero, with possible poles only at $\kappa = 0$. Let $\cV$ be the vertex algebra with coefficients in $F$ which is freely generated by 
$\{a^{\xi_i} |\ i=1,\dots, n\}$, satisfying
\begin{equation}\label{dfam:affinenew} a^{\xi_i}(z) a^{\xi_j}(w) \sim \delta_{i,j} (z-w)^{-2} + \frac{1}{\kappa}a^{[\xi_i, \xi_j]}(w) (z-w)^{-1}.\end{equation}
For all $k\neq 0$, we have a surjective vertex algebra homomorphism $$\cV \ra V^k(\gg),\qquad a^{\xi_i} \mapsto \frac{1}{\sqrt{k}} X^{\xi_i},$$ whose kernel is the vertex algebra ideal $(\kappa - \sqrt{k}) \cV$ generated by $\kappa - \sqrt{k}$, which is regarded as a element of the weight-zero subspace $\cV[0] \cong F$. In particular, for all $k\neq 0$ we can realize $V^k(\gg)$ as a quotient
 $$V^k(\gg) \cong \cV/ (\kappa - \sqrt{k}) \cV.$$
 
 Next, observe that all structure constants appearing in the OPE algebra \eqref{dfam:affinenew} have degree at most zero in $\kappa$, so the $\kappa\rightarrow \infty$ limit is well defined. Since the first-order poles vanish in the limit and only the second order poles survive, we see that $\cV^{\infty} = \lim_{\kappa \ra\infty} \cV$ has generators $\alpha^{\xi_i} = \lim_{\kappa \rightarrow \infty} a^{\xi_i}$ satisfying
\begin{equation} \alpha^{\xi_i}(z) \alpha^{\xi_j}(w) \sim \delta_{i,j} (z-w)^{-2}.\end{equation} Therefore $\cV^{\infty}$ is isomorphic to the rank $n$ Heisenberg algebra $\cH(n)$.

The vertex algebra $\cV$ is a special case of the notion of a {\it deformable family} that was developed by the second author and Creutzig \cite{CLI,CLIII}. These are vertex algebras defined over a ring $F_K$ of rational functions in a formal variable $\kappa$ over $\mathbb{C}$, with possible poles in some subset $K\subseteq \mathbb{C}$. In particular, $\cB$ is required to be a free $F_K$-module with a grading
\begin{equation} \label{gradingonb}  \cB = \bigoplus_{d \geq 0} \cB[d],\end{equation} where $d\in \mathbb{Z}$ or $\frac{1}{2} \mathbb{N}$. The weight $d$ component $\cB[d]$ is a free $F_K$-module of finite rank, and we assume that $\cB[0] \cong F_K$.

For each $k\in \mathbb{C} \setminus K$, the ideal $(\kappa - k) \subseteq F_K$ is a maximal ideal and $F_K / (\kappa - k) \cong \mathbb{C}$. Let $(\kappa - k) \cB$ denote the vertex algebra ideal generated by $\kappa - k$, which is regarded as an element of $\cB[0] \cong F_K$. Clearly $(\kappa - k) \cB$ consists of all finite sums $\sum_i f_i b_i$, where $b_i \in \cB$ and $f_i \in (\kappa - k) \subseteq F_K$.
The quotient 
$$\cB^k = \cB / (\kappa - k)\cB$$ is then a vertex algebra over $\mathbb{C}$, and $\cB$ and $\cB^k$ have the same graded character, i.e.,
$$\chi(\cB ,q) = \sum_{d \geq 0} \text{rank}_{F_K}(\cB[d]) q^d = \sum_{d \geq 0} \text{dim}_{\mathbb{C}}(\cB^k[d]) q^d = \chi(\cB^k ,q).$$

Since $F_K$ consists of rational functions of degree at most zero, $\cB$ has a well-defined limit $\cB^{\infty} = \lim_{\kappa\ra \infty} \cB$.  Fix a basis $\{a_i|\ i \in I\}$ of $\cB$ as an $F_K$-module, where each $a_i$ is homogeneous with respect to weight, and define $\cB^{\infty}$ to be the vector space over $\mathbb{C}$ with basis $\{\alpha_i|\ i \in I\}$. It inherits the grading $\cB^{\infty} = \bigoplus_{m\geq 0} \cB^{\infty}[m]$ if we define $\text{wt} \ \alpha_i = \text{wt} \ a_i$. We have a map $$\phi: \cB \ra \cB^{\infty},\qquad \phi(\sum_{i\in I} f_i a_i) = \sum_{i\in I} (\lim_{\kappa \ra \infty} f_i) \alpha_i,\qquad f_i \in F_K$$ which satisfies
 $$\phi( f \omega +g \nu) = (\lim_{\kappa \ra \infty} f) \phi(\omega) + (\lim_{\kappa \ra \infty} g) \phi(\nu).$$ The vertex algebra structure on $\cB^{\infty}$ is defined first on the basis $\{a_i|\ i \in I\}$ by 
\begin{equation} \label{phimorphism} \alpha_i \circ_n \alpha_j = \phi( a_i \circ_n a_j), \qquad i,j\in I, \qquad n\in \mathbb{Z}, \end{equation} and then extended by linearity. For all $\omega,\nu \in \cB$ and $n \in \mathbb{Z}$, we have by construction
 \begin{equation} \label{eq:preserve} \phi(\omega \circ_n \nu) = \phi(\omega) \circ_n \phi(\nu).\end{equation}
Then $\cB^{\infty}$ is a vertex algebra over $\mathbb{C}$ with graded character $\chi(\cB^{\infty} ,q) = \chi(\cB ,q)$, and the vertex algebra structure is independent of our choice of basis of $\cB$.

The most important feature of deformable families is that a strong generating set for the limit will give rise to a strong generating sets for the family after tensoring with a ring of the form $F_S$ for a subset $S\subseteq \mathbb{C}$ containing $K$. The precise statement appears in \cite{CLI,CLIII}. 

\begin{thm} Let $K\subseteq \mathbb{C}$ be at most countable, and let $\cB$ be a vertex algebra over $F_K$ with weight grading \eqref{gradingonb} such that $\cB[0] \cong F_K$. Let $U = \{\alpha_i|\ i\in I\}$ be a strong generating set for $\cB^{\infty}$, and let $T = \{a_i|\ i\in I\}$ be a subset of $\cB$ such that $\phi(a_i) = \alpha_i$. There exists a subset $S\subseteq \mathbb{C}$ containing $K$ which is at most countable, such that $F_S \otimes_{F_K}\cB$ is strongly generated by $T$. Here we have identified $T$ with the set $\{1 \otimes a_i|\ i\in I\} \subseteq F_S \otimes_{F_K} \cB$. \end{thm}

For the benefit of the reader we reproduce the proof of this lemma, following (and slightly paraphrasing) the argument appearing in \cite{CLIII}.

\begin{proof} For all $n >0$, let $$d_n = \text{rank}_{F_K}(\cB[n]) = \text{dim}_{\mathbb{C}} \cB^{\infty}[n],$$ and fix a basis $\{b_1,\dots, b_{d_n}\}$ for the weight $n$ submodule $\cB[n]$ as an $F_K$-module. Then the corresponding set $\{\beta_1,\dots, \beta_{d_n}\}$, where $\beta_j = \phi(b_j)$, forms a basis of $\cB^{\infty}[n]$. 

Since $U = \{\alpha_i|\ i\in I\}$ strongly generates $\cB^{\infty}$, there is a subset of normally ordered monomials in the generators
$$\{\mu_1,\dots, \mu_{d_n}\},$$ which each have the form $:\partial^{k_1} \alpha_{i_1} \cdots \partial^{k_r} \alpha_{i_r}:$ for $i_1,\dots,i_i \in I$ and $k_1,\dots, k_r \geq 0$, which is another basis for $\cB^{\infty}[n]$. Observe that $\cB^{\infty}$ need not be {\it freely} generated by $U$, so this subset may not include all possible normally ordered monomials of weight $n$. Let $M = (m_{j,k}) \in GL(n)(\mathbb{C})$ denote the change of basis matrix such that 
$$\mu_j = \sum_k m_{j,k} \beta_{k}.$$
Next, let $\{m_1,\dots, m_{d_n}\}$ be the monomials in the elements $\{a_i|\ i \in I\}$ and their derivatives obtained from $\mu_1,\dots, \mu_{d_n}$ by replacing each $\alpha_i$ by $a_i$. It follows from \eqref{phimorphism} that $\phi(m_j) = \mu_j$. Moreover, since $\{b_1,\dots, b_{d_n}\}$ is a basis for $\cB[n]$ as an $F_K$-module, we can write 
$$m_j = \sum_k m_{j,k}(\kappa) b_k,$$ for some functions $m_{j,k}(\kappa) \in F_K$. 
Taking the limit shows that $\lim_{\kappa \ra \infty} m_{j,k}(\kappa) = m_{j,k}$, and since the matrix $M = (m_{j,k})$ has nonzero determinant, this holds for $M(\kappa) = (m_{j,k}(\kappa))$ as well. Moreover, $\text{det}(M(\kappa))$ lies in $F_K$ and hence has degree zero as a rational function. Then $\frac{1}{\text{det}(M(\kappa))}$ has degree zero, but need not lie in $F_K$ since the denominator may have roots that do not lie in $K$. But if we let $S_n$ be the union of $K$ and this set of roots, $\text{det}(M(\kappa))$ will be invertible in $F_{S_n}$ and $\{m_1,\dots, m_{d_n}\}$ will form a basis of $F_{S_n} \otimes_{F_K} \cB[n]$ as an $F_{S_n}$-module. We now take $S = \bigcup_{n\geq 0} S_n$. The set $T$ clearly has the desired properties. 
\end{proof}

\begin{cor} \label{cor:passage} For $k \in \mathbb{C} \setminus S$, the vertex algebra $\cB^k = \cB/ (\kappa -k)\cB$ is strongly generated by the image of $T$ under the map $\cB \ra \cB^k$.
\end{cor}

Now we restrict to the case where $\cB$ is the vertex algebra $\cV$ defined as above, which admits $V^k(\gg)$ as a quotient for all $k\neq 0$. If $G$ is a reductive group of automorphisms of $V^k(\gg)$ as a one-parameter vertex algebra, then $G$ acts on $\cV$ as well as on the limit $\cV^{\infty} \cong \cH(n)$. The following lemma is nontrivial and is a consequence of the fact that $\cV$ admits a good increasing filtration in the sense of \cite{LiII}.

\begin{lemma} \label{deffamGcommute} \cite{LIV} \label{ginfcom} $(\cV^G)^{\infty} = (\cV^{\infty})^G = \cH(n)^G$. In particular, $\cV^G$ is a deformable family with limit $\cH(n)^G$.
\end{lemma}

Since $\cH(n)^G$ is strongly finitely generated by Theorem \ref{Hilbert:freefield}, the following is immediate from Corollary \ref{cor:passage} and Lemma \ref{deffamGcommute}.

\begin{thm} \label{Hilbert:affine} For any simple Lie algebra $\gg$ and any reductive group of automorphisms of $V^k(\gg)$, the orbifold $V^k(\gg)^G$ is strongly finitely generated for generic values of $k$.
\end{thm}

\subsection{The case of affine vertex superalgebras}
It is straightforward to generalize these results to affine vertex superalgebras. This was carried out in \cite{CLIII}, and we briefly recall the modifications that are needed. Let $\gg = \gg_{\bar 0} \oplus \gg_{\bar 1}$ be a finite-dimensional, simple Lie superalgebra over $\mathbb{C}$, where $\text{dim}\ \gg_{\bar 0} = n$ and $\text{dim}\ \gg_{\bar 1} = 2m$, with its standard bilinear form $(,)$. Fix a basis $\{\xi_1,\dots, \xi_n\}$ for $\gg_{\bar 0}$ and $\{\eta^{\pm}_1,\dots, \eta^{\pm}_m\}$ for $\gg_{\bar 1}$, so the generators $X^{\xi_i}, X^{\eta^{\pm}_j}$ of $V^k(\gg)$ satisfy \begin{equation} \begin{split} &X^{\xi_i}(z) X^{\xi_j}(w) \sim \delta_{i,j} k (z-w)^{-2} + X^{[\xi_i, \xi_j]}(w) (z-w)^{-1}, \\
&X^{\eta^{+}_i}(z) X^{\eta^{-}_j}(w) \sim \delta_{i,j} k (z-w)^{-2} + X^{[\eta^{+}_i, \eta^{-}_j]}(w) (z-w)^{-1},\\
&X^{\xi_i}(z) X^{\eta^{\pm}_j}(w) \sim X^{[\xi_i, \eta^{\pm}_j]}(w) (z-w)^{-1},\\ 
& X^{\eta^{\pm}_i}(z) X^{\eta^{\pm}_j}(w) \sim X^{[\eta^{\pm}_i, \eta^{\pm}_j]}(w) (z-w)^{-1}.\\
\end{split}
\end{equation}

As above, let $\kappa$ be a formal variable satisfying $\kappa^2 = k$, and let $F$ be the ring of rational functions of $\kappa$ degree at most zero, with poles only at $\kappa = 0$. Define a vertex algebra $\cV$ over $F$ which is freely generated by 
$\{a^{\xi_i}, a^{\eta^{\pm}_j} |\ i=1,\dots, n,\ j=1,\dots, m\}$, satisfying
\begin{equation} \begin{split} &a^{\xi_i}(z) a^{\xi_j}(w) \sim \delta_{i,j} (z-w)^{-2} + \frac{1}{\kappa}a^{[\xi_i, \xi_j]}(w) (z-w)^{-1}, \\
& a^{\eta^{+}_i}(z) a^{\eta^{-}_j}(w) \sim \delta_{i,j} (z-w)^{-2} + \frac{1}{\kappa}a^{[\eta^{+}_i, \eta^{-}_j]}(w) (z-w)^{-1},\\
& a^{\xi_i}(z) a^{\eta^{\pm}_j}(w) \sim + \frac{1}{\kappa}a^{[\xi_i, \eta^{\pm}_j]}(w) (z-w)^{-1},\\ 
& a^{\eta^{\pm}_i}(z) a^{\eta^{\pm}_j}(w) \sim + \frac{1}{\kappa}a^{[\eta^{\pm}_i, \eta^{\pm}_j]}(w) (z-w)^{-1}.\\
\end{split}
\end{equation}
For $k\neq 0$, we then have a surjective homomorphism  
$$\cV \ra V^k(\gg),\qquad a^{\xi_i} \mapsto \frac{1}{\sqrt{k}} X^{\xi_i},\qquad a^{\eta^{\pm}_j} \mapsto \frac{1}{\sqrt{k}} a^{\eta^{\pm}_j},$$ whose kernel is the ideal $(\kappa - \sqrt{k})\cV$, so $V^k(\gg) \cong \cV/ (\kappa - \sqrt{k})$. Then $\cV^{\infty} = \lim_{\kappa \ra\infty} \cV$  has even generators $\alpha^{\xi_i}$ for $i=1,\dots, n$, and odd generators $e^{\eta^+_j}, e^{\eta^-_j}$ for $j=1,\dots, m$, satisfying \begin{equation} \begin{split} \alpha^{\xi_i}(z) \alpha^{\xi_j}(w) \sim \delta_{i,j} (z-w)^{-2},\\ e^{\eta^+_i}(z) e^{\eta^{-}_j}(w) \sim \delta_{i,j} (z-w)^{-2}.\end{split} \end{equation} In particular, $\cV^{\infty} \cong \cH(n) \otimes \cA(m)$, where $\cA(m)$ is the rank $m$ symplectic fermion algebra.

The analogous result to Lemma \ref{deffamGcommute} was proved in \cite{CLIII}: for any reductive group $G$ of $V^k(\gg)$, we have 
$$(\cV^G)^{\infty} = (\cV^{\infty})^G = (\cH(n) \otimes \cA(m))^G.$$ In particular, $\cV^G$ is a deformable family with limit $(\cH(n) \otimes \cA(m))^G$. Again by Theorem \ref{Hilbert:freefield}, $(\cH(n) \otimes \cA(m))^G$ is strongly finitely generated, so this holds for $V^k(\gg)$ for generic values of $k$. We obtain

\begin{thm} \label{Hilbert:affinesuper} For any simple Lie superalgebra $\gg$ and any reductive group of automorphisms of $V^k(\gg)$, the orbifold $V^k(\gg)^G$ is strongly finitely generated for generic values of $k$.
\end{thm}

\section{Hilbert problem for $\cW$-algebras} \label{sec:hilbertW} 
In this section, we outline the proof the the following theorem of \cite{CLIV}. 
\begin{thm} \label{thm:hilbertW} Let $\gg$ be a simple, finite-dimensional Lie superalgebra and let $f \in \gg$ be a nilpotent element. Let $G$ be a reductive automorphism group of $\cW^k(\gg,f)$. Then the orbifold $\cW^k(\gg,f)^G$ is strongly finitely generated for generic values of $k$.
\end{thm} 
This is a vast generalization of Theorem \ref{Hilbert:affinesuper}, which is the case when $f$ is the zero nilpotent. The key ingredient in the proof is to show that all such $\cW$-algebras admit a large level limit which is a tensor product of free field algebras. There are four infinite families of such free field algebras; they are either of orthogonal or sympletic type, and the generators are either even or odd. They are natural generalizations of the standard free field algebras introduced in Section \ref{sec:vertex}. Once we show that all $\cW$-algebras admit such limits, as in the previous section we show that Theorem \ref{thm:hilbertW} reduces to proving that the Hilbert theorem holds for tensor products of these free field algebras. 

\subsection{General free field algebras} 
We first recall the general notion of free field algebra introduced in \cite{CLIV}. This is a vertex superalgebra $\cV$ with weight grading
$$\cV = \bigoplus_{d \in \frac{1}{2} \mathbb{Z}_{\geq 0} }\cV[d],\qquad \cV[0] \cong \mathbb{C},$$ with strong generators $\{X^i|\ i \in I\}$ satisfying OPE relations
\begin{equation} \label{gffa:def} X^i(z) X^j(w) \sim a_{i,j} (z-w)^{-\text{wt}(X^i) - \text{wt}(X^j)},\quad a_{i,j} \in \mathbb{C},\ \quad a_{i,j} = 0\ \ \text{if} \ \text{wt}(X^i) +\text{wt}(X^j)\notin \mathbb{Z}.\end{equation} In other words, the only nontrivial terms appear in the OPEs are the constant terms. Note that we do not assume that $\cV$ has a conformal structure. There are four infinite families of standard free field algebras that are constructed starting with a vector space $V$ and a form which is either symmetric or skew-symmetric, such that the fields are even or odd.
\smallskip

\noindent {\it Even algebras of orthogonal type}. These are generalizations of the Heisenberg algebra that are constructed from an $n$-dimensional  vector space $V$ equipped with a nondegenerate, symmetric bilinear form $\langle, \rangle$, and an even integer $k\geq 2$. Let $\cO_{\text{ev}}(V,k)$ be the vertex algebra with even generators $\{a^u|\ u \in V\}$ of weight $\frac{k}{2}$, and OPEs
$$a^u(z) a^v(w) \sim \langle u, v \rangle (z-w)^{-k}.$$
The case $k=2$ corresponds to the usual Heisenberg algebra $\cH(V)$, and $\cO_{\text{ev}}(V,k)$ has automorphism group $O(n)$, i.e., the group of linear automorphisms of $V$ which preserve the pairing  $\langle, \rangle$. Here $\cH(V)$ is the vertex algebra with even generators $\{\partial^k \alpha^u|\ u \in V\}$ of weight $1$, and OPE relations $\alpha^u(z) \alpha^v(w) \sim \langle u, v \rangle (z-w)^{-2}.$

There is a concrete realization of $\cO_{\text{ev}}(V,k)$ inside $\cH(V)$ as the subalgebra generated by $\{\partial^k \alpha^u|\ u \in V\}$. If we choose an orthonormal basis $v_1,\dots, v_n$ for $V$ relative to $\langle, \rangle$, and denote the corresponding fields by $\alpha^{1},\dots, \alpha^n$, we set
$$ a^i = \frac{\epsilon}{\sqrt{(k-1)!}} \partial^{k/2-1}\alpha^i,\qquad i=1,\dots, n.$$ Here $\epsilon = \sqrt{-1}$ if $4|k$, and otherwise $\epsilon = 1$. They freely generate $\cO_{\text{ev}}(V,k)$ and satisfy
$$a^i(z) a^j(w) \sim \delta_{i,j} (z-w)^{-k}.$$ 
We use the notation $\cO_{\text{ev}}(n,k)$ from now on when we have chosen this generating set. Note that $\cO_{\text{ev}}(n,k)$ has no Virasoro element for $k\geq 4$, but it is simple and its conformal weight grading is inherited from the grading on $\cH(n)$. For $n,m\geq 1$ and $k$ fixed, we have 
\begin{equation} \label{type:oev} \cO_{\text{ev}}(n,k) \otimes \cO_{\text{ev}}(m,k) \cong \cO_{\text{ev}}(n+m,k).\end{equation}

\smallskip

\noindent {\it Even algebras of symplectic type}.  These are generalizations of the $\beta\gamma$-system that are constructed from an $2n$-dimensional symplectic vector space $V$ with a nondegenerate, skew-symmetric bilinear form $\langle, \rangle$, and an odd integer $k\geq 1$. Let $\cS_{\text{ev}}(V,k)$ be the vertex algebra with even generators $\{\psi^u|\ u \in V\}$ of weight $\frac{k}{2}$, which are linear in $u \in V$ and satisfy
$$\psi^u(z) \psi^v(w) \sim \langle u, v \rangle (z-w)^{-k}.$$
The case $k=1$ corresponds to the usual $\beta\gamma$-system $\cS(V)$, which has even generators $\{\psi^u|\ u \in V\}$ of weight $\frac{1}{2}$ and OPE relations
$\psi^u(z) \psi^v(w) \sim \langle u, v \rangle (z-w)^{-1}$. For all $k\geq 1$, $\cS_{\text{ev}}(V,k)$ has automorphism group $Sp(2n)$, i.e., the group of linear automorphisms which preserve the pairing $\langle, \rangle$.

For all odd $k \geq 1$, $\cS_{\text{ev}}(V,k)$ can be realized inside $\cS(V)$ as the subalgebra with generators $\{\partial^k \psi^u|\ u \in V\}$. Fixing a symplectic basis $u_1,\dots, u_n, v_1,\dots, v_n$ for $V$ as above, we denote the corresponding generators $\psi^{u_i}$ and $\psi^{v_i}$ of $\cS(V) = \cS(n)$ by $\beta^i$ and $\beta^i$, respectively. We then define 
$$a^i =  \frac{\epsilon}{\sqrt{(k-1)!}} \partial^{(k-1)/2} \beta^i, \qquad b^i =  \frac{\epsilon}{\sqrt{(k-1)!}} \partial^{(k-1)/2}\gamma^i, \qquad i=1,\dots, n,$$ where $\epsilon$ is as above. Then $\cS_{\text{ev}}(V,k)$ is freely generated by $a^i, b^i$, and these satisfy
\begin{equation} \begin{split} a^i(z) b^{j}(w) &\sim \delta_{i,j} (z-w)^{-k},\qquad b^{i}(z)a^j(w)\sim -\delta_{i,j} (z-w)^{-k},\\ a^i(z)a^j(w) &\sim 0,\qquad\qquad\qquad \ \ \ \ b^i(z)b^j (w)\sim 0.\end{split} \end{equation} 
We use the notation $\cS_{\text{ev}}(n,k)$ when we have chosen this generating set. As above, it has no Virasoro element for $k\geq 3$, but it is simple and its conformal weight grading is inherited from the grading on $\cS(n)$. For  $n,m\geq 1$ and $k$ fixed, we have 
\begin{equation} \label{type:sev} \cS_{\text{ev}}(n,k) \otimes \cS_{\text{ev}}(m,k) \cong \cS_{\text{ev}}(n+m,k).\end{equation}

\smallskip

\noindent {\it Odd algebras of symplectic type}. These are generalizations of the symplectic fermion algebra that are constructed from an $2n$-dimensional symplectic vector space $V$ with a nondegenerate, skew-symmetric bilinear form $\langle, \rangle$, and an even integer $k\geq 2$. 

Let $\cS_{\text{odd}}(V,k)$ be the vertex algebra with odd generators $\{\xi^u|\ u \in V\}$ of weight $\frac{k}{2}$, which are linear in $u \in V$ and satisfy
$$\xi^u(z) \xi^v(w) \sim \langle u, v \rangle (z-w)^{-k}.$$
The case $k=2$ corresponds to the usual symplectic fermion algebra $\cA(V)$ which has odd generators $\{\xi^u|\ u \in V\}$ of weight $1$ satisfying $\xi^u(z) \xi^v(w) \sim \langle u, v \rangle (z-w)^{-2}$. Then $\cS_{\text{odd}}(V,k)$ has automorphism group $Sp(2n)$, i.e., the group of linear automorphisms which preserve the pairing $\langle, \rangle$.

For all even $k \geq 2$, $\cS_{\text{odd}}(V,k)$ can be realized inside the symplectic fermion algebra  $\cA(V)$ as the subalgebra with generators $\{\partial^k \xi^u|\ u \in V\}$. Fixing a symplectic basis $u_1,\dots, u_n, v_1,\dots, v_n$ for $V$ as above, we denote the corresponding generators $\xi^{u_i}$ and $\xi^{v_i}$ of $\cA(V) = \cA(n)$ by $e^i$ and $f^i$, respectively. We then define 
$$a^i =  \frac{\epsilon}{\sqrt{(k-1)!}} \partial^{k/2-1} e^i, \qquad b^i =  \frac{\epsilon}{\sqrt{(k-1)!}} \partial^{k/2-1} f^i,\qquad i = 1,\dots, n.$$ They freely generate $\cS_{\text{odd}}(n,k)$ and satisfy
\begin{equation} \begin{split} a^{i} (z) b^{j}(w)&\sim \delta_{i,j} (z-w)^{-k},\qquad b^{j}(z) a^{i}(w)\sim - \delta_{i,j} (z-w)^{-k},\\ a^{i} (z) a^{j} (w)&\sim 0,\qquad\qquad\qquad\ \ \ \  b^{i} (z) b^{j} (w)\sim 0. \end{split} \end{equation}
We use the notation $\cS_{\text{odd}}(n,k)$ when we have chosen this generating set. It has no Virasoro element for $k\geq 4$, but it is simple and its conformal weight grading is inherited from the grading on $\cA(n)$. For $n,m\geq 1$ and $k$ fixed, we have 
\begin{equation} \label{type:sodd} \cS_{\text{odd}}(n,k) \otimes \cS_{\text{odd}}(m,k) \cong \cS_{\text{odd}}(n+m,k).\end{equation}

\smallskip

\noindent {\it Odd algebras of orthogonal type}. 
These are generalizations of the free fermion algebra that are constructed from an $n$-dimensional  vector space $V$ equipped with a nondegenerate, symmetric bilinear form $\langle, \rangle$, and an odd integer $k\geq 1$. 
We let $\cO_{\text{odd}}(V,k)$ be the vertex algebra with odd generators $\{\phi^u|\ u \in V\}$ of weight $\frac{k}{2}$, and OPEs
$$\phi^u(z) \phi^v(w) \sim \langle u, v \rangle (z-w)^{-k}.$$
The case $k=1$ corresponds to the usual free fermion algebra $\cF(V)$, which has odd generators $\{\phi^u|\ u \in V\}$ of weight $\frac{1}{2}$ satisfying $\phi^u(z) \phi^v(w) \sim \langle u, v \rangle (z-w)^{-1}$. Then $\cO_{\text{odd}}(V,k)$ has automorphism group $O(n)$, i.e., the group of linear automorphisms of $V$ which preserve the pairing  $\langle, \rangle$.

For all odd $k\geq 1$, $\cO_{\text{odd}}(V,k)$ can be realized inside $\cF(V)$ as the subalgebra with generators $\{\partial^k \alpha^u|\ u \in V\}$. Fix an orthonormal basis $v_1,\dots, v_n$ for $V$ relative to $\langle, \rangle$, and denote the corresponding generators of $\cF(V) = \cF(n)$ by $\phi^{1},\dots, \phi^n$. We then define 
$$a^i = \frac{\epsilon}{\sqrt{(k-1)!}} \partial^{(k-1)/2} \phi^i,\qquad i = 1,\dots, n.$$ They freely generate $\cO_{\text{odd}}(n,k)$ and satisfy
$$a^i(z) a^j(w) \sim \delta_{i,j} (z-w)^{-k}.$$ 
As above, we use the notation $\cO_{\text{odd}}(n,k)$ when we have chosen this generating set. It has no Virasoro element for $k\geq 3$, but it is simple and its conformal weight grading is inherited from the grading on $\cF(n)$. For $n,m\geq 1$ and $k$ fixed, we have 
\begin{equation} \label{type:oodd} \cO_{\text{odd}}(n,k) \otimes \cO_{\text{odd}}(m,k) \cong \cO_{\text{odd}}(n+m,k).\end{equation}

\subsection{Free field limits of $\cW$-algebras} \label{sec:freewalg} 
In the previous section, we showed that for any simple Lie superalgebra $\gg$, $V^k(\gg)$ has large level limit $\cH(n) \otimes \cA(m)$, where $n = \text{dim} \ \gg_{\bar{0}}$ and $2m = \text{dim} \ \gg_{\bar{1}}$. A generalization of this result was given by Theorem 3.5 and Corollary 3.4 of \cite{CLIV}.

\begin{thm} \label{thm:wfreelimit} Let $\gg$ be a Lie superalgebra  with invariant, nondegenerate supersymmetric bilinear form $(,)$, and let $f$ be a nilpotent in the even part of $\mathfrak{g}$. Then $\cW^{k}(\gg,f)$ admits a limit $\cW^{\text{free}}(\gg,f)$ which decomposes as a tensor product of standard free field algebras of the form $\cO_{\text{ev}}(n,k)$, $\cO_{\text{odd}}(n,k)$,  $\cS_{\text{ev}}(n,k)$, and $\cS_{\text{odd}}(n,k)$.
\end{thm}

In the language of the previous section, there exists a deformable family $\cW(\gg,f)$ such that $$\cW^k(\gg,f) \cong \cW(\gg,f) / (\kappa - \sqrt{k}) \cW(\gg,f),\ \text{for all} \ k \neq 0.$$ In this notation, 
$$\cW^{\text{free}}(\gg,f) \cong \cW^{\infty}(\gg,f) = \lim_{\kappa \rightarrow \infty} \cW(\gg,f).$$

In the notation of Theorem \ref{thm:Walgmain}, recall that we have a decomposition $\gg^f = \bigoplus_{j\geq 0} \gg^f_{-j}$, where $j \in  \frac{1}{2} \mathbb{N}$. Fix a basis $J^f_{-j}$ for $\gg^f_{-j}$; then for each $q^{\alpha} \in J^f_{-j}$, we have a field $K^{\alpha} \in \cW^k(\gg,f)$ of weight $1+j$, and these fields strongly and freely generate $\cW^k(\gg,f)$. We denote the corresponding fields in the limit $\cW^{\infty}(\gg,f)$ by $X^{\alpha}$. Finally, we partition $J^f_{-j}$ into subsets $J^f_{-j, \text{ev}}$ and $J^f_{-j, \text{odd}}$ consisting of even and odd elements. Then

\begin{enumerate}
\item If $j$ is a half-integer, the span of $\{X^{\alpha}|\ \alpha \in J^f_{-j,\text{ev}}\}$ has a skew-symmetric pairing, and gives rise to even algebra of symplectic type.

\item If $j$ is an integer, $\{X^{\alpha}|\ \alpha \in J^f_{-j,\text{ev}}\}$ has a symmetric pairing, and gives rise to even algebra of orthogonal type.

\item If $j$ is a half-integer, $\{X^{\alpha}|\ \alpha \in J^f_{-j,\text{odd}}\}$ has a symmetric pairing, and gives rise to odd algebra of orthogonal type.

\item If $j$ is an integer, $\{X^{\alpha}|\ \alpha \in J^f_{-j,\text{odd}}\}$ has a skew-symmetric pairing, and gives rise to odd algebra of symplectic type.
\end{enumerate}

We remark that Theorem \ref{thm:wfreelimit} has many other applications in addition to Theorem \ref{thm:hilbertW}. For example, since $\cW^{\infty}(\gg,f)$ is simple, it implies the simplicity of $\cW^{k}(\gg,f)$ for generic values of $k$ whenever $\gg$ admits an invariant, nondegenerate supersymmetric bilinear form $(,)$ as above. This appears as Theorem 3.6 of \cite{CLIV}, and was previously known only in some cases such as principal and minimal $\cW$-algebras.

As in the case of affine superalgebras, $\cW^k(\gg,f)$ possesses a good increasing filtration in the sense of Li \cite{LiII}; see Lemma 4.1 of \cite{CLIV}. Here the generating fields of $\cW^k(\gg,f)$ of weight $d$ need to be assigned filtration degree $d$ in order for the properties of a good increasing filtration to hold. As before, this filtration can be used to show that the orbifold $\cW(\gg,f)^G$ of the deformable family $\cW(\gg,f)$ is again a deformable family, and 
$$\lim_{\kappa \rightarrow \infty} \cW(\gg,f)^G \cong \big(\cW^{\infty}(\gg,f)\big)^G.$$ Therefore a strong generating set for $\big(\cW^{\infty}(\gg,f)\big)^G$ gives rise to a strong generating set for $\cW^k(\gg,f)^G$ for generic values of $k$. Since $\cW^{\infty}(\gg,f)$ is a tensor product of free field algebras, the proof of Theorem \ref{thm:hilbertW} now boils down to proving the strong finite generation of orbifolds of such free field algebras.

\subsection{Orbifolds of general free field algebras}
Orbifolds of the above free field algebras under reductive automorphism groups can be studied in an analogous way to the case of the usual free field algebras, which we discussed in Section \ref{sec:hilbertfreefield}. In particular, we have the following result \cite{CLIV}.

\begin{thm}\label{Hilbert:genfreefield}  Let $\cV = \bigotimes_{i=1}^r \cV_i$, where each $\cV_i$ is one of the free field algebras $\cS_{\text{ev}}(n,k)$, $\cS_{\text{odd}}(n,k)$, $\cO_{\text{ev}}(n,k)$, or $\cO_{\text{odd}}(n,k)$. For any reductive group $G$ of automorphisms of $\cV$, $\cV^G$ is strongly finitely generated.
\end{thm}

First, without loss of generality we may assume that in this decomposition, no two factors are of the same type for fixed $k$. This is due to \eqref{type:sev}, \eqref{type:oev}, \eqref{type:sev}, and \eqref{type:oodd}, which say that tensor products of algebras of the same time can be combined into a single one of this type.

Next, if $G$ is an automorphism group of $\cV  = \bigotimes_{i=1}^r \cV_i$, $G$ must preserve the conformal weights and parity of the generators, so it actually preserves each of the factors $\cV_i$. Therefore $G \subseteq G_1 \otimes \cdots \otimes G_r$, where each $G_i = \text{Aut} \ \cV_i$, which is either an orthogonal or a symplectic group.
Therefore the first step is to work out the structure of $\cV^{\text{Aut} \ \cV}$ when $\cV$ is one of the above standard free field algebras $\cS_{\text{ev}}(n,k)$, $\cS_{\text{odd}}(n,k)$, $\cO_{\text{ev}}(n,k)$, or $\cO_{\text{odd}}(n,k)$.

The following result of \cite{CLIV} is a generalization of Theorem \ref{Hilbert:fullauto}, which is the special case obtained by obtained by taking $k=1$ for $\cS_{\text{ev}}(n,k)$ and $\cO_{\text{odd}}(n,k)$, and taking $k=2$ for $\cS_{\text{odd}}(n,k)$ and $\cO_{\text{ev}}(n,k)$.

\begin{thm} \label{Hilbert:fullautogen} Let $n\geq 1$ be a positive integer.
 \begin{enumerate}
\item For all odd $k\geq 1$, $\cS_{\text{ev}}(n,k)^{Sp(2n)}$ has a minimal strong generating set 
$$\omega^{j} = \frac{1}{2} \sum_{i=1}^n \big(:a^i \partial^{j} b^i: \  - \  :(\partial^{j} a^i )b^i:\big),\qquad j = 1,3,\dots, 2n^2+3n-1 + nk. $$ Since $\omega^{j}$ has weight $k+j$, $\cS_{\text{ev}}(n,k)^{Sp(2n)}$ is of type $$\cW \big(k+1,k+3,\dots, 2n^2+3n -1 +(n+1)k\big).$$ 

\item For all even $k\geq 2$, $\cS_{\text{odd}}(n,k)^{Sp(2n)}$ has a minimal strong generating set 
$$\omega^{j} = \frac{1}{2} \sum_{i=1}^n \big( :a^i \partial^{j} b^i: + :(\partial^j a^i) b^i:\big),\qquad j=0,2,\dots, k(n+1)-k -2.$$ Since $\omega^j$ has weight $k+j$, $\cS_{\text{odd}}(n,k)^{Sp(2n)}$ is of type $$\cW(k, k+2\dots,k(n+1) -2).$$

\item For all odd $k\geq 1$, $\cO_{\text{odd}}(n,k)^{O(n)}$ has a minimal strong generating set 
$$\omega^{j} = \frac{1}{2} \sum_{i=1}^n :\phi^i \partial^{j} \phi^i:,\qquad j=1,3,\dots, (n+1)(k+1) -2 -k.$$ Since $\omega^j$ has weight $k+j$, $\cO_{\text{odd}}(n,k)^{O(n)}$ is of type $$\cW(k+1, k+3,\dots,(n+1)(k+1) -2).$$

\item For all even $k\geq 2$, $\cO_{\text{ev}}(n,k)^{O(n)}$ has a minimal strong generating set 
$$\omega^{j} = \sum_{i=1}^n :a^i \partial^{j} a^i:,\qquad j=0,2,\dots, N,$$ for some $N \geq (n+1) k + n(n+1) -k-2$. Since $\omega^j$ has weight $k+j$, $\cO_{\text{ev}}(n,k)^{O(n)}$ is of type $$\cW(k, k+2,\dots, k+N).$$ We expect that $N$ can be taken to be $ (n+1) k + n(n+1) -k-2$ but we do not prove this.
\end{enumerate}
\end{thm}

The proof of Theorem \ref{Hilbert:fullautogen} is exactly the same as the proof of the corresponding statements for ordinary free field algebras in Section \ref{sec:hilbertfreefield}. In fact, using the realizations of these free field algebras as subalgebras of ordinary free field algebras, the proof can be carried out using the recursive formulas already obtained in \cite{LIV,LV,CLII}. Using Theorem \ref{Hilbert:fullautogen}, the following results are also proven in a similar way to the case of ordinary free field algebras.

\begin{enumerate}
\item If $\cV$ is any of the free field algebras $\cO_{\text{ev}}(n,k)$, $\cS_{\text{ev}}(n,k)$, $\cS_{\text{odd}}(n,k)$, or $\cS_{\text{odd}}(n,k)$, each irreducible $\cV^{\text{Aut}\ \cV}$-submodule of $\cV$ has one-dimensional top space, and is $C_1$-cofinite.

\smallskip

\item If $\cV = \bigotimes_{i=1}^r\cV_i$ is a tensor product of the above free field algebras, $\cV^G$ has a strong generating set that lies in the direct sum of finitely many irreducible modules over the subalgebra $\bigotimes_{i=1}^r \cV_i^{\text{Aut} \ \cV_i}$. This is achieved using Weyl's theorem on polarizations.

\smallskip

\item The above finiteness results combine to yield the strong finite generation of $\cV^G$
\end{enumerate}

\section{Structure theory of affine cosets} \label{sec:cosets} 
The {\it coset construction} is a standard way to construct new vertex algebras from old ones. Given a vertex algebra $\cA$ and a subalgebra $\cV \subseteq \cA$, the {\it coset} or {\it commutant} of $\cV$ in $\cA$, denoted by 
$\text{Com}(\cV,\cA)$, is the subalgebra of elements $a\in\cA$ such that $$[v(z),a(w)] = 0,\ \text{for all} \  v\in\cV.$$ This was introduced by Frenkel and Zhu in \cite{FZ}, generalizing earlier constructions in representation theory \cite{KP} and physics \cite{GKO}, where it was used to construct the unitary discrete series representations of the Virasoro algebra. Equivalently, 
$$a\in \text{Com}(\cV,\cA) \Leftrightarrow\ v_{(n)} a= 0,\  \text{for all}\ v\in\cV,\  \text{and}\ n\geq 0.$$ The {\it double commutant} $\text{Com}(\text{Com}(\cV,\cA),\cA)$ always contains $\cV$ as a subalgebra, If $\cV = \text{Com}( \text{Com}(\cV,\cA),\cA)$, we say that $\cV$ and $\text{Com}(\cV,\cA)$ form a {\it Howe pair} inside $\cA$. Note that if $\cA$ and $\cV$ have Virasoro elements $L^{\cA}$ and $L^{\cV}$, then $\cC = \text{Com}(\cV,\cA)$ has Virasoro element $L^{\cC} = L^{\cA} - L^{\cV}$ as long as $ L^{\cA} \neq L^{\cV}$. Therefore the map $\cV \otimes \cC \hookrightarrow \cA$ is a conformal embedding.

If $\cV$ is an affine vertex algebra $V^k(\gg)$ for some Lie algebra $\gg$, then $\cC$ can be identified with the invariant subalgebra $\cA^{\gg[t]}$ which is annihilated by the Lie subalgebra $\gg[t]$ spanned by the non-negative modes of the generating fields $X^{\xi} \in V^k(\gg)$. In this case, we call $\cC$ an {\it affine coset}. It is natural to ask whether there is a vertex algebra Hilbert problem for affine cosets: if $\cA$ is strongly finitely generated and $\gg$ is reductive, is $\cC$ also strongly finitely generated? Although this problem is difficult to address in a general setting, we will give an affirmative answer for a large class of VOAs that include $\cW^k(\gg,f)$ for any simple finite-dimensional Lie superalgebra $\gg$ and any even nilpotent $f\in\gg$, for generic values of $k$.

As in \cite{CLIII,CLIV}, we shall axiomatize a class of vertex algebras $\cA^k$ that depend algebraically on a parameter $k$, which admit a homomorphism $V^k(\gg)\rightarrow \cA^k$, such that the coset $\cC^k = \text{Com}(V^k(\gg), \cA^k)$ can be studied using the machinery of deformable families. For the moment, we assume $\gg$ is a simple Lie algebra and $\text{dim} \ \gg = r$. We say that such a vertex algebra $\cA^k$ is {\it good} if the following hypotheses are satisfied.

\begin{enumerate}

\item $\cA^k$ comes from a deformable family, that is, a vertex algebra $\cA$ over $F_K$ for some at most countable subset $K\subseteq \mathbb{C}$, such that $$\cA/ (\kappa - \sqrt{k})\cA \cong \cA^k,\  \text{for all}\ \sqrt{k} \notin K.$$

\smallskip

\item The map $V^k(\mathfrak{g}) \rightarrow \cA^k$ is induced by a homomorphism of deformable families $\cV \rightarrow \cA$, where $\cV$ is the algebra introduced in Section \ref{sec:hilbertaffine} satisfying $V^k(\mathfrak{g}) \cong \cV/ (\kappa - \sqrt{k})$.

\smallskip

\item The action of $\gg$ on $\cA$ coming from the zero modes of the generators of $V^k(\gg)$ integrates to an action of a connected Lie group $G$, and $\cA$ decomposes into finite-dimensional $G$-modules.

\smallskip
 
\item The limit $\cA^{\infty} = \lim_{\kappa \rightarrow \infty} \cA$ decomposes as 
$$\cA^{\infty} \cong \cV^{\infty} \otimes \tilde{\cA} \cong \cH(r) \otimes \tilde{\cA},$$ for some vertex subalgebra $\tilde{\cA}\subseteq \cA^{\infty}$.
\end{enumerate}

Under these hypotheses, $G$ acts on $\tilde{\cA}$, and the following result appeared as Theorem 6.10 of \cite{CLIII}: 

\begin{thm} \label{thm:deformcoset} Let $\cA^k$ be a one-parameter vertex algebra satisfying the above hypotheses. Then
\begin{enumerate}
\item The coset $\cC = \text{Com}(\cV, \cA)$ is a deformable family, and 
$$\cC/ (\kappa - \sqrt{k})\cC \cong \cC^k = \text{Com}(V^k(\mathfrak{g}), \cA^k),$$
for generic $k$. In particular, this holds for all $k > - h^{\vee}$.

\smallskip

\item We have an isomorphism \begin{equation} \begin{split} \cC^{\infty} & \cong \text{Com}(\cV^{\infty}, \cV^{\infty} \otimes \tilde{\cA})^G
\\ & \cong \text{Com}(\cH(r), \cH(r) \otimes \tilde{\cA})^G
\\ & \cong \tilde{\cA}^G.
\end{split} \end{equation}
\end{enumerate}
\end{thm}

In particular, a strong generating set for $\tilde{\cA}^G$ gives rise to a strong generating set for $\cC^k$ for generic values of $k$. This theorem continues to hold if the simple Lie algebra $\mathfrak{g}$ is replaced by a reductive Lie algebra, or the sum of a reductive Lie algebra and finitely many copies of $\mathfrak{osp}_{1|2n}$ for integers $n\geq 1$.

\subsection{The case of $\cW$-algebras}

Let $\gg$ be a simple finite-dimensional Lie superalgebra and $f\in \gg$ an even nilpotent element. Suppose that $\cW^k(\gg,f)$ has affine subalgebra $V^{k'}(\ga)$ where the even part of $\ga$ has dimension $d$ and the odd part has dimension $2e$. Note that 
$$\lim_{k'\ra \infty} V^{k'}(\ga) \cong \cO_{\text{ev}}(d,2) \otimes \cS_{\text{odd}}(e,2) = \cH(d) \otimes \cA(e).$$
Then in the decomposition $\cW^{\infty}(\gg,f) \cong \bigotimes_{i=1}^r \cV_i$, we may assume that $\cV_1 \cong  \cO_{\text{ev}}(d,2)$ and $\cV_2 \cong \cS_{\text{odd}}(e,2)$.

Let $\gb \subseteq \ga$ be a reductive Lie subalgebra of dimension $r$, and let 
$$V^{\ell}(\gb) \subseteq V^{k'}(\ga) \subseteq \cW^k(\gg,f),$$ be the corresponding affine subalgebra. 
Write $\cV_1 \cong \ \cO_{\text{ev}}(r,2)  \otimes \cO_{\text{ev}}(d-r,2)$, so that 
\begin{equation} \cW^{\infty}(\gg,f) \cong \cO_{\text{ev}}(r,2)  \otimes \cO_{\text{ev}}(d-r,2) \otimes \bigg(\bigotimes_{i=2}^r \cV_i\bigg).\end{equation}

Then the action of $\gb$ coming from $V^{\ell}(\gb)$ lifts to an action of a connected Lie group $B$ on $\cW^k(\gg,f)$, and $B$ preserves each of the factors $\cO_{\text{ev}}(d-r,2)$ and $\cV_i$ for $i = 2,\dots, r$. 

By specializing Theorem \ref{thm:deformcoset} to this setting, we obtain 
\begin{thm} \label{thm:cosetofw} The coset 
$$\cC^k = \text{Com}(V^{\ell}(\gb), \cW^k(\gg,f)),$$ is a deformable family with limit
$$\cC^{\infty} \cong  \bigg(\cO_{\text{ev}}(d-r,2) \otimes \big(\bigotimes_{i=2}^r \cV_i\big)\bigg)^B.$$ 
\end{thm}

By Theorem \ref{thm:hilbertW}, $\bigg(\cO_{\text{ev}}(d-r,2) \otimes \big(\bigotimes_{i=2}^r \cV_i\big)\bigg)^B$ is strongly finitely generated, so by our general result Corollary \ref{cor:passage}, we obtain

\begin{thm} \label{thm:Wcoset} Let $\cW^k(\gg, f)$ be the $\cW$-algebra associated to a simple Lie (super)algebra $\gg$ and an even nilpotent element $f\in\gg$. For any affine subalgebra $V^{\ell}(\gb) \subseteq \cW^k(\gg, f)$ where $\gb$ is a reductive Lie algebra, the coset
$$\cC^k = \text{Com}(V^{\ell}(\gb), \cW^k(\gg,f)),$$ is strongly finitely generated for generic values of $k$.
\end{thm}

\section{The Gaiotto-Rap\v{c}\'ak conjectures} \label{sec:gaiotto} 

Motivated from physics, Gaiotto and Rap\v{c}\'ak introduced a family of VOAs $Y_{L,M,N}[\psi]$ called  $Y$-algebras \cite{GR}. They considered interfaces of GL-twisted $\mathcal N=4$ supersymmetric 
gauge theories with gauge groups $U(L), U(M), U(N)$. The name \lq\lq $Y$-algebras" comes from the fact that shape of these interfaces is a $Y$, and local operators at the corner of these interfaces are supposed to form a VOA. Also, note that GL here stands for geometric Langlands. These interfaces should satisfy a permutation symmetry which then induces a corresponding symmetry on the associated VOAs. This led \cite{GR} to conjecture a triality of isomorphisms of $Y$-algebras. Let $\psi$ be defined by 
$$\psi = -\frac{\epsilon_2}{\epsilon_1}, \qquad \epsilon_1 + \epsilon_2 +\epsilon_3 =0,$$
and set $$Y^{\epsilon_1, \epsilon_2, \epsilon_3}_{N_1, N_2, N_3} := Y_{N_1, N_2, N_3}[\psi].$$
In this notation, the expected triality symmetry is then
$$Y^{\epsilon_{\sigma(1)}, \epsilon_{\sigma(2)}, \epsilon_{\sigma(3)}}_{N_{\sigma(1)}, N_{\sigma(2)}, N_{\sigma(3)}} \cong Y^{\epsilon_1, \epsilon_2, \epsilon_3}_{N_1, N_2, N_3}, \qquad \text{for} \ \sigma \in S_3.$$
The $Y$-algebras with one label being zero are up to a Heisenberg algebra, certain affine cosets of $\cW$-(super)algebras of type $A$, which we now describe. Recall first that conjugacy classes of nilpotents $f\in \mathfrak{sl}_N$ correspond to partitions of $N$. Set $N = n+m$ for $n\geq 1$, and write 
$$\mathfrak{sl}_{n+m} = \mathfrak{sl}_n \oplus \mathfrak{gl}_m \oplus \bigg(\mathbb{C}^n \otimes (\mathbb{C}^m)^* \bigg)\ \oplus \bigg(( \mathbb{C}^n)^* \otimes \mathbb{C}^m\bigg).$$
Let $f_{n,m} \in \mathfrak{sl}_{n+m}$ be the nilpotent which is principal in $\mathfrak{sl}_n$ and trivial in $\mathfrak{gl}_m$. It corresponds to the hook-type partition $N = n + 1 + \dots + 1$, and the corresponding $\cW$-algebra $\cW^k(\gs\gl_{n+m}, f_{n,m})$ was called a {\it hook-type} $\cW$-algebra in \cite{CLIV}. It is a common generalization of the principal $\cW$-algebra $\cW^k(\gs\gl_n)$ (the case $m=0$), the affine VOA $V^k(\gs\gl_{n+1})$ (the case $n=1$), the subregular $\cW$-algebra $\cW^k(\gs\gl_{n+1}, f_{\text{subreg}})$ (the case $m=1$), and the minimal $\cW$-algebra $\cW^k(\gs\gl_{m+2}, f_{\text{min}})$ (the case $n=2$). 

For convenience, we replace the level $k$ with the {\it critically shifted level} $\psi = k+n+m$, and we set
$$\cW^{\psi}(n,m) := \cW^k(\mathfrak{sl}_{n+m}, f_{n+m}) =  \cW^{\psi - n-m}(\mathfrak{sl}_{n+m}, f_{n+m}).$$

For $m\geq 2$, $\cW^{\psi}(n,m)$ has affine subalgebra $V^{\psi - m-1}(\mathfrak{gl}_m) = \cH \otimes V^{\psi - m-1}(\mathfrak{sl}_m)$, additional 
even generators in weights $2,3,\dots, n$ which commute with the affine subalgebra, together with $2m$ fields in weight $\frac{n+1}{2}$ which transform under $\mathfrak{gl}_m$ as $\mathbb{C}^m \oplus (\mathbb{C}^m)^*$. We also define the case $\cW^{\psi}(0,m)$ separately as follows. 
\begin{enumerate}
\item For $m\geq 2$, $$\cW^{\psi}(0,m) = V^{\psi-m}(\mathfrak{sl}_m) \otimes \cS(m),$$ where $\cS(m)$ is the rank $m$ $\beta\gamma$-system.
\item $\cW^{\psi}(0,1) = \cS(1)$.
\item $\cW^{\psi}(0,0) \cong \mathbb{C}$.
\end{enumerate}

Next, we define a similar class of $\cW$-superalgebras. For $n+m\geq 2$, $n\geq 1$, and $n\neq m$, write
$$\mathfrak{sl}_{n|m} = \mathfrak{sl}_n \oplus \mathfrak{gl}_m \oplus \bigg(\mathbb{C}^n \otimes (\mathbb{C}^m)^* \bigg)\ \oplus \bigg(( \mathbb{C}^n)^* \otimes \mathbb{C}^m\bigg).$$ Let $f_{n|m} \in \mathfrak{sl}_{n|m}$ be the even nilpotent element which is principal in $\mathfrak{sl}_n$ and trivial in $\mathfrak{gl}_m$. Let $\psi = k + n-m$, and let 
$$\cV^{\psi}(n,m) = \cW^k(\mathfrak{sl}_{n|m}, f_{n|m}) =  \cW^{\psi - n+m}(\mathfrak{sl}_{n|m}, f_{n|m}).$$
In the case $n = m \geq 2$, we need a slightly different definition to get a simple algebra: we define $\cV^{\psi}(n,n) = \cW^{\psi}(\mathfrak{psl}_{n|n}, f_{n|n})$.

For $m\geq 2$, $\cV^{\psi}(n,m)$ has affine subalgebra 
\begin{equation} \begin{split} & V^{-\psi - m+1}(\mathfrak{gl}_m),\qquad m \neq n,
\\ & V^{-\psi - n+1}(\mathfrak{sl}_n), \qquad m = n. \end{split} \end{equation}

$\cV^{\psi}(n,m)$ has additional even generators in weights $2,3,\dots, n$, together with $2m$ odd fields in weight $\frac{n+1}{2}$ transforming under $\mathfrak{gl}_m$ as $\mathbb{C}^m \oplus (\mathbb{C}^m)^*$.

We define the case $\cV^{\psi}(0,m)$ separately as follows. 
\begin{enumerate}
\item For $m\geq 2$, $$\cV^{\psi}(0,m) = V^{-\psi-m}(\mathfrak{sl}_m) \otimes \cE(m),$$ where $\cE(m)$ is the rank $m$ $bc$-system.
\item $\cV^{\psi}(1,1) = \cA(1)$, rank one symplectic fermion algebra.
\item $\cV^{\psi}(0,1) = \cE(1)$.
\item $\cV^{\psi}(0,0) \cong \cV^{\psi}(1,0) \cong \mathbb{C}$.
\end{enumerate}
Consider the affine cosets
\begin{equation*} \begin{split} & \cC^{\psi}(n, m) = \text{Com}(V^{\psi - m-1}(\mathfrak{gl}_m), \cW^{\psi}(n, m)),
\\ & \cD^{\psi}(n, m) = \text{Com}(V^{-\psi -m+1}(\mathfrak{gl}_m), \cV^{\psi}(n, m)), \qquad n\neq m,
\\ & \cD^{\psi}(n, n) = \text{Com}(V^{-\psi -n+1}(\mathfrak{sl}_n), \cV^{\psi}(n, n))^{U(1)}. \end{split}\end{equation*} Note that $\cV^{\psi}(n, n)$ has an action of $U(1)$ by outer automorphisms, which then acts on the coset $\text{Com}(V^{-\psi -n+1}(\mathfrak{sl}_n), \cV^{\psi}(n, n))$ as well, and it is necessary to take the $U(1)$ orbifold. It is not difficult to check that $\cC^{\psi}(n,m)$ has central charge $$c  =  -\frac{(n \psi  - m - n -1) (n \psi - \psi - m - n +1 ) (n \psi +  \psi  -m - n)}{(\psi -1) \psi},$$ and that in both cases $n\neq m$ and $n=m$, $\cD^{\psi}(n,m)$ has central charge 
$$c  = -\frac{(n \psi + m - n -1) (n \psi  - \psi + m - n +1) (n \psi + \psi +m - n )}{(\psi -1) \psi}.$$
The main result of \cite{CLIV} is the following
\begin{thm} (Theorem 1.1, \cite{CLIV}) \label{thm:triality} Let $n \geq m \geq 0$ be integers. We have isomorphisms of one-parameter VOAs 
\begin{equation*} \cD^{\psi}(n, m) \cong \cC^{\psi^{-1}}(n-m, m) \cong \cD^{\psi'}(m,n),\end{equation*}
where $\psi'$ is defined by $\displaystyle \frac{1}{\psi} + \frac{1}{\psi'} = 1$.
\end{thm}

In our previous notation, we have
\begin{equation}
\begin{split} 
Y_{0, M, N}[\psi] &= \cC^\psi(N-M, M) \otimes \cH , \qquad M \leq N, \\
Y_{0, M, N}[\psi] &= \cC^{-\psi+1}(M-N, N)  \otimes \cH, \qquad M > N, \\
Y_{L, 0, N}[\psi] &= \cD^\psi(N, L) \otimes \cH, \\
Y_{L, M, 0}[\psi] &= \cD^{-\psi+1}(M, L) \otimes \cH.
\end{split}
\end{equation}

By definition one has $Y_{0, N, M}[\psi] \cong Y_{0, M, N}[1-\psi]$ for $N\neq M$. We also have $\cC^\psi(0, M) \cong  \cC^{1-\psi}(0, M)$ and hence this statement also holds for $N=M$. Clearly also $Y_{L, 0, N}[\psi] \cong Y_{L, N, 0}[1-\psi]$. Combining this with the isomorphisms in Theorem \ref{thm:triality}, we obtain
\begin{equation}
\begin{split}
Y_{0, M, N}[\psi] &\cong Y_{0, N, M}[1-\psi] \cong Y_{M, 0, N}[\psi^{-1}] \cong Y_{M, N, 0}[1-\psi^{-1}] \\
&\cong Y_{N, 0, M}[(1-\psi)^{-1}] \cong Y_{N, M, 0}[(1-\psi^{-1})^{-1}]. 
\end{split}
\end{equation}
This is the Gaoitto-Rap\v{c}\'ak triality conjecture in type $A$ with one of the labels zero. This result unifies and vastly generalize several well-known results in the theory of $\cW$-algebras. First, Feigin-Frenkel duality says that the principal $\cW$-algebra of a simple Lie algebra $\gg$ at level $\psi - h^\vee$ is isomorphic to the principal $\cW$-algebra of the dual Lie algebra ${}^L\gg$ at level $\psi' - {}^Lh^\vee$ where $\ell \psi \psi'=1$ and $\ell$ is the lacety of $\gg$ \cite{FFII}. The special case 
$\cD^\psi(n, 0) \cong \cC^{\psi^{-1}}(n, 0)$ of Theorem \ref{thm:triality} result reproduces Feigin-Frenkel duality in type $A$. We therefore regard the family of isomorphisms $\cD^{\psi}(n, m) \cong \cC^{\psi^{-1}}(n-m, m)$ as of {\it Feigin-Frenkel type}. Also, Theorem \ref{ACLmain} in type $A$ is the case $\cD^{\psi}(n, 0) \cong \cD^{\psi'}(0, n)$ with $\psi'$ defined by $\frac{1}{\psi} +\frac{1}{\psi'} =1$. We therefore regard the family of isomorphisms $\cD^{\psi}(n, m) \cong \cD^{\psi'}(m, n)$ as of {\it coset realization type}.

\subsection{The case of types $B$, $C$, and $D$} 
Gaiotto and Rap\v{c}\'ak also define a notion of orthosymplectic $Y$-algebras in \cite{GR}. These are analogous to the above algebras can be identified with affine cosets of certain families of $\cW$-(super)algebras of types $B$, $C$, and $D$. First, let $\gg$ be either $\mathfrak{so}_{2n+1}$, $\mathfrak{sp}_{2n}$, $\mathfrak{so}_{2n}$, or $\mathfrak{osp}_{n|2r}$. We have a decomposition $$\mathfrak{g} = \mathfrak{a} \oplus \mathfrak{b} \oplus \rho_{\mathfrak{a}} \otimes \rho_{\mathfrak{b}}.$$
Here $\mathfrak{a}$ and $\mathfrak{b}$ are Lie sub(super)algebras of $\mathfrak{g}$, and $\rho_{\mathfrak{a}}$, $\rho_{\mathfrak{b}}$ transform as the standard representations of $\mathfrak{a}$, $\mathfrak{b}$. In the case $\gg = \mathfrak{osp}_{n|2r}$, $\rho_{\ga}$ and $\rho_{\gb}$ have the same parity, which can be even or odd.

Next, let $f_{\mathfrak{b}} \in \mathfrak{g}$ be the nilpotent element which is principal in $\mathfrak{b}$ and trivial in $\mathfrak{a}$, and let $\cW^k(\gg,f_{\gb})$ be the corresponding $\cW$-(super)algebra. In all cases, $\cW^k(\gg,f_{\gb})$ is of type
$$\cW \bigg(1^{\text{dim}\ \mathfrak{a}}, 2,4,\dots, 2m, \bigg(\frac{{d_{\mathfrak{b}}} + 1}{2} \bigg)^{d_{\mathfrak{a}}}\bigg).$$ In particular, there are $\text{dim}\ \mathfrak{a}$ fields in weight $1$ which generate an affine vertex (super)algebra of $\ga$. The fields in weights $2,4,\dots, 2m$ are even and are invariant under $\mathfrak{a}$. The $d_{\mathfrak{a}}$ fields in weight $\frac{d_{\mathfrak{b}} + 1}{2}$ can be even or odd, and transform as the standard $\mathfrak{a}$-module.

For $n,m\geq 0$ we have the following cases where $\mathfrak{b} = \mathfrak{so}_{2m+1}$.
\begin{enumerate}
\item {\bf Case 1B}: $\ \ \displaystyle \mathfrak{g} =\mathfrak{so}_{2n+2m+2},\quad \mathfrak{a} = \mathfrak{so}_{2n+1}$. 
\item {\bf Case 1C}: $\ \ \displaystyle\mathfrak{g} =\mathfrak{osp}_{2m+1|2n}, \quad \mathfrak{a} = \mathfrak{sp}_{2n}$.
\item{\bf Case 1D}: $\ \ \displaystyle\mathfrak{g}  =\mathfrak{so}_{2n+2m+1},  \quad \mathfrak{a} = \mathfrak{so}_{2n}$.
\item{\bf Case 1O}:  $\ \ \displaystyle\mathfrak{g} =\mathfrak{osp}_{2m+2|2n},  \quad \mathfrak{a} = \mathfrak{osp}_{1|2n}$.
\end{enumerate}
For $n\geq 0$ and $m\geq 1$ we have the following cases where $\mathfrak{b} = \mathfrak{so}_{2m}$.
\begin{enumerate}
\item {\bf Case 2B}:  $\ \ \displaystyle\mathfrak{g} =\mathfrak{osp}_{2n+1|2m},\quad \mathfrak{a} = \mathfrak{so}_{2n+1}$.
 \item{\bf Case 2C}: $\ \ \displaystyle\mathfrak{g} =\mathfrak{sp}_{2n+2m}, \quad \mathfrak{a} = \mathfrak{sp}_{2n}$.
\item{\bf Case 2D}:  $\ \ \displaystyle\mathfrak{g} =\mathfrak{osp}_{2n|2m}, \quad \mathfrak{a} = \mathfrak{so}_{2n}$.
\item{\bf Case 2O}:  $\ \ \displaystyle\mathfrak{g} =\mathfrak{osp}_{1|2n+2m},  \quad \mathfrak{a} = \mathfrak{osp}_{1 | 2 n}$. 
\end{enumerate}
 As above, we always replace $k$ with the critically shifted level $\psi = k+h^{\vee}$. For $i = 1,2$ and $X = B,C,D,O$, we denote the corresponding $\cW$-algebras by $\cW^{\psi}_{iX}(n,m)$. Note that for $i=1$ and $m=0$, the nilpotent $f_{\gb} \in \gg$ is trivial, so we have
\begin{enumerate}
\item $\cW^{\psi}_{1B}(n,0) = V^{\psi-2n}(\gs\go_{2n+2})$,
\item $\cW_{1C}^{\psi}(n,0) = V^{\psi +2n+1}(\mathfrak{osp}_{1|2n})$,
\item $\cW_{1D}^{\psi}(n,0) = V^{\psi -2n +1}(\gs\go_{2n+1})$,
\item $\cW_{1O}^{\psi}(n,0) = V^{\psi +2n}(\mathfrak{osp}_{2|2n})$.
\end{enumerate}

For $i=2$ and $m=0$, we define the algebras $\cW^{\psi}_{2X}(n,0)$ in a different way so that our results hold uniformly for all $n,m\geq 0$. Recall that $\cF(m)$ and $\cS(m)$ denote the rank $m$ free fermion algebra and $\beta\gamma$-system, respectively. First, we define
\begin{equation}  \cW^{\psi}_{2B}(n,0) = \left\{
\begin{array}{ll}
V^{-2 \psi - 2 n + 1}(\gs\go_{2n+1}) \otimes \cF(2n+1) & n \geq 1,
\smallskip
\\  \cF(1) & n =0.
\\ \end{array} 
\right.
\end{equation} Since $\cF(2n+1)$ has an action of $L_1(\gs\go_{2n+1})$, for $n\geq 1$ $\cW_{2B}^{\psi}(n,0)$ has a diagonal action of $V^{-2 \psi - 2 n + 2}(\gs\go_{2n+1})$.
Next, we define
\begin{equation}  \cW^{\psi}_{2C}(n,0) = \left\{
\begin{array}{ll}
 V^{\psi-n-1} (\gs\gp_{2n}) \otimes \cS(n) & n \geq 1,
\smallskip
\\  \mathbb{C} & n =0.
\\ \end{array} 
\right.
\end{equation}
Since $\cS(n)$ has an action of $L_{-1/2}(\gs\gp_{2n})$, $\cW_{2C}^{\psi}(n,0)$ has a diagonal action of $V^{\psi-n-3/2}(\gs\gp_{2n})$ for $n\geq 1$. 
Next, we define
\begin{equation}  \cW^{\psi}_{2D}(n,0) = \left\{
\begin{array}{ll}
 V^{- 2 \psi - 2 n +2}(\gs\go_{2n}) \otimes \cF(2n)  & n \geq 1,
\smallskip
\\  \mathbb{C} & n =0.
\\ \end{array} 
\right.
\end{equation}
 Since $\cF(2n)$ has an action of $L_1(\gs\go_{2n})$, $\cW_{2D}^{\psi}(n,0)$ has a diagonal action of $V^{- 2 \psi - 2 n +3}(\gs\go_{2n})$ for $n\geq 1$.
 Finally, we define
\begin{equation}  \cW^{\psi}_{2O}(n,0) = \left\{
\begin{array}{ll}
  V^{\psi-n-1/2}(\mathfrak{osp}_{1|2n}) \otimes \cS(n) \otimes \cF(1)  & n \geq 1,
\smallskip
\\   \cF(1) & n =0.
\\ \end{array} 
\right.
\end{equation}
It is well known that $\cS(n) \otimes \cF(1)$ has an action of $L_{-1/2}(\go\gs\gp_{1|2n})$, so $\cW_{2O}^{\psi}(n,0)$ has a diagonal action of $V^{\psi -n-1}(\go\gs\gp_{1|2n})$ for $n\geq 1$. 

In the cases $i = 1,2$ and $X = C$, we denote the corresponding affine cosets by $\cC^{\psi}_{iC}(n,m)$. More precisely, we have
\begin{equation}   \cC^{\psi}_{1C}(n,m) = \left\{
\begin{array}{ll}
 \text{Com}(V^{-\psi/2 - n - 1/2}(\gs\gp_{2n}), \cW_{1C}^{\psi}(n,m))
& m \geq 1,\  n \geq 1,
\smallskip
\\  \text{Com}(V^{-\psi/2 - n - 1/2}(\gs\gp_{2n}), V^{\psi +2n+1}(\mathfrak{osp}_{1|2n}))& m =0, \ n\geq 1,
 \smallskip
\\  \cW^{\psi-2m+1}(\gs\go_{2m+1})& m \geq 1, \ n = 0,
\smallskip
\\ \mathbb{C} & m = n = 0.
\\ \end{array} 
\right.
\end{equation}
\begin{equation}   \cC^{\psi}_{2C}(n,m) = \left\{
\begin{array}{ll}
\text{Com}(V^{\psi-n-3/2}(\gs\gp_{2n}), \cW_{2C}^{\psi}(n,m))
 & m \geq 1,\  n \geq 1,
 \smallskip
\\ \text{Com}(V^{\psi-n-3/2}(\gs\gp_{2n}), V^{\psi-n-1} (\gs\gp_{2n}) \otimes \cS(n)) & m =0, \ n\geq 1,
 \smallskip
\\ \cW^{\psi -m-1}(\gs\gp_{2m})  & m \geq 1, \ n = 0,
\smallskip
\\ \mathbb{C} & m = n = 0.
\\ \end{array} 
\right.
\end{equation}

In the cases $X = B$, for all $n\geq 0$, the zero mode action of $\gs\go_{2n+1}$ integrates to an action of $SO(2n+1)$ which in fact lifts to $O(2n+1)$, so there is an additional action of $ \mathbb{Z}/2\mathbb{Z}$, and $\cC^{\psi}_{iB}(n,m)$ denotes the $ \mathbb{Z}/2\mathbb{Z}$-orbifold of the affine coset.
\begin{equation}  \cC^{\psi}_{1B}(n,m) = \left\{
\begin{array}{ll}
\text{Com}(V^{\psi-2n}(\mathfrak{so}_{2n+1}), \cW^{\psi}_{1B}(n,m))^{ \mathbb{Z}/2\mathbb{Z}} & m \geq 1,\  n \geq 1,
\smallskip
\\  \text{Com}(V^{\psi-2n}(\mathfrak{so}_{2n+1}), V^{\psi-2n}(\gs\go_{2n+2}))^{ \mathbb{Z}/2\mathbb{Z}} & m =0, \ n\geq 1,
\smallskip
\\ \cW^{\psi -2m}(\gs\go_{2m+2})^{ \mathbb{Z}/2\mathbb{Z}} & m \geq 1, \ n = 0,
\smallskip
\\ \cH(1)^{ \mathbb{Z}/2\mathbb{Z}} & m = n = 0.
\\ \end{array} \right.\end{equation}
\begin{equation}   \cC^{\psi}_{2B}(n,m) = \left\{
\begin{array}{ll}
 \text{Com}(V^{-2 \psi - 2 n + 2}(\gs\go_{2n+1}), \cW_{2B}^{\psi}(n,m) )^{ \mathbb{Z}/2\mathbb{Z}}
 & m \geq 1,\  n \geq 1,
 \smallskip
\\ \text{Com}(V^{-2 \psi - 2 n + 2}(\gs\go_{2n+1}), V^{-2 \psi - 2 n + 1}(\gs\go_{2n+1}) \otimes \cF(2n+1))^{ \mathbb{Z}/2\mathbb{Z}} & m =0, \ n\geq 1,
 \smallskip
\\ \cW^{\psi-m-1/2}(\go\gs\gp_{1|2m})^{ \mathbb{Z}/2\mathbb{Z}} & m \geq 1, \ n = 0,
\smallskip
\\ \cF(1)^{ \mathbb{Z}/2\mathbb{Z}} & m = n = 0.
\\ \end{array} \right.\end{equation}
Similarly, for $X = O$, for all $n\geq 0$, there is an additional action of $ \mathbb{Z}/2\mathbb{Z}$, and $\cC^{\psi}_{iO}(n,m)$ denotes the $ \mathbb{Z}/2\mathbb{Z}$-orbifold of the affine coset.
\begin{equation}   \cC^{\psi}_{1O}(n,m) = \left\{
\begin{array}{ll}
\text{Com}(V^{- \psi/2 -n}(\go\gs\gp_{1|2n}), \cW_{1O}^{\psi}(n,m))^{ \mathbb{Z}/2\mathbb{Z}}
 & m \geq 1,\  n \geq 1,
 \smallskip
\\  \text{Com}(V^{- \psi/2 -n}(\go\gs\gp_{1|2n}), V^{\psi+2n}(\mathfrak{osp}_{2|2n})) & m =0, \ n\geq 1,
 \smallskip
\\ \cW^{\psi-2m}(\gs\go_{2m+2})^{ \mathbb{Z}/2\mathbb{Z}} & m \geq 0, \ n = 0,
\smallskip
\\ \cH(1)^{ \mathbb{Z}/2\mathbb{Z}} & m = n = 0.
\\ \end{array} \right.\end{equation}
\begin{equation}  \cC^{\psi}_{2O}(n,m) = \left\{
\begin{array}{ll}
\text{Com}(V^{\psi -n-1}(\go\gs\gp_{1|2n}), \cW_{2O}^{\psi}(n,m) )^{ \mathbb{Z}/2\mathbb{Z}} & m \geq 1,\  n \geq 1,
 \smallskip
\\ \text{Com}(V^{\psi -n-1}(\go\gs\gp_{1|2n}), V^{\psi-n-1/2}(\mathfrak{osp}_{1|2n}) \otimes \cS(n) \otimes \cF(1))^{ \mathbb{Z}/2\mathbb{Z}} & m =0, \ n\geq 1,
\smallskip
\\ \cW^{\psi -m -1/2}(\go\gs\gp_{1|2m})^{ \mathbb{Z}/2\mathbb{Z}}  & m \geq 1, \ n = 0,
\smallskip
\\ \cF(1)^{ \mathbb{Z}/2\mathbb{Z}} & m = n = 0.
\\ \end{array} \right. \end{equation}

Finally, in the cases $X = D$, for all $n\geq 1$, there is an additional action of $ \mathbb{Z}/2\mathbb{Z}$, and $\cC^{\psi}_{iD}(n,m)$ denotes the $ \mathbb{Z}/2\mathbb{Z}$-orbifold of the affine coset. For $n=0$ there is no additional  $ \mathbb{Z}/2\mathbb{Z}$-action.
\begin{equation}  \cC^{\psi}_{1D}(n,m) = \left\{
\begin{array}{ll}
\text{Com}(V^{\psi-2n+1}(\mathfrak{so}_{2n}), \cW_{1D}^{\psi}(n,m) )^{ \mathbb{Z}/2\mathbb{Z}}
 & m \geq 1,\  n > 1,
 \smallskip
 \\ \text{Com}(\cH(1),  \cW^{\psi -2m-1}(\gs\go_{2m+3}, f_{\text{subreg}}))^{ \mathbb{Z}/2\mathbb{Z}}
 & m \geq 1,\  n = 1,
 \smallskip
 \\ \text{Com}(V^{\psi-2n+1}(\mathfrak{so}_{2n}), V^{\psi-2n+1}(\gs\go_{2n+1})^{ \mathbb{Z}/2\mathbb{Z}} & m =0, \ n\geq 1,
 \smallskip
\\ \cW^{\psi-2m+1}(\gs\go_{2m+1}) & m \geq 1, \ n = 0,
\smallskip
\\ \mathbb{C} & m = n = 0.
\\ \end{array} 
\right.
\end{equation}
\begin{equation}   \cC^{\psi}_{2D}(n,m) = \left\{
\begin{array}{ll}
 \text{Com}(V^{- 2 \psi - 2 n +3}(\gs\go_{2n}), \cW_{2D}^{\psi}(n,m) )^{ \mathbb{Z}/2\mathbb{Z}}
 & m \geq 1,\  n > 1,
 \smallskip
\\  \text{Com}(\cH(1), \cW^{\psi - m}(\mathfrak{osp}_{2|2m}) )^{ \mathbb{Z}/2\mathbb{Z}}
 & m \geq 1,\  n = 1,
 \smallskip
\\ \text{Com}(V^{- 2 \psi - 2 n +3}(\gs\go_{2n}), V^{- 2 \psi - 2 n +2}(\gs\go_{2n}) \otimes \cF(2n))^{ \mathbb{Z}/2\mathbb{Z}} & m =0, \ n\geq 1,
 \smallskip
\\ \cW^{\psi -m-1}(\gs\gp_{2m})  & m \geq 1, \ n = 0,
\smallskip
\\ \mathbb{C} & m = n = 0.
\\ \end{array} 
\right.
\end{equation}

The main result of \cite{CLV} is that there are four families of trialities among the eight families of algebras $\cC^{\psi}_{iX}(n,m)$.
\begin{thm} (Theorem 4.1, \cite{CLV}) \label{thm:trialityBCD} For all $m, n \in \mathbb{N}$ with $m\geq n$, we have the following isomorphisms of one-parameter VOAs.
\begin{equation} \label{2b2b2o}  \cC^{\psi}_{2B}(n,m) \cong  \cC^{\psi'}_{2O}(n,m-n) \cong \cC^{\psi''}_{2B}(m,n), \qquad \psi' = \frac{1}{4\psi},  \qquad \frac{1}{\psi}  + \frac{1}{\psi''} = 2,
 \end{equation}
\begin{equation} \label{1c1c2c} \cC^{\psi}_{1C}(n,m) \cong \cC^{\psi'}_{2C}(n,m-n) \cong  \cC^{\psi''}_{1C}(m,n), \qquad \psi'= \frac{1}{2\psi}, \qquad \frac{1}{\psi} + \frac{1}{\psi''} = 1,
\end{equation}
\begin{equation}  \label{2d1d1o} \cC^{\psi}_{2D}(n,m) \cong  \cC^{\psi'}_{1D}(n,m-n ) \cong \cC^{\psi''}_{1O}(m,n-1),\qquad \psi' = \frac{1}{2\psi}, \qquad \frac{1}{2\psi} + \frac{1}{\psi''} = 1,
\end{equation}
\begin{equation} \label{1o1b2d} \cC^{\psi}_{1O}(n,m) \cong \cC^{\psi'}_{1B}(n,m-n) \cong  \cC^{\psi''}_{2D}(m+1,n) ,\qquad \psi' = \frac{1}{\psi},  \qquad  \frac{1}{\psi}  + \frac{1}{2\psi''} = 1.
\end{equation}
\end{thm}

In the case $n=0$, the first isomorphism $\cC^{\psi}_{1C}(0,m) \cong \cC^{\psi'}_{2C}(0,m)$ of \eqref{1c1c2c} is just Feigin-Frenkel duality in types $B$ and $C$, since $$\cC^{\psi}_{1C}(0,m) = \cW^{\psi-2m+1}(\mathfrak{so}_{2m+1}),\qquad \cC^{\psi'}_{2C}(0,m) \cong \cW^{\psi' - m - 1}(\gs\gp_{2m}).$$
Similarly, the isomorphism $\cC^{\psi}_{2D}(0,m) \cong  \cC^{\psi'}_{1D}(0,m)$ in \eqref{2d1d1o} again reproduces Feigin-Frenkel duality in types $B$ and $C$, since
$$ \cC^{\psi}_{2D}(0,m) = \cW^{\psi-m-1}(\mathfrak{sp}_{2m}),\qquad \cC^{\psi'}_{1D}(0,m) = \cW^{\psi'-2m+1}( \mathfrak{so}_{2m+1}).$$
Likewise, the isomorphism $\cC^{\psi}_{1O}(0,m) \cong \cC^{\psi'}_{1B}(0,m)$ in \eqref{1o1b2d} is just the $\mathbb{Z}_2$-invariant part of Feigin-Frenkel duality in type $D$, since 
$$ \cC^{\psi}_{1O}(0,m) = \cW^{\psi-2m}(\mathfrak{so}_{2m+2})^{\mathbb{Z}_2}, \qquad \cC^{\psi'}_{1B}(0,m) =  \cW^{\psi'-2m}(\mathfrak{so}_{2m+2})^{\mathbb{Z}_2}.$$ As explained in \cite{CLV}, it is straightforward to extend this to an isomorphism $\cW^{\psi-2m}(\mathfrak{so}_{2m+2}) \cong \cW^{\psi'-2m}(\mathfrak{so}_{2m+2})$, which is the full duality.
Finally, the isomorphism $\cC^{\psi}_{2B}(0,m) \cong \cC^{\psi'}_{2O}(0,m)$ in \eqref{2b2b2o} is the $\mathbb{Z}_2$-invariant part of Feigin-Frenkel duality for principal $\cW$-algebras of $\go\gs\gp_{1|2m}$, since
$$ \cC^{\psi}_{2B}(0,m) = \cW^{\psi-m-1/2}(\go\gs\gp_{1|2m})^{\mathbb{Z}_2} , \qquad \cC^{\psi'}_{2O}(0,m) = \cW^{\psi' -m -1/2}(\go\gs\gp_{1|2m})^{\mathbb{Z}_2}.$$

Next, the special case 
$$\cC^{\psi}_{2D}(n,0) \cong \cC^{\psi''}_{1O}(0,n-1), \qquad \frac{1}{2\psi} + \frac{1}{\psi''} =1, \qquad n\geq 2$$ of \eqref{2d1d1o} provides a new proof of the $\mathbb{Z}/2\mathbb{Z}$-invariant part of Theorem \ref{ACLmain} in type $D$.
The special case 
$$\cC^{\psi}_{2B}(n,0)  \cong \cC^{\psi''}_{2B}(0,n),\qquad \frac{1}{\psi}  + \frac{1}{\psi''} = 2$$ of \eqref{2b2b2o} recovers the $\mathbb{Z}/2\mathbb{Z}$-invariant part of the coset realization of the principal $\cW$-superalgebra of $\go\gs\gp_{1|2n}$, which was also proven in \cite{CGe} using the methods of \cite{ACL}. As above, both isomorphisms are easily extended to give new proofs of the full coset realizations.

More importantly, the special case 
 $$\cC^{\psi}_{1C}(n,0) \cong  \cC^{\psi''}_{1C}(0,n), \qquad \frac{1}{\psi} + \frac{1}{\psi''} =1$$ of \eqref{1c1c2c} provides a new coset realization of principal $\cW$-algebras of type $B$ (and type $C$ via Feigin-Frenkel duality), since
 \begin{equation*} \begin{split}  & \cC^{\psi}_{1C}(n,0) = \text{Com}(V^k(\gs\gp_{2n}), V^k(\go\gs\gp_{1|2n})),\qquad k = - \frac{1}{2} (\psi +2n +1).\\ & \cC^{\psi''}_{1C}(0,n) = \cW^{\psi''-2n+1}(\mathfrak{so}_{2n+1}).
\end{split}\end{equation*} This is quite different from the coset realizations of $\cW^k(\gg)$ for $\gg$ simply-laced given by Theorem \ref{ACLmain} since it involves affine vertex superalgebras.

\subsection{Outline of proof}

The proof of Theorems  \ref{thm:triality} and \ref{thm:trialityBCD} require all of the general machinery we have developed so far, and can be divided into the following steps.

\begin{enumerate}
\item Using Theorem \ref{thm:Wcoset}, which can be made constructive in these examples, we can give explicit minimal strong generating sets for the type $A$ cosets $\cC^{\psi}(n,m)$ and $\cD^{\psi}(n,m)$ as one-parameter VOAs. Aside from the extreme cases $\cC^{\psi}(0,0)$, $\cC^{\psi}(1,0)$, $\cC^{\psi}(2,0)$, and $\cD^{\psi}(0,0)$, $\cD^{\psi}(0,1)$, $\cD^{\psi}(1,0)$, $\cD^{\psi}(2,0)$, for all other values of $n,m$, $\cC^{\psi}(n,m)$ and $\cD^{\psi}(n,m)$ are all of type $\cW(2,3,\dots, N)$ for some $N$. Similarly, in types $B$, $C$, and $D$, Theorem \ref{thm:Wcoset} can be used to show that aside from some extreme cases, all the algebras $\cC^{\psi}_{iX}(n,m)$ are of type $\cW(2,4,\dots, 2N)$ for some $N$.

\smallskip

\item VOAs of type $\cW(2,3,\dots, N)$ for some $N$ satisfying some mild hypotheses, are classified as one-parameter quotients of the universal $2$-parameter VOA $\cW(c,\lambda)$ of type $\cW(2,3,\dots)$ which was constructed by the second author in \cite{LVI}. This is a VOA defined over the polynomial ring $\mathbb{C}[c,\lambda]$, and its one-parameter quotients are in bijection with a certain family of curves in the parameter space $\mathbb{C}^2$ called truncation curves. Similarly, VOAs of type $\cW(2,4,\dots, 2N)$ satisfying mild hypotheses are classified as one-parameter quotients of the universal two-parameter even spin VOA $\cW^{\text{ev}}(c,\lambda)$ of type $\cW(2,4,\dots)$, which was constructed by Kanade and the second author in \cite{KL}. Again, there is a bjiection between such one-parameter quotients, and a family of truncation curves in $\mathbb{C}^2$.

\smallskip

\item We know that $\cW^{\psi}(n,m)$ is an extension of $V^{\psi-m-1}(\gg\gl_m) \otimes \cC^{\psi}(n,m)$ by $2m$ fields in weight $\frac{n+1}{2}$ which transform under $\gg\gl_m$ as $\mathbb{C}^m \oplus (\mathbb{C}^m)^*$. The existence of such a VOA extension allows us to compute the defining ideal for the truncation curve realizing $\cC^{\psi}(n,m)$ as a one-parameter quotient of $\cW(c,\lambda)$. The same procedure works for $\cD^{\psi}(n,m)$. Theorem \ref{thm:triality} follows from our explicit formulas for these ideals, except in the above extreme cases, which are easy to verify separately. Similarly, we can explicitly compute the truncation curves realizing all the algebras $\cC^{\psi}_{iX}(n,m)$ as one-parameter quotients of $\cW^{\text{ev}}(c,\lambda)$, and Theorem \ref{thm:trialityBCD} is a consequence of our formulas.

\end{enumerate}

Since the approach in type $A$ is similar to the others but somewhat less involved, we will only discuss the proof of Theorem \ref{thm:triality} for the rest of this section.

\subsection{Step 1: Strong generating types of $\cC^{\psi}(n,m)$ and  $\cD^{\psi}(n,m)$}

First, $\cW^{\psi}(n,m) = \cW(\gs\gl_{n+m}, f_{n+m})$ has free field limit
\begin{equation} \cW^{\text{free}}(n,m) = \lim_{\psi \ra \infty} \cW^{\psi}(n,m) = \cO_{\text{ev}}(m^2,2) \otimes \tilde{\cW},\end{equation}
where $\cO_{\text{ev}}(m^2,2) = \cH(m^2)$ is the just the rank $m^2$ Heisenberg algebra coming from the affine subalgebra. Moreover,  if $n$ is even,
$$\tilde{\cW} \cong \big(\bigotimes_{i=2}^n \cO_{\text{ev}}(1,2i)\big)  \otimes \cS_{\text{ev}}(m, n+1),$$
and if $n$ is odd,
$$\tilde{\cW} \cong \big(\bigotimes_{i=2}^n \cO_{\text{ev}}(1,2i)\big) \otimes \cO_{\text{ev}}(2m, n+1).$$ Here $\cO_{\text{ev}}(1,2i)$ is the algebra generated by $\omega^i$ for $i = 2,\dots, n$, and the fields $\{G^{\pm, i}|\ i = 1,\dots, m\}$ generate $\cS_{\text{ev}}(m, n+1)$ or $\cO_{\text{ev}}(2m, n+1)$ when $n$ is even or odd, respectively. 

It then follows from Theorem \ref{thm:cosetofw} that $$\lim_{\psi \ra \infty} \cC^{\psi}(n,m) \cong \tilde{\cW}^{GL(m)},$$ and that the strong generating type of $\cC^{\psi}(n,m)$ is the same as that of $\tilde{\cW}^{GL(m)}$ for generic $\psi$. Since $GL(m)$ acts trivially on each factor $\cO_{\text{ev}}(1,2i)$, we have
\begin{equation} \begin{split} \lim_{\psi \ra \infty} \cC^{\psi}(n,m) & \cong \bigg( \big(\bigotimes_{i=2}^n \cO_{\text{ev}}(1,2i)\big)\otimes \cS_{\text{ev}}(m, n+1)\bigg)^{GL(m)} \\ & \cong \big(\bigotimes_{i=2}^n \cO_{\text{ev}}(1,2i)\big) \otimes \big(\cS_{\text{ev}}(m, n+1)\big)^{GL(m)},\end{split} \end{equation} for $n$ even, and
\begin{equation} \begin{split} \lim_{{\psi} \ra \infty} \cC^{\psi}(n,m) & \cong \bigg(  \big(\bigotimes_{i=2}^n \cO_{\text{ev}}(1,2i)\big)\otimes \cO_{\text{ev}}(2m, n+1)\bigg)^{GL(m)} 
\\ & \cong  \big(\bigotimes_{i=2}^n \cO_{\text{ev}}(1,2i)\big) \otimes \big(\cO_{\text{ev}}(2m, n+1)\big)^{GL(m)}.\end{split} \end{equation} for $n$ odd.
Since the generator of $\cO_{\text{ev}}(1,2i)$ has weight $i$, $\big(\bigotimes_{i=2}^n \cO_{\text{ev}}(1,2i)\big)$ is of type $\cW(2,3,\dots,n)$. Therefore to determine the strong generating type of $\cC^{\psi}(n,m)$, it suffices to find the strong generating type of $\big(\cS_{\text{ev}}(m, n+1)\big)^{GL(m)}$ and $\big(\cO_{\text{ev}}(2m, n+1)\big)^{GL(m)}$. This can be carried out using the first and second fundamental theorem of invariant theory of $GL(m)$ in a similar way to the proof of Theorem \ref{Hilbert:fullautogen}.

\begin{thm}  Let $n$ be a positive integer. 
\begin{enumerate}
\item  For all odd $k\geq 1$, $\cS_{\text{ev}}(n,k)^{GL(n)}$ has a minimal strong generating set $$\omega^{j} = \sum_{i=1}^n :a^i \partial^{j} b^i: ,\qquad j = 0,1,\dots,  n(n+1) + n k -1. $$ Since $\omega^{j}$ has weight $k+j$, $\cS_{\text{ev}}(n,k)^{Sp(2n)}$ is of type $$\cW \big(k,k+1,k+2,\dots, n(n+1) + k(n+1) -1 \big).$$ 
\item For all even $k\geq 2$,  $\cO_{\text{ev}}(2n,k)^{GL(n)}$ has a minimal strong generating set $$\omega^j = \sum_{i=1}^n :e^i \partial^j f^i:,\qquad j =0,1,\dots, n(n+1) + n k -1.$$ Since $\omega^{j}$ has weight $k+j$, $\cO_{\text{ev}}(2n,k)^{GL(n)}$ is of type $$\cW(k,k+1,\dots, n(n+1) + k(n+1) -1).$$ 
\end{enumerate} \end{thm}

An immediate consequence is 
\begin{cor} \label{cor:stronggencnm} For $m\geq 1$ and $n\geq 0$, $\cC^{\psi}(n,m)$ is of type $$\cW(2,3,\dots, (m+1)(m+n+1)-1)$$ as a one-parameter VOA; equivalently, this holds for generic values of $\psi$. \end{cor}

In a completely parallel way, one can determine the strong generating type of $\cD^{\psi}(n,m)$. By passing to the free field limit, this boils down to determining the strong generating type of the $GL(n)$-orbifolds of $\cS_{\text{odd}}(n,k)$ and $\cO_{\text{odd}}(2n,k)$. Here $GL(n)$ is regarded as the subgroup of the full automorphism groups $Sp(2n)$ and $O(2n)$ of $\cS_{\text{odd}}(n,k)$ and $\cO_{\text{odd}}(2n,k)$, respectively, such that the standard module $\mathbb{C}^{2n}$ of $Sp(2n)$ and $O(2n)$ decomposes under $GL(n)$ as $\mathbb{C}^n \oplus (\mathbb{C}^n)^*$.

\begin{thm} Let $n$ be a positive integer.
\begin{enumerate}
\item For all even $k\geq 2$, $\cS_{\text{odd}}(n,k)^{GL(n)}$ has a minimal strong generating set $$\omega^{j} = \sum_{i=1}^n :a^i \partial^{j} b^i: ,\qquad j = 0,1,\dots, n k -1. $$ Since $\omega^{j}$ has weight $k+j$, $\cS_{\text{odd}}(n,k)^{GL(n)}$ is of type $$\cW \big(k,k+1,k+2,\dots, k(n+1) -1 \big).$$
\item For all odd $k\geq 1$,  $\cO_{\text{odd}}(2n,k)^{GL(n)}$ has a minimal strong generating set $$\omega^j = \sum_{i=1}^n :e^i \partial^j f^i:,\qquad j =0,1,\dots, n k -1.$$ Since $\omega^{j}$ has weight $k+j$, $\cO_{\text{odd}}(2n,k)^{GL(n)}$ is of type $\cW(k,k+1,\dots, k(n+1) -1)$.
\end{enumerate}  \end{thm}

We obtain 
\begin{cor} For $m\geq 1$ and $n\geq 1$, $\cD^{\psi}(n,m)$ is of type $$\cW(2,3,\dots, (m+1)(n+1)-1)$$ as a one-parameter VOA; equivalently, this holds for generic values of $\psi$.
\end{cor}

\subsection{Step 2: Classification of VOAs of type $\cW(2,3,\dots,N)$.}
The existence and uniqueness of a two-parameter VOA $\cW_{\infty}[\mu]$ which interpolates between the principal $\cW$-algebras $\cW^k(\gs\gl_n)$ for all $n\geq 2$ was conjectured for many years in the physics literature \cite{YW,BK,B-H,BS,GGI,GGII,ProI,ProII,PRI,PRII}, and was recently proven by the second author in \cite{LVI}. Recently, $\cW_{\infty}[\mu]$ has become important in the duality between of two-dimensional conformal field theories and higher spin gravity on three-dimensional Anti-de-Sitter space \cite{GGI,GGII}. The structure constants in $\cW_{\infty}[\mu]$ are algebraic functions of the central charge $c$ and the parameter $\mu$. If we set $\mu = n$, there is a truncation at weight $n+1$ such that the simple quotient is isomorphic to $\cW^k(\gs\gl_n)$ as a one-parameter VOA. In the quasi-classical limit, the existence of a Poisson VOA of type $\cW(2,3,\dots)$ which interpolates between the classical $\cW$-algebras of $\gs\gl_n$ for all $n$, has been known for many years \cite{KZ,KM}. 

The algebra $\cW_{\infty}[\mu]$ acquires better properties if it is tensored with a rank one Heisenberg algebra $\cH$ to obtain the universal two-parameter $\cW_{1+\infty}$-algebra. This VOA is closely related to a number of other algebraic structures that arise in very different contexts. For example, up to a suitable completion, its associative algebra of modes is isomorphic to the Yangian of $\widehat{\gg\gl_1}$ \cite{AS,MO,Ts}, as well as the algebra ${\bf SH}^c$ defined in \cite{SV} as a certain limit of degenerate double affine Hecke algebras of $\mathfrak{gl}_n$. This identification was used by Schiffmann and Vasserot \cite{SV} to define the action of the principal $\cW$-algebra of $\gg\gl_r$ on the equivariant cohomology of the moduli space of $U_r$-instantons.

In \cite{LVI}, a different parameter $\lambda$ was used, which is related to $\mu$ by 
$$\lambda =  \frac{(\mu-1) (\mu+1)}{(\mu-2) (3 \mu^2  - \mu -2+ c (\mu + 2))},$$ and we denoted the universal algebra by $\cW(c,\lambda)$. This choice is not canonical but is natural because $\cW(c,\lambda)$ is then defined over the polynomial ring $\mathbb{C}[c,\lambda]$. Instead of using either the primary strong generating fields, or the quadratic basis of \cite{ProI}, our strong generators are defined as follows. We begin with the primary weight $3$ field $W^3$, normalized so that $W^3_{(5)} W^3 = \frac{c}{3} 1$, and we define the remaining fields recursively by $$W^i= W^3_{(1)} W^{i-1} \qquad i\geq 4.$$ With this choice, the rich connections between the representation theory of $\cW_{\infty}[\mu]$ and the combinatorics of box partitions are no longer apparent. Our choice has the advantage, however, that the recursive behavior of the OPE algebra is more transparent. This is essential in the proof of existence and uniqueness of this algebra given in \cite{LVI}.

In addition to $\cW^k(\gs\gl_n)$, $\cW(c,\lambda)$ admits many other one-parameter quotients as well. In fact, any one-parameter VOA of type $\cW(2,3,\dots, N)$ for some $N$ satisfying mild hypotheses, arises as such as quotient, so $\cW(c,\lambda)$ can be viewed as a classifying object for such VOAs. The simple one-parameter quotients are in bijection with a family of plane curves which are known as  truncation curves.

The main idea is to show that by imposing all Jacobi identities among the generators, the structure constants in the OPEs of $L(z) W^i(w)$ and $W^i(z) W^j(w)$ are uniquely and consistently determined as polynomials in $c$ and $\lambda$, for all $i$ and $j$. This \lq\lq bootstrap" method of determining the OPEs among the generators of a VOA by imposing Jacobi identities has appeared in a number of papers in the physics literature including \cite{KauWa,Bow,B-V,Horn}. 
For example, imposing all Jacobi identities of type $(W^i, W^j, W^k)$ for $i+j+k \leq 9$, the following OPEs are forced to hold:
\begin{equation} \label{Winf:standard1} \begin{split} W^3(z) W^3(w) & \sim  \frac{c}{3}(z-w)^{-6} + 2 L(w)(z-w)^{-4} + \partial L(w)(z-w)^{-3} \\ & + W^4(w)(z-w)^{-2}
 + \bigg(\frac{1}{2} \partial W^4 -\frac{1}{12} \partial^3 L\bigg)(w)(z-w)^{-1},\end{split} \end{equation}
\begin{equation} \label{Winf:standard2} \begin{split}  L(z) W^4(w) & \sim 3 c (z-w)^{-6} + 10 L(w)(z-w)^{-4} + 3 \partial L(w)(z-w)^{-3} \\ & + 4 W^4(w)(z-w)^{-2}
 + \partial W^4(w)(z-w)^{-1}, \end{split} \end{equation}
 \begin{equation} \label{Winf:standard3} \begin{split}   L(z) W^5(w) & \sim \bigg(185-80 \lambda (2+c)\bigg) W^3(z)(z-w)^{-4}
\\ & +  \bigg(55-16 \lambda (2+c)\bigg) \partial W^3(z)(z-w)^{-3}
\\ & + 5 W^5(w)(z-w)^{-2} + \partial W^5(w)(z-w)^{-1},\end{split} \end{equation}
\begin{equation} \label{Winf:standard4} \begin{split} W^3(z) W^4(w) & \sim   \bigg(31-16 \lambda (2+c)\bigg) W^3(w)(z-w)^{-4}  +\frac{8}{3} \bigg(5-2\lambda (2+c)\bigg) \partial W^3(w)(z-w)^{-3}
\\ & + W^5(w)(z-w)^{-2}  +\bigg(\frac{2}{5}
\partial W^5 + \frac{32}{5} \lambda :L\partial W^3: -\frac{48}{5} \lambda :(\partial L) W^3: 
\\ & + \frac{2}{15}\big(-5+2 \lambda (1-c)\big) \partial^3 W^3\bigg) (w)(z-w)^{-1}.
\end{split} \end{equation}

The key difficulty is to show that for all $n\geq 9$, imposing Jacobi relations of type $(W^i, W^j, W^k)$ for $i+j+k \leq n+2$ uniquely and consistently determines all OPEs 
$$W^a(z) W^b(w), \ \text{for}\ a+b \leq n.$$
We obtain a {\it nonlinear Lie conformal algebra} over the ring $\mathbb{C}[c,\lambda]$ with generators $\{L, W^i|\ i \geq 3\}$ in the language of De Sole and Kac \cite{DSKI}, and $\mathcal{W}(c,\lambda)$ is the {\it universal enveloping VOA}.

The VOA $\cW(c,\lambda)$ has a conformal weight grading $$\cW(c,\lambda) = \bigoplus_{n\geq 0} \cW(c,\lambda)[n],$$ where each $\cW(c,\lambda)[n]$ is a free $\mathbb{C}[c,\lambda]$-module and $\cW(c,\lambda)[0] \cong \mathbb{C}[c,\lambda]$. There is a symmetric bilinear form on $\cW(c,\lambda)[n]$ given by
$$\langle ,  \rangle_n : \cW(c,\lambda)[n] \otimes_{\mathbb{C}[c,\lambda]} \cW(c,\lambda)[n] \ra \mathbb{C}[c,\lambda],\qquad \langle \omega, \nu \rangle_n = \omega_{(2n-1)} \nu.$$ The level $n$ Shapovalov determinant $\text{det}_n \in \mathbb{C}[c,\lambda]$ is just the determinant of this form. It is nonzero for all $n$, which implies that $\cW(c,\lambda)$ is simple as a VOA over $\mathbb{C}[c,\lambda]$.

Let $p$ be an irreducible factor of $\text{det}_{N+1}$ and let $I = (p) \subseteq \mathbb{C}[c,\lambda] \cong \cW(c,\lambda)[0]$ be the corresponding ideal. Consider the quotient
\begin{equation} \label{winfoneparam} \cW^I(c,\lambda) = \cW(c,\lambda) / I \cdot \cW(c,\lambda),\end{equation} where $I$ is regarded as a subset of the weight zero space $\cW(c,\lambda)[0] \cong \mathbb{C}[c,\lambda]$, and $I \cdot \cW(c,\lambda)$ denotes the VOA ideal generated by $I$.  This is a VOA over the ring $\mathbb{C}[c,\lambda]/I$, which is no longer simple. It contains a singular vector $\omega$ in weight $N+1$, which lies in the maximal proper ideal $\cI\subseteq \cW^I(c,\lambda)$ graded by conformal weight. If $p$ does not divide $\text{det}_{m}$ for any $m<N+1$, $\omega$ will have minimal weight among elements of $\cI$. Often, $\omega$ has the form \begin{equation} \label{sing:intro} W^{N+1} - P(L, W^3,\dots, W^{N-1}),\end{equation} possibly after passing to a localization of the ring $\mathbb{C}[c,\lambda] / I$, where $P$ is a normally ordered polynomial in the fields $L,W^3,\dots,$ $W^{N-1}$, and their derivatives. If this is the case, there will exist relations in the simple graded quotient $\cW_I(c,\lambda) :=\cW^I(c,\lambda) / \cI$ of the form
$$W^m = P_m(L, W^3, \dots, W^N),$$ for all $m \geq N+1$ expressing $W^m$ in terms of $L, W^3,\dots, W^N$ and their derivatives. Then $\cW_I(c,\lambda)$ will be of type $\cW(2,3,\dots, N)$. Conversely, any one-parameter VOA $\cW$ of type $\cW(2,3,\dots, N)$ for some $N$ satisfying mild hypotheses, is isomorphic to $\cW_I(c,\lambda)$ for some $I = (p)$ as above, possibly after localizing. The corresponding variety $V(I) \subseteq\mathbb{C}^2$ is called the {\it truncation curve} for $\cW$.

Note that if $I=(p)$ for some irreducible $p$, then $\cW^I(c,\lambda)$ and $\cW_I(c,\lambda)$ are one-parameter VOAs since $\mathbb{C}[c,\lambda] /(p)$ has Krull dimension $1$. We also consider $\cW^I(c,\lambda)$ when $I\subseteq \mathbb{C}[c,\lambda]$ is a maximal ideal, which has the form $I = (c- c_0, \lambda- \lambda_0)$ for some $c_0, \lambda_0\in \mathbb{C}$. Then $\cW^I(c,\lambda)$ and its quotients are ordinary VOAs over $\mathbb{C}$. Given maximal ideals $I_0 = (c- c_0, \lambda- \lambda_0)$ and $I_1 = (c - c_1, \lambda - \lambda_1)$, let $\cW_0$ and $\cW_1$ be the simple quotients of $\cW^{I_0}(c,\lambda)$ and $\cW^{I_1}(c,\lambda)$, respectively. There is an easy criterion for $\cW_0$ and $\cW_1$ to be isomorphic. We must have $c_0 = c_1$, and if $c_0 \neq 0$ or $-2$, there is no restriction on $\lambda_0$ and $\lambda_1$. For all other values of the central charge, we must have $\lambda_0 = \lambda_1$. This criterion implies that aside from the coincidences at $c=0$ and $-2$, all other pointwise coincidences among simple one-parameter quotients of $\cW(c,\lambda)$ must correspond to intersection points on their truncation curves; see Corollary 10.3 of \cite{LVI}.

Often, a VOA $\cC^k$ arising as a coset of the form $\text{Com}(V^k(\gg), \cA^k)$ for some VOA $\cA^k$, can be identified with a one-parameter quotient $\cW_I(c,\lambda)$ for some $I$. Here $k$ is regarded as a formal variable, and we have a homomorphism
\begin{equation} \label{isoviaparametrization} L \mapsto \tilde{L}, \qquad W^3 \mapsto \tilde{W}^3, \qquad c\mapsto c(k), \qquad \lambda\mapsto \lambda(k). \end{equation} Here $\{\tilde{L}, \tilde{W}^3\}$ are the standard generators of $\cC^k$, where $(\tilde{W}^3)_{(5)} \tilde{W}^3 = \frac{c(k)}{3}$, and $k \mapsto (c(k), \lambda(k))$ is a rational parametrization of the curve $V(I)$. 

There are two subtleties that need to be mentioned. First, for a complex number $k_0$, the specialization $\cC^{k_0} := \cC^k / (k-k_0) \cC^k$ typically makes sense for all $k_0 \in \mathbb{C}$, even if $k_0$ is a pole of $c(k)$ or $\lambda(k)$. At these points, $\cC^{k_0}$ need not be obtained as a quotient of $\cW^I(c,\lambda)$. Second, even if $k_0$ is not a pole of $c(k)$ or $\lambda(k)$, the specialization $\cC^{k_0}$ can be a proper subalgebra of the \lq\lq honest" coset $\text{Com}(V^{k_0}(\gg), \cA^{k_0})$, even though generically these coincide. By Corollary 6.7 of \cite{CLII}, under mild hypotheses that are satisfied in all our examples, if $\gg$ is simple this can only occur for rational numbers $k_0 \leq -h^{\vee}$, where $h^{\vee}$ is the dual Coxeter number of $\gg$. Additionally, if $\gg$ contains an abelian subalgebra $\gh$, the coset becomes larger at the levels where the corresponding Heisenberg fields become degenerate, since it now contains these fields.

\subsection{Step 3: Explicit realization of $\cC^{\psi}(n,m)$ and $\cD^{\psi}(n,m)$ as quotients of $\cW(c,\lambda)$}
Recall that $\cC^{\psi}(n,m)$ and $\cD^{\psi}(n,m)$ are of types  $\cW(2,3,\dots, (m+1)(m+n+1)-1)$ and $\cW(2,3,\dots, (m+1)(n+1)-1)$, respectively. In this subsection, we outline the proof of the following result of \cite{CLIV}.

\begin{thm} \label{GRexplicitC} For $m \geq 1$ and $n\geq 0$, and for $m=0$ and $n\geq 3$, $\cC^{\psi}(n,m) \cong \cW_{I_{n,m}}(c,\lambda)$, where the ideal $I_{n,m} \subseteq \mathbb{C}[c,\lambda]$ is described explicitly via the parametrization 
$$\Phi_{\cC,n,m}: \mathbb{C} \rightarrow V(I_{n,m}), \qquad \Phi_{\cC,n,m}(\psi) = \big(c(\psi), \lambda(\psi)\big).$$
Here \begin{equation} \begin{split} c(\psi) &=  -\frac{(n \psi  - m - n -1) (n \psi - \psi - m - n +1 ) (n \psi +  \psi  -m - n)}{(\psi -1) \psi},
\\ \ 
\\ \lambda(\psi) & = -\frac{(\psi-1) \psi}{(n \psi - n - m -2) (n \psi - 2 \psi - m - n +2 ) (n \psi + 2 \psi  -m - n )}.
\end{split} \end{equation}
\end{thm}

We have an analogous result for $\cD^{\psi}(n,m)$, which is proven in the same way.

\begin{thm} \label{GRexplicitD} For $m \geq 1$ and $n\geq 1$, and for $m=0$ and $n\geq 3$, $\cD^{\psi}(n,m) \cong \cW_{J_{n,m}}(c,\lambda)$, where
the ideal $J_{n,m}$ is described explicitly via the parametrization 
$$\Phi_{\cD,n,m}: \mathbb{C} \rightarrow V(J_{n,m}), \qquad \Phi_{\cD,n,m}(\psi) = \big(c(\psi), \lambda(\psi)\big).$$
Here \begin{equation} \begin{split} c(\psi) & = -\frac{(n \psi + m - n -1) (n \psi  - \psi + m - n +1) (n \psi + \psi +m - n )}{(\psi -1) \psi},
\\ \ 
\\ \lambda(\psi) & = -\frac{(\psi-1) \psi}{(n \psi + m - n -2) (n \psi - 2 \psi + m - n +2 ) (n \psi + 2 \psi +m - n)}.
\end{split} \end{equation}
\end{thm}

For $n\geq m$, these formulas are easily seen to satisfy 
$$ \Phi_{\cD, n, m}(\psi)  =  \Phi_{\cC, n-m,m} (\psi^{-1}) = \Phi_{\cD,m,n}(\psi'),\qquad   \frac{1}{\psi} + \frac{1}{\psi'} = 1.$$
Therefore if $m \geq 1$ and $n\geq 1$, or if $m=0$ and $n\geq 3$, $\cC^{\psi}(n,m)$ and $\cD^{\psi}(n,m)$ arise as $\cW(c,\lambda)$ quotients, and the isomorphisms in Theorem \ref{thm:triality} are an immediate consequence of Theorems \ref{GRexplicitC} and \ref{GRexplicitD}.  In the extreme case $$\cD^\psi(2, 0)  \cong \cC^{\psi^{-1}}(2, 0) \cong \cD^{\psi'}(0, 2),$$ the first isomorphism is {\it Feigin-Frenkel duality} for Virasoro algebra, and second is {\it coset realization} of Virasoro algebra. The remaining cases \begin{equation} \begin{split} & \cD^\psi(1, 0)  \cong \cC^{\psi^{-1}}(1,0) \cong \cD^{\psi'}(0, 1),
\\ & \cD^\psi(0, 0)  \cong \cC^{\psi^{-1}}(0,0) \cong \cD^{\psi'}(0, 0),\end{split} \end{equation} hold trivially because all these vertex algebras are just $\mathbb{C}$.

From now on, we only discuss the proof of Theorem \ref{GRexplicitC}; the argument for Theorem \ref{GRexplicitD} is similar. The first difficulty is that even though $\cC^{\psi}(n,m)$ is of type $\cW(2,3,\dots, (m+1)(m+n+1)-1)$, it is not immediately apparent that it arises as a one-parameter quotient of $\cW(c,\lambda)$ because we don't know yet that is is generated by the weights $2$ and $3$ fields. So instead, we consider the subalgebra $\tilde{\cC}^{\psi}(n,m) \subseteq \cC^{\psi}(n,m)$ generated by $L$ and $W^3$, which is normalized as usual. A priori, $\tilde{\cC}^{\psi}(n,m)$ need not be all of $\cC^{\psi}(n,m)$ and also need not be simple but since $\cC^{\psi}(n,m)$ is at worst an extension of $\tilde{\cC}^{\psi}(n,m)$, it follows that $\cW^{\psi}(n,m)$ is an extension of $V^{\psi - m-1}(\mathfrak{gl}_m) \otimes \tilde{\cC}^{\psi}(n,m)$. Moreover, by character considerations it is apparent that $\tilde{\cC}^{\psi}(n,m)$ is of type $\cW(2,3,\dots, t)$ for some $t \leq (m+1)(m+n+1)-1$, and hence is a quotient of $\cW^{I_{n,m}}(c,\lambda)$ for some ideal $I_{n,m}$.

Recall first that $\cW^{\psi}(n,m)$ has even primary fields $\{G^{\pm,i}|\ i = 1,\dots,m\}$ of weight $\frac{n+1}{2}$ in $\cW^{\psi}(n,m)$ transforming as $\mathbb{C}^m \oplus (\mathbb{C}^m)^*$ under $\mathfrak{gl}_m$. Letting $L^{\cW}$ denote the Virasoro field in $\cW^{\psi}(n,m)$, this means that 
$$L^{\cW}(z) G^{\pm, i}(w) \sim \frac{n+1}{2} G^{\pm,i}(w)(z-w)^{-2} + \partial G^{\pm,i}(w)(z-w)^{-1}.$$ 
Without loss of generality, $\{G^{\pm,i}\}$ can also be assumed primary for action of $V^{\psi - m-1}(\mathfrak{gl}_m)$. In particular, all OPEs between the generators of $V^{\psi - m-1}(\mathfrak{gl}_m)$ and $\{G^{\pm,i}\}$ have only a first-order pole, and this is just determined by the above action of $\gg\gl_m$.

Let $\{L, W^i\}$ be standard generators for $\tilde{\cC}^{\psi}(n,m)$, and recall that $L^{\cW} = L^{\mathfrak{gl}_m} + L$ where $L^{\gg\gl_m}$ is the Sugawara Virasoro field for $V^{\psi - m-1}(\mathfrak{gl}_m)$. Since the OPEs $L^{\mathfrak{gl}_m}(z) G^{\pm,i}(w)$ are completely determined by the action of $\gg\gl_m$, it follows that
the OPE $$L(z) G^{\pm,i}(w)$$ is completely determined from this data as well.

Let $G = G^{+,1}$ to be highest-weight vector for $\mathbb{C}^m$, and consider the OPE
\begin{equation} \label{BW3G}
\begin{split} W^3(z) G (w) & \sim a_0 G (w)(z-w)^{-3} + \big(a_1 \partial G + \dots  \big)(w)(z-w)^{-2} 
\\ & + \big(a_2 :LG: + a_3  \partial^2 G + \dots \big)(w) (z-w)^{-1},
\end{split} \end{equation}
Here $a_0, a_1,a_2,a_3$ are undetermined constants, and the omitted expressions are not needed. We will see that imposing certain Jacobi identities of type $(L, W^3, G)$ determines $a_1, a_2, a_3$ in terms of $a_0$.

First, we impose 
\begin{equation} \begin{split} & L_{(2)} (W^3_{(1)} G) - W^3_{(1)} (L_{(2)} G) -   (L_{(0)} W^3)_{(3)} G   -2 (L_{(1)} W^3)_{(2)} G \\ & - (L_{(2)} W^3)_{(1)} G = 0.\end{split} \end{equation}
This has weight $\frac{n+1}{2}$, and hence the left side is a scalar multiple of $G$. Using the OPE relations between $V^{\psi - m-1}(\mathfrak{gl}_m)$ and $G$ as well as the known formula for $L(z) G(w)$, together with \eqref{BW3G}, we get the following relation among the variables $a_0,a_1,a_2,a_3$:
\begin{equation} \label{Ba0a1a2first} -3 a_0 + a_1 + a_1 n - \frac{(m^2-1) a_1}{m (\psi-1)} - \frac{a_1 (m + n)}{m (n \psi -m - n)} = 0.\end{equation}
Next, we impose 
\begin{equation} \begin{split} & L_{(3)} (W^3_{(0)} G) - W^3_{(0)} (L_{(3)} G) -   (L_{(0)} W^3)_{(3)} G   -3 (L_{(1)} W^3)_{(2)} G \\ &  - 3(L_{(2)} W^3)_{(1)} G -(L_{(3)} W^3)_{(0)} G = 0.\end{split} \end{equation}
This gives us another relation among the variables $a_0,a_1,a_2,a_3$:
\begin{equation} \label{Ba0a1a2second} -6 a_0 + 2 a_2 + 3 a_3 + \frac{a_2 c}{2} + n (2 a_2  + 3 a_3)  -\frac{(m^2-1) (2 a_2 + 3 a_3)}{m (\psi - 1)} -  \frac{(2 a_2 + 3a_3) (m + n)}{m (n \psi -m - n)} =0.\end{equation}
 Finally, we impose 
\begin{equation}  L_{(2)} (W^3_{(0)} G) - W^3_{(0)} (L_{(2)} G) -   (L_{(0)} W^3)_{(2)} G   -2 (L_{(1)} W^3)_{(1)} G  - (L_{(2)} W^3)_{(0)} G = 0.\end{equation}
 This has weight $\frac{n+3}{2}$, but if we extract just the coefficient of $\partial G$, we get another relation
 \begin{equation} \label{Ba0a1a2third} -4 a_1 + 3 a_2 + 4 a_3 + 2 a_3 n - \frac{2 (m^2-1) a_3}{m (\psi - 1)}- \frac{2 a_3 (m + n)}{m (n \psi -m - n)}.\end{equation}
 Solving \eqref{Ba0a1a2first}, \eqref{Ba0a1a2second}, and \eqref{Ba0a1a2third}, we obtain
 \begin{equation} \label{Ba0a1a2} \begin{split} a_0 = & \frac{(n \psi - m - n -2) (n \psi - m - n -1) (n \psi + \psi -m - n) (n \psi  + 2 \psi -m - n)}{6 (\psi-1)^2 (n \psi -m - n)^2} a_3,
\\  a_1 = & \frac{(n \psi - m - n -2) (n \psi + 2 \psi -m - n )}{2 (\psi -1) (n \psi -m - n )} a_3,
\\ a_2 = &-\frac{2 \psi}{(\psi -1) (n \psi -m - n)} a_3.\end{split} \end{equation}

Next, we have
\begin{equation} \label{BW4W5G} 
\begin{split} & W^4(z) G(w) \sim b_0 G(w)(z-w)^{-4} + \cdots,
\\ & W^5(z) G(w) \sim b_1 G(w)(z-w)^{-5} + \cdots,\end{split} \end{equation} for some constants $b_0, b_1$. By imposing four more Jacobi identities, the constants $a_3,b_0, b_1$ are determined up to a sign, and the parameter $\lambda$ in $\cW(c,\lambda)$ is uniquely determined.

We begin with 
\begin{equation} \begin{split} &  W^3_{(3)} (W^3_{(1)} G) - W^3_{(1)} (W^3_{(3)} G) - (W^3_{(0)} W^3)_{(4)} G -3 (W^3_{(1)} W^3)_{(3)} G \\ & -3 (W^3_{(2)} W^3)_{(2)} G  - (W^3_{(3)} W^3)_{(1)} G =0.\end{split} \end{equation} This has weight $\frac{n+1}{2}$, and is therefore a scalar multiple of $G$. Using the OPE relations \eqref{Winf:standard1}-\eqref{Winf:standard4} together with the above data and \eqref{BW4W5G}, we compute this scalar to obtain the following relation
\begin{equation} \label{Bgenjac:1}
1 + 3 a_0 a_1 - b_0 + n - \frac{m^2-1}{m (\psi-1)} - \frac{m + n}{m (n \psi  -m - n)}=0.
\end{equation}
Next, we impose 
\begin{equation} 
\begin{split} W^3_{(4)} (W^3_{(0)} G) & - W^3_{(0)} (W^3_{(4)} G) 
- (W^3_{(0)} W^3)_{(4)} G - 4 (W^3_{(1)} W^3)_{(3)} G  -6 (W^3_{(2)} W^3)_{(2)} G  \\ & - 4 (W^3_{(3)} W^3)_{(1)} G -  (W^3_{(4)} W^3)_{(0)} G =0.\end{split}\end{equation}  Again, this has weight $\frac{n+1}{2}$ and is a scalar multiple of $G$, so we obtain
\begin{equation} \label{Bjac:2}
1 + 6 a_0 (a_2 + 2 a_3) - 2 b_0 + n - \frac{m^2-1}{m (\psi -1)} - 
\frac{m + n}{m (n \psi  -m - n)} = 0.
\end{equation}
Next we impose \begin{equation} W^3_{(0)} (W^4_{(5)} G) - W^4_{(5)} (W^3_{(0)} G)  - (W^3_{(0)} W^4)_{(5)} G = 0,\end{equation} which yields
\begin{equation} \label{Bgenjac:3} \begin{split} & \frac{1}{2}  \bigg(-40 a_3 b_0 + 5 a_0 \big((2 + c) \lambda - 16\big) + 4 b_1 \bigg)  -  \frac{8 a_2 (m n \psi + m \psi - m^2 - m n  - m +1)}{m (\psi -1)} 
 \\ & -  a_2  (3 c + 8 b_0) + \frac{8a_2(m + n)}{m (n \psi  -m - n)} = 0.\end{split} \end{equation}
Finally, we impose
\begin{equation} W^3_{(1)} (W^4_{(4)} G)  - W^4_{(4)} (W^3_{(1)} G)  - (W^3_{(0)} W^4)_{(5)} G - (W^3_{(1)} W^4)_{(4)} G,\end{equation} which yields
\begin{equation} \label{Bgenjac:4}
-8 a_1 b_0 + 5 a_0 \big((2 + c) \lambda-16 \big) + 2 b_1 =0.\end{equation}
Substituting the values of $a_0, a_1, a_2$ in terms of $a_3$ given by \eqref{Ba0a1a2} into the equations \eqref{Bgenjac:1}-\eqref{Bgenjac:4}, and solving for for $a_3, b_0, b_1, \lambda$ yields a unique solution for $b_0$ and $\lambda$, and a unique solution up to sign for $a_3$ and $b_1$. In particular, expressing everything in terms of the parameter $\psi$, the formula for $\lambda$ is exactly the generator of the ideal $I_{n,m}$ appearing earlier. We remark that the sign ambiguity in $a_0,b_1$ reflects $ \mathbb{Z}/2\mathbb{Z}$-symmetry of $\cW(c,\lambda)$ and does not affect isomorphism type.

We have thus shown that the simple quotient $\tilde{\cC}_{\psi}(n,m)$ of $\tilde{\cC}^{\psi}(n,m)$ is isomorphic to $\cW_{I_{n,m}}(c,\lambda)$. The final step is show that $\tilde{\cC}^{\psi}(n,m) = \cC^{\psi}(n,m)$, or equivalently, that $\cC^{\psi}(n,m)$ is generated by the weights $2$ and $3$ fields. For this purpose, we will find certain {\it coincidences}, or nontrivial isomorphisms, between the simple quotient $\tilde{\cC}_{\psi}(n,m)$ and the simple principal $\cW$-algebras $\cW_r(\mathfrak{sl}_s)$ which correspond to the curves $V(I_{s,0})$.

\begin{lemma} \label{coincindences:typeA} For $s\geq 3$, $m\geq 1$, and $n\geq 0$, we have isomorphisms of simple VOAs 
\begin{equation} \tilde{\cC}_{\psi}(n,m) \cong \cW_r(\gs\gl_s),\qquad \psi = \frac{m + n + s}{n} ,\qquad r =-s+ \frac{m + s}{m + n + s}.\end{equation}
\end{lemma} 

\begin{proof} This is immediate from the fact that the curves $V(I_{n,m})$ and $V(I_{s,0})$ intersect at the point $(c,\lambda)$ given by
$$c = -\frac{( s-1) (n s -m - s) (m + n + s + n s)}{(m + s) (m + n + s)}, \quad \lambda =  \frac{(m + s) (m + n + s)}{(s-2) (2 m + 2 s - n s) (2 m + 2 n + 2 s + n s)},$$ which can be checked using the formulas for the defining ideals.
\end{proof}

We now recall that by Theorem 3.8 of \cite{CLIV}, for $\psi$ sufficiently large and $r,s,\psi$ are related as above, $\cW^r(\gs\gl_s)$ has a singular vector in weight $(m+1)(m+n+1)$ and no singular vector in lower weight. Therefore $\cW_r(\gs\gl_s)$ cannot truncate below weight $(m+1)(m+n+1)$, and will truncate to an algebra of type $\cW(2,3,\dots, (m+1)(m+n+1)-1)$ if and only if this singular vector is a decoupling relation for the field $W^{(m+1)(m+n+1)}$. Therefore in view of Lemma \ref{coincindences:typeA}, $\tilde{\cC}_{\psi}(n,m)$ contains the fields in weights $2,3,\dots, (m+1)(m+n+1)-1$. 

The universal algebra $\tilde{\cC}^{\psi}(n,m)$ specialized at this value of $\psi$ cannot truncate below weight $(m+1)(m+n+1)$, and therefore the same holds for the one-parameter algebra $\tilde{\cC}^{\psi}(n,m)$. Since $\cC^{\psi}(n,m)$ is already known to be of type $\cW(2,3,\dots, (m+1)(m+n+1) -1)$, and $\tilde{\cC}^{\psi}(n,m)$ is a subalgebra of $\cC^{\psi}(n,m)$ containing all the strong generating fields, we must have $\tilde{\cC}^{\psi}(n,m) = \cC^{\psi}(n,m)$. This completes the proof of Theorem \ref{GRexplicitC}. The proof of Theorem \ref{GRexplicitD} follows exactly the same strategy, and combining them completes the proof of Theorem \ref{thm:triality}.

Finally, the proof of Theorem \ref{thm:trialityBCD} follows a similar structure, and involves computing the explicit truncation curves realizing all the algebras $\cC^{\psi}_{iX}(n,m)$ as one-parameter quotients of the universal even-spin algebra $\cW^{\text{ev}}(c,\lambda)$.

Using the truncation curves that realize $\cC^{\psi}(n,m)$ and $\cD^{\psi}(n,m)$ as one-parameter quotients of $\cW(c,\lambda)$, one can classify the nontrivial pointwise isomorphisms between the simple quotients of these algebras, which we denote by $\cC_{\psi}(n,m)$ and $\cD_{\psi}(n,m)$. Aside from degenerate cases at central charge $c=0,-2$, such coincidences correspond to intersection points on the trunctation curves. For example, the isomorphisms between $\cC_{\psi}(n,m)$ and the principal $\cW$-algebras of type $A$ are given explicitly by Corollary 6.5 of \cite{CLIV}. 

Similarly, the pointwise coincidences between the simple quotients $\cC_{\psi, iX}(n,m)$ of the algebras $\cC^{\psi}_{iX}(n,m)$ aside from a few degenerate cases, correspond to intersection points on their truncation curves. For example, the pointwise coincidences between $\cC_{\psi, iX}(n,m)$ and principal $\cW$-algebras of type $C$ are given in Appendix B of \cite{CLV}. We end this section with a few application of these isomorphisms.

\subsection{Rationality results}
For each of the simple VOAs $\cC_{\psi, iX}(n,m)$, there are certain values of $\psi$ where it is isomorphic to $\cW_s(\gs\gp_{2r})$ at a level $r$ which is nondegenerate admissible for $\widehat{\gs\gp}_{2r}$, or $\cW_s(\gs\go_{2r})^{\mathbb{Z}/2\mathbb{Z}}$ at a level $r$ which is nondegenerate admissible for $\widehat{\gs\go}_{2r}$. By Arakawa's results \cite{ArIII,ArIV}, these VOAs are $C_2$-cofinite and rational. This observation allowed many new rationality results to be proven in \cite{CLV}.

First, we have an embedding of universal affine VOAs $V^k(\gs\gp_{2n}) \hookrightarrow V^k(\go\gs\gp_{1|2n})$ for all $n \geq 1$. By Prop. 8.1 and 8.2 of \cite{KWVIII}, this map descends to an embedding of simple affine VOAs 
$$L_k(\gs\gp_{2n})\hookrightarrow L_k(\go\gs\gp_{1|2n}),$$ for all positive integers $k$. By Theorem 8.1 of \cite{CLI}, the coset $$\text{Com}(L_k(\gs\gp_{2n}), L_k(\go\gs\gp_{1|2n}))$$ is the simple quotient of 
 $\text{Com}(V^k(\gs\gp_{2n}), V^k(\go\gs\gp_{1|2n}))$, which coincides with the simple quotient $\cC_{\psi,1C}(n,0)$ of $\cC^{\psi}_{1C}(n,0)$, for $k = - \frac{1}{2} (\psi +2n +1)$. Next, Corollary 4.1 of \cite{CLV} together with Feigin-Frenkel duality tells us that we have isomorphisms
 $$\text{Com}(L_k(\gs\gp_{2n}), L_k(\go\gs\gp_{1|2n})) \cong \cW_{\ell}(\gs\gp_{2n}),\qquad \ell = -(n + 1) + \frac{1 + k + n}{1 + 2 k + 2 n}.$$ Note that the level $\ell$ is nondegenerate admissible for $\widehat{\gs\gp}_{2n}$, so $\cW_{\ell}(\gs\gp_{2n})$ is $C_2$-cofinite and rational. Therefore both $L_k(\go\gs\gp_{1|2n})$ and its even subalgebra $L_k(\go\gs\gp_{1|2n})^{ \mathbb{Z}/2\mathbb{Z}}$ are extensions of $L_k(\gs\gp_{2n}) \otimes \cW_{\ell}(\gs\gp_{2n})$ which is $C_2$-cofinite and rational. This extension must be of finite index, since otherwise at least one of the finitely many irreducible modules of $L_k(\gs\gp_{2n}) \otimes \cW_{\ell}(\gs\gp_{2n})$ must appear with infinite multiplicity. This is impossible since conformal weight spaces of both $L_k(\go\gs\gp_{1|2n})$, and its even subalgebra  $L_k(\go\gs\gp_{1|2n})^{ \mathbb{Z}/2\mathbb{Z}}$, are finite-dimensional. It follows that both these extensions are $C_2$-cofinite. The rationality of $L_k(\go\gs\gp_{1|2n})^{\mathbb{Z}/2\mathbb{Z}}$ follows from Proposition 2.2 of \cite{CLV}. Finally, by Theorem 5.13 of \cite{CGN} we obtain

\begin{thm} (Theorem 7.1, \cite{CLV}) \label{rationalityosp} For all positive integers $k,n$, $L_k(\go\gs\gp_{1|2n})$ is a rational vertex superalgebra.
\end{thm}

The rationality of $L_k(\go\gs\gp_{1|2n})$ was previously known only in the case $n=1$ \cite{CFK}.  It was shown in \cite{GK} that if $\gg$ is a simple Lie superalgebra that is not a Lie algebra, $L_k(\gg)$ is $C_2$-cofinite only if $\gg = \go\gs\gp_{1|2n}$ and $k \in \mathbb{N}$. Therefore Theorem \ref{rationalityosp} completes the classification of $C_2$-cofinite and rational affine vertex superalgebras. This is the analogue of Frenkel and Zhu's famous theorem that for a simple Lie algebra $\gg$, $L_k(\gg)$ is $C_2$-cofinite and rational if and only if $k \in \mathbb{N}$.

Using similar methods, the following results were also proven in \cite{CLV}.

\begin{thm} Let $\cW_{\psi-2m-1}(\gs\go_{2m+3}, f_{\text{subreg}})$ denote the $\cW$-algebra of $\gs\go_{2m+3}$ associated to the subregular nilpotent element. For all positive integers $m, r$, $\cW_{\psi-2m-1}(\gs\go_{2m+3}, f_{\text{subreg}})$ is $C_2$-cofinite and rational at the following levels:
\begin{enumerate}
\item $\psi = \frac{3 + 2 m + 2 r}{2m+2}$, where $m+1$ and $2r+1$ are coprime,
\item  $\psi = \frac{2 m + 2 r+1}{2m+1}$ where $r$ and $2m+1$ are coprime,
\item  $ \psi = \frac{2 m}{2 m-1}$.
\end{enumerate}
\end{thm}
The proof involves exhibiting $\cW_{\psi-2m-1}(\gs\go_{2m+3}, f_{\text{subreg}})$ as an extension of a rank one lattice VOA tensored with the algebra $\cC_{\psi,1D}(1,m)$ at a point where it is coincident with a rational VOA of the form $\cW_s(\gs\go_{2r})^{\mathbb{Z}/2\mathbb{Z}}$ or $\cW_s(\gs\gp_{2r})$. Cases (1) and (2) are exceptional $\cW$-algebras, and are predicted to be rational as part of the Kac-Wakimoto rationality conjecture, which was refined by Arakawa and van Ekeren \cite{ArIII,AvE}. Case (3) gives new examples of rational $\cW$-algebras that are not part of these conjectures. Similarly, we have

\begin{thm} Let $\cW_{r-1/2}(\gs\gp_{2n+2}, f_{\text{min}})$ denote the $\cW$-algebra of $\gs\gp_{2n}$ associated to the minimal nilpotent element. For all positive integers $n, r$, $\cW_{r-1/2}(\gs\gp_{2n+2}, f_{\text{min}})$ is $C_2$-cofinite and rational. 
\end{thm}

Again, these were predicted to be rational by the Kac-Wakimoto conjecture, and the proof involves exhibiting $\cW_{r-1/2}(\gs\gp_{2n+2}, f_{\text{min}})$ as an extension of $L_r(\gs\gp_{2n}) \otimes \cW_s(\gs\gp_{2r})$ for $s = -(r+1) + \frac{1+n+r}{3+2n+2r}$.

\section{Orbifolds and the associated variety} \label{sec:orbifoldandassvar} 
A large portion of this paper has been devoted to methods for finding minimal strong generating sets for orbifold VOAs. We conclude by exploring the relationship between the orbifold functor and the associated variety and associated scheme functors.

We recall an important open conjecture in the subject: given a rational $C_2$-cofinite VOA $\cV$ and a finite group $G$ of automorphisms of $\cV$, it is expected that $\cV^G$ is also rational and $C_2$-cofinite. In the case where $G$ is a cyclic group, the $C_2$-cofiniteness of $\cV^G$ was proven by Miyamoto \cite{MiI}, and the rationality was proven by Miyamoto and Carnahan in \cite{CM}. More recently, it was shown in an important paper by McRae that under some natural hypotheses on $\cV$, namely that it is of CFT type and self-contragredient, the $C_2$-cofiniteness of $\cV^G$ would imply the rationality of $\cV^G$ for any finite automorphism group $G$ \cite{McRI}. So the preservation of $C_2$-cofiniteness by taking finite group orbifolds is a fundamental problem. It is perhaps fruitful to recast this question in a broader context.

\begin{question}
Let $\cV$ be a simple VOA, and let $G$ be a finite group of automorphisms of $V$. Do $X_{\cV}$ and $X_{\cV^G}$ always have the same dimension?
\end{question}
If $\cV$ and $\cV^G$ are both strongly finitely generated, the $C_2$-cofiniteness is equivalent to the associated variety having dimension zero. So preservation of $C_2$-cofiniteness would be a special case of preservation of this dimension.

We can also go further and ask whether the associated scheme functor and orbifold functors commute. In other words, given a simple VOA $\cV$ and a finite automorphism group $G$, the action of $G$ descends to $R_{\cV}$, and we can ask what the relationship between $R_{\cV^G}$ and $(R_{\cV})^G$. The map $\cV^G \hookrightarrow \cV$ induces a homomorphism $R_{\cV^G} \rightarrow \cR_{\cV}$ whose image clearly lies in the invariant subalgebra $(R_{\cV})^G$, so it is natural to ask whether this map can be an isomorphism. In general, the answer is no, as the following example of \cite{AL} illustrates.

Let $\cH$ be the rank one Heisenberg algebra with generator $\alpha$ satisfying $$\alpha(z) \alpha(w)\sim (z-w)^{-2}.$$ Recall that the orbifold $\cH^{ \mathbb{Z}/2\mathbb{Z}}$ is strongly generated by a weight $2$ field $:\alpha \alpha:$ and a weight $4$ field $:(\partial^2\alpha)\alpha:$. It is convenient to rescale these fields: we take
$L = \frac{1}{2} :\alpha\alpha:$, which is the Virasoro field, and 
$W = \frac{35}{132} :(\partial^2 \alpha) \alpha:$ to be our generating set. Let $\ell$ and $w$ be the images of $L$ and $W$ in $R_{\cH^{ \mathbb{Z}/2\mathbb{Z}}}$. It was shown in \cite{AL} that 
$$R_{\cH^{ \mathbb{Z}/2\mathbb{Z}}} \cong \mathbb{C}[\ell,w ] / I$$ where $I$ is the ideal generated by $w(w-\ell^2)$ and $\ell^3 w$. Note that $R_{\cV}$ is not reduced; in fact, the nilradical $\cN\subseteq R_{\cV}$ is generated by $w$, and the reduced ring $R_{\cH^{ \mathbb{Z}/2\mathbb{Z}}} / \cN \cong \mathbb{C}[\ell]$.

On other hand, $R_{\cH} \cong \mathbb{C}[a]$ where $a$ is the image of $\alpha$ in $R_{\cH}$. The action of $\theta \in  \mathbb{Z}/2\mathbb{Z}$ on $R_{\cH}$ sends $a\mapsto -a$, so $(R_{\cH})^{ \mathbb{Z}/2\mathbb{Z}} \cong \mathbb{C}[a^2]$. Therefore in this example, the reduced rings of $(R_{\cH})^{ \mathbb{Z}/2\mathbb{Z}}$ and $R_{\cH^{ \mathbb{Z}/2\mathbb{Z}}}$ are isomorphic.

Here is a slightly more interesting example where the associated scheme functor and orbifold functors do not commute, but they do commute at the level of reduced rings. Recall that the rank $2$ Heisenberg algebra $\cH(2)$ has automorphism group $O(2)$. Inside the subgroup $SO(2) \cong U(1)$ there is a copy of $\mathbb{Z}/3\mathbb{Z}$, and the orbifold $\cH(2)^{\mathbb{Z}/3\mathbb{Z}}$ is known to be of type $\cW(2,3^3,4, 5^3)$ \cite{MPS}. It it convenient to change basis for $\cH(2)$ and choose generating fields $\beta^1, \beta^2$ satisfying
$$\beta^1(z) \beta^2(w) \sim (z-w)^{-2},\qquad \beta^i(z) \beta^i(w) \sim 0,\ \text{for} \ i = 1,2.$$
The generator of $\mathbb{Z}/3\mathbb{Z}$ acts by $e^{\frac{2\pi i}{3}}$ on $\beta^1$ and by $e^{-\frac{2\pi i}{3}}$ on $\beta^2$. In this notation, $\cH(2)^{\mathbb{Z}/3\mathbb{Z}}$ is an extension of $\cH(2)^{SO(2)}$, and is generated as an $\cH(2)^{SO(2)}$-module by the weight $3$ fields
$$C^3 =\  :(\beta^1)^3:,\qquad D^3 = \ :(\beta^2)^3:.$$
Moreover, $\cH(2)^{SO(2)}$ is easily seen to be of type $\cW(2,3,4,5)$ with generators $L, W^3, W^4, W^5$. The Virasoro field is
$$L = \ :\beta^1 \beta^2,$$ which has central charge $c=2$, and we define
$$W^3 = \frac{1}{2}( :\beta^1 \partial \beta^2: - :(\partial \beta^1) \beta^2:.$$ which is primary of weight $3$. Then we take
\begin{equation*} \begin{split} W^4 & = (W^3)_{(1)} W^3 = \frac{1}{4} \big(  5 : \beta^1 ( \partial^2 \beta^2):  - 6 :(\partial \beta^1)( \partial \beta^2):  + 5 :(\partial^2 \beta^1) \beta^2:\big),
\\ W^5 & = (W^3)_{(1)} W^4 = \frac{5}{8} \big(  7 : \beta^1 ( \partial^3 \beta^2):  - 9 :(\partial \beta^1)( \partial^2 \beta^2): +9 :(\partial^2 \beta^1)( \partial \beta^2):  -7 :(\partial^3 \beta^1) \beta^2:\big).\end{split} \end{equation*}

We need two more fields of weight $5$ to get a strong generating set for $\cH(2)^{\mathbb{Z}/3\mathbb{Z}}$, namely,
$$C^5 =\  :(\partial^2 \beta^1)(\beta^1)^2:,\qquad D^5 = \  :(\partial^2 \beta^2)(\beta^2)^2:,$$ which is easily seen to be a minimal strong generating set. The corresponding elements of $R_{\cH(2)^{\mathbb{Z}/3\mathbb{Z}}}$, which we also denote by $L,W^3, W^4, W^5, C^3, D^3, C^5, D^5$,  are then a minimal generating set. Therefore 
$$R_{\cH(2)^{\mathbb{Z}/3\mathbb{Z}}}\cong \mathbb{C}[L,W^3, W^4, W^5, C^3, D^3, C^5, D^5] / I,$$ for some ideal $I$. By finding normally ordered relations among these generators and their derivatives, and then taking the image of these relations in $R_{\cH(2)^{\mathbb{Z}/3\mathbb{Z}}}$, we obtain the following elements of $I$:
\begin{equation} \begin{split} \label{nilradcalc}
& (W^5)^2 = 0,
\\ &  (W^4)^3 = 0,
\\ & (W^3)^3 - \frac{8}{27} L W^3 W^4  + \frac{1}{108} L^2 W^5  - 
 \frac{1}{1296} W^4 W^5 = 0,
 \\ & (C^5)^2 = (D^5)^2 = 0,
\\ & C^3 D^3 + 117 (W^3)^2 - 33 L W^4  - L^3  = 0.
 \end{split}
 \end{equation}
 More details about these calculations can be found in the thesis of Dan Graybill \cite{Gr}. It is immediate from the first four relations in \eqref{nilradcalc} that $W^3$, $W^4$, $W^5$, $C^5$, and $D^5$ lie in the nilradical $\cN \subseteq R_{\cH(2)^{\mathbb{Z}/3\mathbb{Z}}}$, so the reduced ring $$R_{\cH(2)^{\mathbb{Z}/3\mathbb{Z}}}/ \cN$$ is generated by $L$, $C^3$, and $D^3$. Also, it follows from the fifth relation in \eqref{nilradcalc} that $$C^3 D^3 - L^3 = 0,\ \text{in}\  R_{\cH(2)^{\mathbb{Z}/3\mathbb{Z}}}/ \cN.$$ Finally, it is not difficult to check that $$R_{\cH(2)^{\mathbb{Z}/3\mathbb{Z}}}/ \cN \cong \mathbb{C}[L,C^3, D^3] / (C^3 D^3 - L^3).$$
This is clearly isomorphic to $(R_{\cH(2)})^{\mathbb{Z}/3\mathbb{Z}}$ since $R_{\cH(2)} \cong \mathbb{C}[\beta^1, \beta^2]$ and the generators of $(R_{\cH(2)})^{\mathbb{Z}/3\mathbb{Z}}$ are clearly $\beta^1\beta^2$, $(\beta^1)^3$, and $(\beta^2)^3$. Note that $(R_{\cH(2)})^{\mathbb{Z}/3\mathbb{Z}}$ is already a reduced ring, so the reduced rings of $(R_{\cH(2)})^{\mathbb{Z}/3\mathbb{Z}}$ and $R_{\cH(2)^{\mathbb{Z}/3\mathbb{Z}}}$ are isomorphic.

One can check the isomorphism between the reduced rings of $R_{\cV^G}$ and $(R_{\cV})^G$ in many other similar but richer examples. We conclude with a rather speculative conjecture: 
\begin{conj} If $\cV$ is a simple, strongly finitely generated VOA which is $\mathbb{N}$-graded by conformal weight, $\cV[0] \cong \mathbb{C}$, and $G$ is a finite group of automorphisms of $\cV$, then the reduced rings of $R_{\cV^G}$ and $(R_{\cV})^G$ are isomorphic.
\end{conj}

Note that $(R_{\cV})^G$ and $R_{\cV}$ have the same Krull dimension, so this statement would imply that the dimensions of $X_{\cV}$ and $X_{\cV^G}$ are the same. In particular, if $\cV$ is $C_2$-cofinite and satisfies the vertex algebra Hilbert theorem (that is, $\cV^G$ is strongly finitely generated), this would imply that $\cV^G$ is also $C_2$-cofinite.

\end{document}